\documentclass[10pt]{article}

\usepackage{amsfonts} 
\usepackage{amsmath,hhline,latexsym}
\usepackage{amssymb}
\usepackage{amsthm}
\usepackage[latin1]{inputenc}
\usepackage[dvips,final]{graphics}
\DeclareGraphicsExtensions{.tiff}
\usepackage{hyperref}

\usepackage{color}
\usepackage{graphicx}

\newtheorem{theorem}{\sc Theorem}[section]
\newtheorem{lemma}{\sc Lemma}[section]
\newtheorem{proposition}{\sc Proposition}[section]

\newtheorem{remark}{Remark}

\newcommand{\dis}{\displaystyle}
\newcommand{\eps}{\varepsilon}
\newcommand{\om}{\omega}
\newcommand{\Om}{\Omega}
\newcommand{\ph}{\varphi}
\newcommand{\phbar}{\overline{\varphi}}
\newcommand{\lambar}{\overline{\lambda}}

\newcommand{\psibar}{\overline{\psi}}

\newcommand{\N}{\mbox{$I \kern -4pt N$}}
\newcommand{\Q}{\mbox{$Q \kern -8pt I$}}
\newcommand{\R}{\mbox{$I \kern -4pt R$}}
\newcommand{\C}{\mbox{$C \kern -8pt I$}}

\newcommand{\jnt}{\dis\int}
\newcommand{\jjntQT}{\jnt\!\!\!\!\jnt_{Q_{T}}}
\newcommand{\jjntqT}{\jnt\!\!\!\!\jnt_{q_{T}}}

\providecommand{\tabularnewline}{\\}

\def\R{{\bf R}}
\title{\textbf{A mixed formulation for the direct approximation of $L^2$-weighted controls for the linear heat equation}}

\author{ 
	\textsc{Arnaud M\"unch}\thanks{Laboratoire de Math\'ematiques, Universit\'e Blaise Pascal (Clermont-Ferrand 2), UMR CNRS 6620, 
	Campus de C\'ezeaux, 63177, Aubi\`ere, France. E-mails: {\tt arnaud.munch@math.univ-bpclermont.fr.}}\quad 
	\and
	\textsc{Diego A. Souza}\thanks{Dpto.\ EDAN, University of Sevilla, 41080~Sevilla, Spain.
	E-mail: {\tt desouza@us.es}. Partially supported by grant~MTM2010-15592 (DGI-MICINN, Spain).}}
	
\begin{document}
\maketitle


\begin{abstract}
	This paper deals with the numerical computation of null controls for the linear heat equation.
	The goal is to compute approximations of controls that drive the solution from a prescribed initial state to 
	zero at a given positive time. In [Fernandez-Cara \& M\"unch, Strong convergence approximations of null controls for the 1D heat equation, 2013], a so-called primal method is described leading to a strongly convergent approximation of distributed control: the controls minimize quadratic weighted functionals involving both the control and the state and are obtained by solving the corresponding optimality conditions. In this work, we adapt the method to approximate the control of minimal square integrable-weighted norm. The optimality conditions of the problem are reformulated as a mixed formulation involving both the state and its adjoint. We prove the well-posedeness of the mixed formulation (in particular the inf-sup condition) then discuss 
	several numerical experiments. The approach covers both the boundary and the inner situation and is valid in any dimension.
\end{abstract}

\vskip 0.25cm

\noindent
\textbf{Mathematics Subject Classification (2010)-}  35K35, 65M12, 93B40, 65K10

\tableofcontents


\section{Introduction. The null controllability problem}
\label{sec:intro}

	Let $\Om\subset \mathbb{R}^N$ be a bounded connected open set whose boundary $\partial\Om$ is regular enough 
	(for instance of class $C^2$). Let $\om\subset \Om$ be a (small) nonempty open subset and assume that $T>0$.  
	In the sequel, for any $\tau > 0$ 
	we denote by $Q_\tau$, $q_{\tau}$ and $\Sigma_{\tau}$ the sets $\Om\times(0, \tau)$, $\omega\times(0,\tau)$ and 
	$\partial\Om\times (0,\tau)$, respectively.

This work is concerned with the null controllability problem for the heat equation
\begin{equation}
\label{eq:heat}
	\left\{
		\begin{array}{lll}
   			y_{t} - \nabla\cdot(c(x) \nabla y) + d(x, t) y= v\,1_{\omega}, 		& \textrm{in}& Q_T,    \\
   			y  = 0,										  		&\textrm{on}& \Sigma_T, \\
   			y(x, 0) = y_0(x), 								  		& \textrm{in}& \Om.
   		\end{array} 
 	\right.
\end{equation}
	Here, we assume that $c:=(c_{i,j})\in C^1(\overline\Om;\mathcal{M}_{N}(\mathbb{R}))$ with $(c(x)\xi,\xi)\geq c_0|\xi|^2$ 
	in~$\overline\Om~(c_0>0)$, $d \in L^\infty(Q_T)$ and $y_0 \in L^2(\Om)$; $v = v(x,t)$ is the \textit{control} (a function 
	in~$L^2(q_T)$) and $y=y(x,t)$ is the associated state. Moreover, $1_{\omega}$ is the characteristic function associated to the set $\omega$.

In the sequel, we shall use the following notation :
\begin{equation}
	L\, y:=y_{t} - \nabla\cdot(c(x) \nabla y) + d(x, t) y, \qquad L^{\star}\ph:=-\ph_{t}- \nabla\cdot(c(x) \nabla \ph)+ d(x,t)\ph. \nonumber
\end{equation}
 
For any $y_0\in L^2(\Omega)$ and $v\in L^2(q_T)$, there exists exactly one solution $y$ to 
	\eqref{eq:heat}, with the regularity $y \in C^0([0, T]; L^2(\Om)) \cap L^2(0, T; H^{1}_0(\Om))$ (see \cite{LionsMagenes,HarauxCazenave}). Accordingly, for any final time 
	$T>0$, the associated null controllability problem at time $T$ is the following\,: for each $y_0\in L^2(\Omega)$, 
	find $v\in L^2(q_T)$ such that the corresponding solution to \eqref{eq:heat} satisfies
\begin{equation}\label{eq:nullT}
	y(\cdot\,,T) = 0 \quad \text{in }\ \Om.
\end{equation}

The controllability of PDEs is an important area of research and has been the subject of many papers in recent years.
Some relevant references are~\cite{LT1} and~\cite{russell78}. In particular, we refer to~\cite{FursikovImanuvilov} and \cite{LebRob} where the null controllability of (\ref{eq:heat}) is proved. 
  
  The numerical approximation is also a fundamental issue, since it is not in general possible to get explicit expression of controls. 
  Due to the strong regularization property of the heat kernel, numerical approximation of controls is a rather delicate issue. The same holds in inverse problems theory 
  when parabolic equations and systems are involved (see \cite{engl}). This has been exhibited numerically in~\cite{carthel} who made use of a duality argument and focused on the control of minimal square integrable norm: the problem reads   
\begin{equation}
\label{P-FI-L2}
\left\{
\begin{array}{l}
\dis \hbox{Minimize }\ J_1(y,v) := {1 \over 2} \jjntqT  |v(x,t)|^2 \,dx\,dt \\
\noalign{\smallskip}
\hbox{Subject to }\ (y,v) \in \mathcal{C}(y_0,T)
\end{array}
\right.
   \end{equation}
where $\mathcal{C}(y_0;T)$ denotes the linear manifold
   $$
\mathcal{C}(y_0;T) := \{\, (y,v) : v \in L^2(q_T),\ \hbox{$y$ solves (\ref{eq:heat})\, and\, satisfies (\ref{eq:nullT})} \,\}.
   $$
The earlier contribution is due to Glowinski and Lions in~\cite{GL96} (updated in \cite{Glo08}) and relies on duality arguments. Duality allows to replace the original constrained minimization problem by an unconstrained and \textit{a priori} easier minimization (dual) problem. The dual problem associated with (\ref{P-FI-L2}) is :
\begin{equation}\label{eq:J}
\min_{\ph_T\in \mathcal{H}} \; J_1^{\star}(\ph_T) := \frac{1}{2} \jjntqT  |\ph(x,t)|^2\, dx dt + \int_{\Omega} y_0(x) \ph(x, 0) dx 
\end{equation}
where the variable $\ph$ solves the backward heat equation : 
\begin{equation}\label{eq:wr}
L^{\star}\ph=0 \quad \textrm{in}\quad Q_T, \qquad \ph=0\quad \textrm{on}\quad \Sigma_T; \quad \ph(\cdot, T)=\ph_T \quad \textrm{in}\quad \Omega,
\end{equation}
and the Hilbert space $\mathcal{H}$ is defined as the completion of $\mathcal{D}(\Omega)$ with respect to the norm $\|\ph_T\|_{\mathcal{H}}:=\|\ph\|_{L^2(q_T)}$. In view of the unique continuation property  to \eqref{eq:wr}, the mapping $\ph_T\mapsto \|\ph_T\|_{\mathcal{H}}$ is a Hilbertian norm 
	in $\mathcal{D}(\Om)$. Hence, we can certainly consider the completion of $\mathcal{D}(\Om)$ for this norm. The coercivity of the functional 
	$J_1^{\star}$ in $\mathcal{H}$ is a consequence of the so-called \textit{observability inequality}
\begin{equation}
\label{eq:obs-ineq}
	\dis \|\ph(\cdot,0)\|^2_{L^2(\Om)}\leq C\jjntqT|\ph(x,t)|^2\,dx\,dt\quad \forall \ph_T\in \mathcal{H},
\end{equation}
where $\ph$	solves (\ref{eq:wr}). This inequality holds  for some constant $C = C(\om, T)$ and, in turn,  is a consequence of some appropriate global Carleman inequalities; 
	see \cite{FursikovImanuvilov}. The minimization of $J_1^{\star}$ is numerically ill-posed, essentially because of the hugeness 
	of the completed space $\mathcal{H}$. The control of minimal square integrable norm highly oscillates near the final time $T$, property which is hard to capture numerically. We refer to \cite{belgacem,kindermann,micu-zuazua,MunchZuazuaRemedies} where this phenomenon is highlighted under several perspectives.
	
Moreover, at the level of the approximation, the minimization of $J_1^{\star}$ requires to find a finite dimensional and conformal approximation of $\mathcal{H}$ such that the corresponding discrete adjoint solution satisfies (\ref{eq:wr}), which is in general impossible for polynomial piecewise approximations. In practice, the trick initially described in~\cite{GL96}, consists first to introduce a discrete and consistent approximation of (\ref{eq:heat}) and then to minimize the corresponding discrete conjugate functional. However, this requires to get some uniform discrete observability inequalities which is a delicate issue, strongly depend on the approximations used (we refer to \cite{boyer,ervedoza_valein,ZuazuaIcm} and the references therein) and is still open in the general case of the heat equation with non constant coefficients. This fact and the hugeness of $\mathcal{H}$ has raised many authors to relax the controllability problem: precisely, the constraint (\ref{eq:nullT}). We mention the references \cite{boyer,carthel,ZuazuaIcm} and notably \cite{boyercanum12,EFC-AM-dual,labbe_trelat} for some numerical realizations. 

In \cite{EFC-AM-sema} (see also \cite{EFC-AM-mcrf} in a semi-linear case), a different - so-called primal approach - allowing more general results has been used and consists to solve directly optimality conditions : specifically, the following general extremal problem (initially introduced by Fursikov and~Ima\-nu\-vi\-lov in~\cite{FursikovImanuvilov}) is considered : 
   \begin{equation}
\label{P-FI}
\left\{
\begin{array}{l}
\dis \hbox{Minimize }\ J(y,v) := {1 \over 2} \jjntQT \rho^2 |y|^2 \,dx\,dt + {1 \over 2} \jjntqT \rho_0^2 |v|^2 \,dx\,dt \\
\noalign{\smallskip}
\hbox{Subject to }\ (y,v) \in \mathcal{C}(y_0,T).
\end{array}
\right.
   \end{equation}
  The weights $\rho=\rho(x,t)$ and $\rho_0=\rho_0(x,t)$ are continuous, uniformly positive and are assumed to belong to $L^\infty(Q_{T-\delta})$ for any $\delta > 0$ (hence, they can blow up as $t \to T^-$). Under those conditions, the extremal problem (\ref{P-FI}) is well-posed (see \cite{EFC-AM-sema}). 
  
  Moreover, the explicit occurrence of the term $y$ in the functional allow to solve directly the optimality conditions associated with (\ref{P-FI}): defining the Hilbert space $P$ as the completion of the linear space $P_0=\{q\in C^{\infty}(\overline{Q_T}): q=0\,\,\textrm{on}\,\, \Sigma_T\}$ with respect to the scalar product 
 \begin{equation}
\left( p, q \right)_P:= \jjntQT \rho^{-2} L^{\star}p \, L^{\star}q \,dx\,dt + \jjntqT \rho_0^{-2} \, p \, q \,dx\,dt, 
\label{IP}
\end{equation}
the optimal pair $(y,v)$ for $J$ is characterized as follows
\begin{equation}\label{eq:yvp}
y = \rho^{-2} L^{\star}p \quad \textrm{in}\quad Q_T, \qquad v = - \rho_0^{-2}p\, 1_{\omega} \quad\textrm{in}\quad Q_T
\end{equation}
in term of an additional variable $p\in P$ unique solution to the following variational equality :
\begin{equation}\label{eq:varp}
(p,q)_P= \int_{\Omega} y_0(x)\,  q(x,0) \,dx, \quad \forall q \in P.
\end{equation}
The well-posedeness of this formulation is ensured as soon as the weights $\rho_0,\rho$ are of Carleman type (in particular $\rho$ and $\rho_0$ blow up exponentially as $t\to T^{-}$); this specific behavior near $T$ reinforces the null controllability requirement and prevents the control of any oscillations near the final time. 

The search of a control $v$ in the manifold 
$\mathcal{C}(y_0,T)$ is reduced to solve the (elliptic) variational formulation (\ref{eq:varp}). In \cite{EFC-AM-sema}, the approximation of (\ref{eq:varp}) is performed 
in the framework of the finite element theory through a discretization of the space-time domain $Q_T$. In practice, an approximation $p_h$ of $p$ is obtained in a direct way by inverting a symmetric positive definite matrix, in contrast with the iterative (and possibly divergent) methods used within dual methods. Moreover, a major advantage of this approach is that a conformal approximation, say $P_h$ of $P$, leads to the strong convergence of $p_h$ toward $p$ in $P$, and consequently from (\ref{eq:yvp}), to a strong convergence in $L^2(q_T)$ of 
$v_h:=- \rho_0^{-2}p_h 1_{\omega}$ toward $v$, a null control for (\ref{eq:heat}). It is worth to mention that, for any $h>0$, $v_h$ is not \textit{a priori} an exact control for any finite dimensional system (which is not necessary at all in practice) but an approximation for the $L^2$-norm of the control $v$. 

The variational formulation (\ref{eq:varp}) derived from the optimality conditions (\ref{eq:yvp}) is obtained assuming that the weights $\rho$ and $\rho_0$ are both strictly positive in $Q_T$ and $q_T$ respectively. In particular, this approach does not apply for the control of minimal $L^2$-norm, for which simply $\rho:=0$ and $\rho_0:=1$. The main reason of the present work is to adapt this approach to cover the case $\rho:=0$ and therefore obtain directly an approximation $v_h$ of the control of some minimal weighted $L^2$-norm. To do so, we adapt the idea developed in \cite{NC-AM-mixedwave} devoted to the wave equation.  
We also mention \cite{munch_EJAM_2014} where a different space-time variational approach (based on Least-squares principles) is used to approximate null controls for the heat equation.

The paper is organized as follows. In Section \ref{mixed_reformulation}, we associate to the dual problem (\ref{eq:J}) an equivalent mixed formulation
which relies on the optimality conditions associated to the problem (\ref{P-FI}) with $\rho=0$. In Section \ref{penalized_one}, we first address the penalization case and write the constraint $L^{\star}\ph=0$ as an equality in $L^2(Q_T)$. We then show the well-posedness of this mixed formulation, in particular we check the inf-sup condition (Theorem \ref{th:mf1}). The mixed formulation allows to approximate simultaneously the dual variable and the primal one, controlled solution of (\ref{eq:heat}). Interestingly, we also derive an equivalent extremal problem in the primal variable $y$ only (see Prop \ref{prop_equiv_dual}, Section \ref{duasectionmf1}). In Section \ref{second_mixed_form}, we reproduce the analysis relaxing the condition $L^{\star}\ph=0$ in the weaker space $L^2(0,T, H^{-1}(\Omega))$. Then, in Section \ref{third_mixed_form}, by using the Global Carleman estimate (\ref{crucial_estimate}), we show that a well-posed mixed formulation is also available for the limit and singular case for which $\eps=0$ leading to Theorem \ref{th:mf3}. Section \ref{numerics} is devoted to the numerical approximation of the mixed formulation (\ref{eq:mf1}) in the case $\eps>0$ (Section \ref{discretization_mf1}) and of the mixed formulation (\ref{eq:mf3}) in the case $\eps=0$ (section \ref{discretization_mf3}). Conformal approximations based on space-time finite elements are employed. In Section \ref{infsuptest}, we numerically check that the approximations used lead to discrete inf-sup properties, uniformly w.r.t. the discretization parameter $h$. Then the remaining of Section \ref{numerics} is devoted to some experiments which emphasize the remarkable robustness of the method. Section \ref{concluding_remarks} concludes with some perspectives.


\section{Control of minimal weighted $L^2$-norm~: mixed reformulations}
\label{mixed_reformulation}

In order to avoid the minimization of the conjugate functional $J^{\star}$ with respect to the final state $\ph_T$ by an iterative process, we now present a direct way 
to approximate the control of minimal square integrable norm in the spirit of the primal approach recalled in the introduction and developed in \cite{EFC-AM-sema}. We adapt the case of the wave equation studied in \cite{NC-AM-mixedwave}.

\subsection{The penalized case: Mixed formulation I} \label{penalized_one}

Let $\rho_{\star}\in \mathbb{R}^+_{\star}$ and let $\rho_0 \in \mathcal{R}$ defined by 
\begin{equation}
\mathcal{R}:=\{w: w\in C(Q_T); w\geq \rho_{\star}>0 \,\, \textrm{in}\,\, Q_T; w\in L^{\infty}(Q_{T-\delta}) \,\, \forall \delta>0\}   \label{def_spaceR}
\end{equation}
so that in particular, the weight $\rho_0$ may blow up as $t\to T^{-}$. We first consider the approximate controllability case. For any $\varepsilon>0$, the problem reads as follows: 
\begin{equation}
\label{P-FI-L2-eps}
\left\{
\begin{array}{l}
\dis \hbox{Minimize }\ J_{\varepsilon}(y,v) := {1 \over 2} \jjntqT \rho_0^2 |v|^2 \,dt  + \frac{1}{2\varepsilon} \Vert y(\cdot,T)\Vert^2_{L^2(\Omega)}\\
\noalign{\smallskip}
\hbox{Subject to }\ (y,v) \in \mathcal{A}(y_0;T)
\end{array}
\right.
   \end{equation}
	where $\varepsilon$ denotes a penalty parameter (see \cite{boyercanum12,carthel,EFC-AM-dual}) and 
	where $\mathcal{A}(y_0;T)$ denotes the linear manifold 
	$\mathcal{A}(y_0;T):= \{\,(y,v) : v \in L^2(q_T),\, \hbox{$y$ solves \eqref{eq:heat}}\,\}$. The corresponding conjugate and well-posed problem is given by 
 \begin{equation}
\label{eq:min_eps_adj}
	\left\{
		\begin{array}{l}
			\dis \hbox{Minimize }\ J^{\star}_\eps(\ph_T):={1\over2} \jjntqT \rho^{-2}_0|\ph(x, t)|^2 dx\,dt 
			+ {\eps\over2}\|\ph_T\|^2_{L^2(\Om)}+(y_0,\ph(\cdot,0))_{L^2(\Omega)}\\
			\noalign{\smallskip}
			\hbox{Subject to }\ \ph_T \in L^2(\Om)
		\end{array}
	\right.
\end{equation}
	where $\ph$ solves \eqref{eq:wr}. 
	
We recall that the penalized problem (\ref{P-FI-L2-eps}) is a consistent approximation of the original null controllability problem, in the sense that its unique solution converges to the solution of \eqref{P-FI} with $\rho=0$ as $\eps\to 0$. We refer for instance to \cite{EFC-AM-dual}, Prop. 3.3 for a proof of the following result, consequence of the null controllability for the heat equation.
\begin{proposition} 
\label{th:limit}  Let $(y_\eps,v_\eps)$ be the solution of Problem \eqref{P-FI-L2-eps} and let $(y,v)$ be the solution of Problem \eqref{P-FI} with $\rho=0$. Then, 
	 one has
	 \[ 
y_\eps\to y~~\hbox{strongly in}~~L^2(Q_T), \quad v_\eps\to v~~\hbox{strongly in}~~L^2(q_T)
\]
	as $\eps\to0^+$.
\end{proposition}
	
\subsubsection{Mixed formulation}

Since the variable $\ph$, solution of (\ref{eq:wr}), is completely and uniquely determined by the data $\ph_T$, the main idea of the reformulation is to keep $\ph$ as main variable.

We introduce the linear space $\Phi^0:=\{\ph\in C^2(\overline{Q_T}),\, \ph=0\, \textrm{ on }\, \Sigma_T\}$. For any $\eta>0$, we define the bilinear form 
\[
	(\ph, \phbar)_{\Phi^0} := \jjntqT \!\!\!\rho_0^{-2}\ph \,\phbar\, dx\,dt + \eps(\ph(\cdot,T),\phbar (\cdot,T))_{L^2(\Om)} + \eta  \jjntQT \!\!\!L^{\star}\ph\, L^{\star}\phbar\, dx\,dt , \,~\forall \ph, \phbar \in \Phi^0.
\]	
From the unique continuation property for the heat equation, this bilinear form defines for any $\eps\geq 0$ a scalar product.  For any $\eps>0$, let $\Phi_\eps$ be the completion of $\Phi^0$ for this scalar product.  We denote the norm over $\Phi_\eps$ by $\Vert\cdot\Vert_{\Phi_{\eps}}$ such that 
\begin{equation} 
\label{eq:norm_PHIeps}
	\|\ph\|_{\Phi_\eps}^2 :=\|\rho^{-1}_0\ph\|^2_{L^2(q_T)} + \eps\|\ph(\cdot,T)\|^2_{L^2(\Om)}+ \eta \|L^{\star}\ph\|^2_{L^2(Q_T)}, \quad \forall \ph \in \Phi_\eps.
\end{equation}
%
%
%
%
Finally, we defined the closed subset $W_{\eps}$ of $\Phi_{\eps}$ by 
\[
W_{\eps}=\{\ph\in \Phi_{\eps}: L^{\star}\ph=0 \,\, \textrm{in}\,\, L^2(Q_T)\}
\]
and we endow $W_{\eps}$ with the same norm than $\Phi_{\eps}$.

Then, we define the following extremal problem : 
\begin{equation}
\label{eq:min_eps_refor1}
	\min_{\ph\in W_\eps} \hat{J}^{\star}_\eps(\ph):={1\over2} \jjntqT \rho_0^{-2}|\ph(x, t)|^2 dx\,dt 
	+ {\eps\over2}\|\ph(\cdot,T)\|^2_{L^2(\Om)}+(y_0,\ph(\cdot,0))_{L^2(\Omega)}.
\end{equation}
Standard energy estimates for the heat equation imply that, for any $\ph\in W_{\eps}$, $\ph(\cdot,0)\in L^2(\Omega)$ so that the functional $\hat{J}^{\star}_{\eps}$ is well-defined over $W_{\eps}$. Moreover, since for any $\ph\in W_{\varepsilon}$, $\ph(\cdot,T)$ belongs to $L^2(\Omega)$, Problem \eqref{eq:min_eps_refor1} is equivalent to the minimization problem \eqref{eq:min_eps_adj}. 
As announced, the main variable is now $\ph$ submitted to the constraint equality (in $L^2(Q_T)$) $L^{\star}\ph=0$. This constraint equality is addressed by introducing a Lagrangian multiplier. 	
	
We consider the following mixed formulation~: find $(\ph_\eps,\lambda_\eps)\in \Phi_\eps\times L^2(Q_T)$ solution of 
\begin{equation} 
\label{eq:mf1}
	\left\{
		\begin{array}{rcll}
			\noalign{\smallskip} a_\eps(\ph_\eps, \phbar) + b(\phbar, \lambda_\eps) 
			& = & l(\phbar), & \quad \forall \phbar \in \Phi_\eps \\
			\noalign{\smallskip} b(\ph_\eps, \lambar) & = & 0, & \quad \forall \lambar \in L^2(Q_T),
		\end{array}
	\right.
\end{equation}
	where
\begin{align}
\nonumber
	a_\eps : \Phi_\eps \times \Phi_\eps \to \mathbb{R}, & \quad a_\eps(\ph, \phbar) 
	:= \jjntqT \rho_0^{-2}\ph \,\phbar\, dx\,dt + \eps(\ph(\cdot,T),\phbar (\cdot,T))_{L^2(\Omega)} \\
	\nonumber
	b: \Phi_\eps \times L^{2}(Q_T)  \to \mathbb{R}, & \quad b(\ph, \lambda) := -\jjntQT L^{\star} \ph\, \lambda \,dx\, dt \\
	\nonumber
	l: \Phi_\eps \to \mathbb{R}, & \quad l(\ph) :=- (y_0,\ph(\cdot, 0))_{L^2(\Omega)}.
\end{align}
We have the following result : 
\begin{theorem}
\label{th:mf1}
\begin{enumerate}
	\item The mixed formulation \eqref{eq:mf1} is well-posed.
	\item The unique solution $(\ph_\eps, \lambda_\eps)\in \Phi_\eps\times L^2(Q_T)$ is the unique saddle-point of the Lagrangian 
	         ${\mathcal{L_\eps} : \Phi_\eps \times L^2(Q_T) \to \mathbb{R}}$ defined by
		\begin{equation}
		\label{eq:calL1}
			\mathcal{L_\eps}(\ph, \lambda) := {1\over2}a_\eps(\ph, \ph) + b(\ph, \lambda)  - l(\ph).
		\end{equation}
       
       \item The optimal function $\ph_\eps$ is the minimizer of $\hat{J}^{\star}_\eps$ over $W_\eps$ while the optimal multiplier $\lambda_\eps\in L^2(Q_T)$ 
	is the state of the heat equation \eqref{eq:heat} in the weak sense. 
\end{enumerate}
\end{theorem}
\begin{proof} 
	We easily check that the bilinear form $a_\eps$ is continuous over $\Phi_\eps\times \Phi_\eps$, symmetric and positive and that 
	the bilinear form $b_\eps$ is continuous over $\Phi_\eps\times L^2(Q_T)$. Furthermore, for any fixed $\eps$, the continuity of the linear form $l$ 
	over $\Phi_\eps$ can be viewed from the energy estimate : 
\[
		\Vert \ph(\cdot,0)\Vert^2_{L^2(\Om)} \leq C\left(\jjntQT |L^{\star}\ph|^2 dx\,dt + \Vert \ph(\cdot,T)\Vert^2_{L^2(\Om)}\right), 
\quad \forall \ph\in \Phi_\eps,
\]
for some $C>0$, so that $\Vert \ph(\cdot,0)\Vert^2_{L^2(\Omega)} \leq  C \max( \eta^{-1},\eps^{-1}) \Vert \ph\Vert^2_{\Phi_\eps}$. 

	Therefore, the well-posedness of the mixed formulation is a consequence of the following two properties (see \cite{BrezziFortin}):
\begin{itemize}
	 \item
	 	$a_\eps$ is coercive on $\mathcal{N}(b)$, where $\mathcal{N}(b)$ denotes the kernel of $b$~:
		\[
			\mathcal{N}(b) := \{ \ph \in \Phi_\eps~:~b(\ph, \lambda) = 0 \textrm{ for every } \lambda \in L^2(Q_T)\};
		\]
	\item  
		$b$ satisfies the usual ``inf-sup" condition over $\Phi_\eps\times L^2(Q_T)$: there exists $\delta > 0$ such that 
		\begin{equation}
		\label{eq:infsup1}
			\inf_{\lambda \in L^2(Q_T)} \sup_{\ph\in \Phi_\eps} \frac{b(\ph, \lambda)}{\|\ph\|_{\Phi_\eps} \|\lambda\|_{L^2(Q_T)}} \geq \delta.
		\end{equation}
\end{itemize}
	From the definition of $a_\eps$, the first point is clear : for all $\ph\in \mathcal{N}(b)=W_\eps$, 
	$a_\eps(\ph,\ph)=\Vert \ph\Vert^2_{\Phi_\eps}$. Let us check the inf-sup condition. For any fixed $\lambda^0\in L^2(Q_T)$, we 
	define the (unique) element $\ph^0$ of 
	\[
	L^{\star}\ph^0=-\lambda^0 \quad \textrm{in}\quad Q_T, \qquad \ph^0=0\quad \textrm{on}\quad \Sigma_T; \quad \ph^0(\cdot, T)=0 \quad \textrm{in}\quad \Omega,
\]
so that $\ph^0$ solves the backward heat equation with source term $-\lambda^0\in L^2(Q_T)$, null Dirichlet boundary
	condition and zero initial state. Since $-\lambda^0\in L^2(Q_T)$, then using energy estimates, there exists 
	a constant $C_{\Omega,T} > 0$ such that the solution $\ph^0$ of the backward heat equation with source term $\lambda^0$ satisfies the 
	inequality
\[
	 \jjntqT \rho_0^{-2}|\ph^0|^2 dx\,dt \leq \rho_{\star}^{-2}\jjntqT |\ph^0|^2  dx\,dt \leq  \rho_{\star}^{-2}\, C_{\Omega,T}\, \|\lambda^0\|_{L^2(Q_T)}^2. 
\]
	Consequently, $\ph^0\in \Phi_\eps$. In particular, we have $b(\ph^0,\lambda^0)= \Vert \lambda^0\Vert^2_{L^2(Q_T)}$ and
\[
	\sup_{\ph \in \Phi_\eps} \!\!\frac{b(\ph, \lambda^0)}{\|\ph\|_{\Phi_\eps}\! \|\lambda^0\|_{L^2(Q_T)}} 
	\!\geq \!\frac{b(\ph^0, \lambda^0)}{\|\ph^0\|_{\Phi_\eps} \!\|\lambda^0\|_{L^2(Q_T)}} 
	\!= \!\frac{\|\lambda^0\|_{L^2(Q_T)}^2}
	{( \|\rho^{-1}_0\ph^0\|^2_{L^2(q_T)}\!+\!\eta \|\lambda_0\|^2_{L^2(Q_T)})^\frac{1}{2}\|\lambda_0\|_{L^2(Q_T)}}.
\]

	Combining the above two inequalities, we obtain
\[
	\sup_{\ph_0 \in \Phi_\eps}\frac{b(\ph_0, \lambda_0)}{\|\ph_0\|_{\Phi_\eps}
	\|\lambda_0\|_{L^2(Q_T)}}\geq\frac{1}{\sqrt{\rho^2_{\star}\,C_{\Omega,T}+\eta}}
\]
	and, hence, \eqref{eq:infsup1} holds with $\delta = \left(\rho^2_{\star}\, C_{\Omega,T}+ \eta \right)^{-1/2}$.

	The point $(ii)$ is due to the symmetry and to the positivity of the bilinear form $a_\eps$. $(iii)$ Concerning the third point, the equality $b(\ph_\eps,\lambar)=0$ 
	for all $\lambar\in L^2(Q_T)$ implies that $L^{\star}\ph_\eps=0$ as an $L^2(Q_T)$ function, so that if $(\ph_\eps,\lambda_\eps)\in 
	\Phi_\eps\times L^2(Q_T)$  solves the mixed formulation, then $\ph_\eps\in W_\eps$ and $\mathcal{L_\eps}(\ph_\eps,\lambda_\eps)=
	\hat{J}_\eps^{\star}(\ph_\eps)$. Finally, the first equation of the mixed formulation reads as follows: 
\[
	 \jjntqT \rho_0^{-2}\ph_\eps \,\phbar\, dx\,dt + \eps(\ph_\eps(\cdot,T),\phbar (\cdot,T))
	 - \jjntQT L^{\star}\overline\ph(x, t)\, \lambda_\eps(x,t)\, dx\, dt = l(\phbar), \quad \forall \phbar\in \Phi_\eps,
\]
	or equivalently, since the control is given by $v_\eps:=\rho_0^{-2}\ph_\eps\,1_{\omega}$, 
\[
	\jjntqT v_\eps \,\phbar\, dx\,dt + (\eps\ph_\eps(\cdot,T),\phbar (\cdot,T))
	- \jjntQT L^{\star}\overline\ph(x, t)\, \lambda_\eps(x,t)\, dx\, dt = l(\phbar), \quad \forall \phbar\in \Phi_\eps.
\]
But this means that $\lambda_\eps\in L^2(Q_T)$ is solution of the heat equation in the transposition sense. Since $y_0\in L^2(\Om)$ and 
	$v_\eps\in L^2(q_T)$, $\lambda_\eps$ must coincide with the unique weak solution to~\eqref{eq:heat} $(y_\eps=\lambda_\eps)$ such that $\lambda_\eps(\cdot,T)=-\eps\ph_\eps(\cdot,T)$. As a conclusion, the optimal pair $(y_\eps,v_\eps)$ to \eqref{P-FI-L2-eps} is
	characterized in term of the adjoint variable $\ph_\eps$ solution of (\ref{eq:mf1}) by $v_\eps=\rho_0^{-2}\ph_\eps\,1_{\omega}$ and
	$y_\eps(\cdot,T)=-\eps\ph_\eps(\cdot,T)$. 
\end{proof}

\

Theorem \ref{th:mf1} reduces the search of the approximated control to the resolution of the mixed 
	formulation \eqref{eq:mf1}, or equivalently the search of the saddle point for $\mathcal{L_\eps}$. In general, it is  
	convenient to ``augment" the Lagrangian 
	(see \cite{fortinglowinski}), and consider instead the Lagrangian $\mathcal{L}_{\eps,r}$ defined for any $r>0$ by
\[
	\left\{
		\begin{aligned}
			& \mathcal{L}_{\eps,r}(\ph,\lambda):= \frac{1}{2}a_{\eps,r}(\ph,\ph) + b(\ph,\lambda) - l(\ph), \\ 
			& a_{\eps,r}(\ph,\ph):=a_{\eps}(\ph,\ph) +r \jjntQT \vert L^{\star}\ph\vert^2 \, dx\,dt.
		\end{aligned}   
	\right.
\]
	Since $a_{\eps}(\ph,\ph)=a_{\eps,r}(\ph,\ph)$ on $W_{\eps}$ and since the function $\ph_{\eps}$ such that $(\ph_{\eps},\lambda_{\eps})$ is the saddle point of $\mathcal{L}_{\eps}$ verifies $\ph_{\eps}\in W_{\eps}$,  the lagrangian $\mathcal{L_{\eps}}$ and $\mathcal{L}_{\eps,r}$ share the same saddle-point. 

\subsubsection{Dual problem of the extremal problem \eqref{eq:min_eps_refor1}}
	
\label{duasectionmf1}

	The mixed formulation allows to solve simultaneously the dual variable $\ph_\eps$, argument of the conjugate functional \eqref{eq:min_eps_refor1}, 
	and the Lagrange multiplier $\lambda_\eps$. Since $\lambda_\eps$ turns out to be the (approximate) controlled state of \eqref{eq:heat}, we may qualify $\lambda_\eps$ 
	as the primal variable of the problem. We derive in this section the corresponding extremal problem involving only that variable $\lambda_\eps$.

	For any $r>0$, let us define the linear operator $\mathcal{A}_{\eps,r}$ from $L^2(Q_T)$ into $L^2(Q_T)$ by 
\[
	\mathcal{A}_{\eps,r}\lambda:= L^{\star}\ph, \quad \forall \lambda\in L^2(Q_T)
\]
	where $\ph \in \Phi_{\eps}$ is the unique solution to
\begin{equation}
\label{eq:imageA}
	a_{\eps,r}(\ph, \phbar) = -b(\phbar, \lambda), \quad \forall \phbar \in \Phi_{\eps}.   
\end{equation}
	Note that the assumption $r>0$ is necessary here in order to guarantee the well-posedness of \eqref{eq:imageA}. 
	Precisely, for any $r>0$, the form $a_{\eps,r}$ defines a norm equivalent to the norm on $\Phi_{\eps}$ (see \eqref{eq:norm_PHIeps}).

	We have the following crucial lemma :
\begin{lemma}
\label{propA}
	For any $r>0$, the operator $\mathcal{A}_{\eps,r}$ is a strongly elliptic, symmetric isomorphism from $L^2(Q_T)$ into $L^2(Q_T)$.
\end{lemma}
\begin{proof}
	From the definition of $a_{\eps,r}$, we easily get that $\Vert \mathcal{A}_{\eps,r}\lambda\Vert_{L^2(Q_T)}\leq r^{-1} \Vert \lambda\Vert_{L^2(Q_T)}$ 
	and the continuity of $\mathcal{A}_{\eps,r}$. Next, consider any $\lambda^{\prime}\in L^2(Q_T)$ and denote by $\ph^{\prime}$ the corresponding 
	unique solution of \eqref{eq:imageA} so that $\mathcal{A}_{\eps,r}\lambda^{\prime}:=L^{\star}\ph^{\prime}$. Relation \eqref{eq:imageA} with 
	$\phbar=\ph^{\prime}$ then implies that 
\begin{equation}
\label{arAlambda}
	 \jjntQT  (\mathcal{A}_{\eps,r}\lambda^{\prime})\lambda \,dx\,dt = a_{\eps,r}(\ph,\ph^{\prime})    
\end{equation}
	and therefore the symmetry and positivity of $\mathcal{A}_{\eps,r}$. The last relation with $\lambda^{\prime}=\lambda$ implies that 
	$\mathcal{A}_{\eps,r}$ is also positive definite.

	Finally, let us check the strong ellipticity of $\mathcal{A}_{\eps,r}$, equivalently that the bilinear functional 
$$
	(\lambda,\lambda^{\prime})\mapsto \jjntQT (\mathcal{A}_{\eps,r}\lambda)\lambda^{\prime}\,dx\,dt
$$ 
	is $L^2(Q_T)$-elliptic. Thus we want to show that 
\begin{equation}
\label{ellipticity_A}
	\jjntQT (\mathcal{A}_{\eps,r}\lambda)\lambda\,dx\,dt \geq C \Vert \lambda\Vert^2_{L^2(Q_T)}, \quad\forall \lambda\in L^2(Q_T)
\end{equation}
	for some positive constant $C$. Suppose that \eqref{ellipticity_A} does not hold; there exists then a sequence 
	$\{\lambda_n\}_{n\geq 0}$ of $L^2(Q_T)$ such that 
\[
	\Vert \lambda_n\Vert_{L^2(Q_T)}=1, \quad\forall n\geq 0, \qquad \lim_{n\to\infty} \jjntQT (\mathcal{A}_{\eps,r}\lambda_n)\lambda_n\,dx\,dt=0.
\]
	Let us denote by $\ph_n$ the solution of \eqref{eq:imageA} corresponding to $\lambda_n$. From \eqref{arAlambda}, we then obtain that 
\begin{equation}
\label{limit}
	\lim_{n\to\infty} \Vert L^{\star}\ph_n\Vert_{L^2(Q_T)}=0, \qquad \lim_{n\to\infty} \Vert\rho_0^{-1}\ph_{n}\Vert_{L^2(q_T)}=0, 
	\qquad \lim_{n\to\infty} \Vert\ph_{n}(\cdot,T)\Vert_{L^2(\Om)}=0. 
\end{equation}
From \eqref{eq:imageA} with $\lambda=\lambda_n$ and $\ph=\ph_n$, we have
\begin{equation} 
\label{phinlambdan}
	\jjntqT \rho_0^{-2}\ph_n \,\phbar\, dx\,dt + \eps\int_0^1\ph_n(\cdot,T)\phbar (\cdot,T)dx+\jjntQT 
	 (r L^{\star}\ph_n-\lambda_n)L^{\star}\phbar \, dx\,dt =0,
	\quad \forall \phbar\in \Phi_\eps.
\end{equation}
	We define the sequence $\{\phbar_n\}_{n\geq 0}$ as follows : 
\begin{equation}\nonumber
L^{\star}\phbar_n=r\,L^{\star}\ph_n-\lambda_n \quad \textrm{in}\quad Q_T, \qquad \phbar_n=0\quad \textrm{on}\quad \Sigma_T; \quad \phbar_n(\cdot, T)=0 \quad \textrm{in}\quad \Omega,
\end{equation}
so that, for all $n\geq0$, $\phbar_n$ is the solution of the backward heat equation with zero initial datum and source term 
	$r\,L^{\star}\ph_n-\lambda_n$ in $L^2(Q_T)$. 
	Using again energy type estimates, we get 
$$
	\Vert \rho^{-1}_0\phbar_{n}\Vert_{L^2(q_T)}\leq  \rho^{-1}_{\star} \Vert \phbar_{n}\Vert_{L^2(q_T)}\leq  \rho^{-1}_{\star} C_{\Omega,T} \Vert rL^{\star}\ph_n-\lambda_n\Vert_{L^2(Q_T)},
$$ 
	so that $\phbar_n\in \Phi_\eps$. Then, using \eqref{phinlambdan} with $\phbar=\phbar_n$, we get 
\[
	\Vert rL^{\star}\ph_n-\lambda_n\Vert_{L^2(Q_T)} \leq \rho_{\star}^{-1} C_{\Omega,T} \Vert  \rho^{-1}_0\ph_{n}\Vert_{L^2(q_T)}.
\]
	Then, from \eqref{limit}, we conclude that $\lim_{n\to +\infty} \Vert \lambda_n\Vert_{L^2(Q_T)}=0$ leading to a contradiction and to the strong 
	ellipticity of the operator $\mathcal{A}_{\eps,r}$.
\end{proof}

	The introduction of the operator $\mathcal{A}_{\eps,r}$ is motivated by the following proposition: 
\begin{proposition}\label{prop_equiv_dual}
	For any $r>0$, let $\ph^0\in \Phi_\eps$ be the unique solution of 
\[
	a_{\eps,r}(\ph^0,\phbar)= l(\phbar), \quad \forall \phbar\in \Phi_\eps
\]
	and let $J_{\eps,r}^{\star\star}:L^2(Q_T)\to L^2(Q_T)$ be the functional defined by 
\[
	J_{\eps,r}^{\star\star}(\lambda) := \frac{1}{2} \jjntQT (\mathcal{A}_{\eps,r} \lambda) \,\lambda\, dx \,dt - b(\ph^0, \lambda).
\]
	The following equality holds : 
\[
	\sup_{\lambda\in L^2(Q_T)}\inf_{\ph\in \Phi_\eps} \mathcal{L}_{\eps,r}(\ph,\lambda) 
	= - \inf_{\lambda\in L^2(Q_T)} J_{\eps,r}^{\star\star}(\lambda)~~+~ \mathcal{L}_{\eps,r}(\ph^0,0).
\]
\end{proposition}
\begin{proof}
	For any $\lambda\in L^2(Q_T)$, let us denote by $\ph_{\lambda}\in \Phi_\eps$ the minimizer of $\ph\mapsto \mathcal{L}_{\eps,r}(\ph,\lambda)$; 
	$\ph_{\lambda}$ satisfies the equation
\[
	a_{\eps,r}(\ph_{\lambda},\phbar) + b(\phbar,\lambda) = l(\phbar), \quad \forall \phbar\in \Phi_{\eps}
\]
	and can be decomposed as follows : $\ph_{\lambda} = \psi_{\lambda} + \ph^0$ where $\psi_{\lambda}\in \Phi_{\eps}$ solves 
\[
	a_{\eps,r}(\psi_{\lambda},\phbar) + b(\phbar,\lambda) = 0, \quad \forall \phbar\in \Phi_{\eps}.
\]
	We then have 
\[
	\begin{aligned}
		\inf_{\ph\in \Phi_{\eps}} \mathcal{L}_{\eps,r}(\ph,\lambda)  & = \mathcal{L}_{\eps,r}(\ph_\lambda,\lambda) = \mathcal{L}_{\eps,r}(\psi_\lambda+ \ph^0,\lambda)\\
		&= \frac{1}{2}a_{\eps,r}(\psi_\lambda+ \ph^0,\psi_\lambda+ \ph^0) + b(\psi_\lambda+ \ph^0,\lambda) - l(\psi_\lambda+ \ph^0) \\
		& := X_1 +X_2 + X_3
	\end{aligned}
\]
	with
\[
	\left\{
		\begin{aligned}
			& X_1 := \frac{1}{2}a_{\eps,r}(\psi_\lambda,\psi_\lambda) + b(\psi_\lambda,\lambda) + b(\ph^0,\lambda)  \\
			& X_2 := a_{\eps,r}(\ph^0,\psi_\lambda) - l(\psi_\lambda), \quad X_3 :=  \frac{1}{2}a_{\eps,r}(\ph^0,\ph^0) - l(\ph^0).
		\end{aligned}
	\right.
\]
	From the definition of $\ph^0$, $X_2=0$ while $X_3= \mathcal{L}_{\eps,r}(\ph^0,0)$. Eventually, from the definition of $\psi_\lambda$, 
\[
	X_1 = \frac{1}{2}b(\psi_\lambda,\lambda) + b(\ph^0,\lambda)
	 = -\frac{1}{2}\jjntQT (\mathcal{A}_{\eps,r} \lambda) \,\lambda\, dx \,dt + b(\ph^0,\lambda)=-J_{\eps,r}^{\star \star}(\lambda)
\]
	and the result follows. 
\end{proof}
	From the ellipticity of the operator $\mathcal{A}_{\eps,r}$, the minimization of the functional $J_{\eps,r}^{\star\star}$ over $L^2(Q_T)$ is well-posed. 
	It is interesting to note that with this extremal problem involving only $\lambda$, we are coming to the primal variable, controlled solution of 
	\eqref{eq:heat} (see Theorem \ref{th:mf1}, (iii)). This argument allows notably to avoid the direct minimization of $J_{\eps}$ (introduced in Problem \eqref{P-FI-L2-eps}) with respect to the state $y$ (ill-conditioned due to the term $\eps^{-1}$ for $\eps$ small). Here, any constraint equality is assigned to the variable $\lambda$.


\subsection{The penalized case : Mixed formulation II (relaxing the condition $L^{\star}\ph_{\eps}=0$ in $L^2(Q_T)$)}
\label{second_mixed_form}

The previous mixed formulation amounts to find a backward solution $\ph_{\varepsilon}$ satisfying the condition $L^{\star}\ph_{\eps}=0$ in $L^2(Q_T)$. For numerical purposes, it may be interesting to relax this condition, which typically leads to the use of $C^1$ type approximations in the space variable (see Section \ref{numerics}). In order to circumvent this difficulty, we introduce and analyze in this section a second penalized mixed formulation  
where the condition on $\ph_{\eps}$ is relaxed, namely we impose the constraint $L^{\star}\ph_{\eps}=0$ in $L^2(0,T;H^{-1}(\Omega))$. 

Considering as before the full adjoint variable $\ph$ as the main variable, we associated to \eqref{eq:min_eps_adj} the following extremal problem :
\begin{equation}
\label{eq:min_eps_refor2}
	\min_{\ph\in \widehat W_\eps} \hat J^{\star}_\eps(\ph)={1\over2} \jjntQT \rho_0^{-2}|\ph(x, t)|^2 dx\,dt 
	+ {\eps\over2}\|\ph(\cdot,T)\|^2_{L^2(\Om)}+\int_\Om y_0(x) \ph(x, 0) dx,
\end{equation}
over the space $\widehat W_\eps = \left\{ \ph\in \widehat\Phi_\eps:\, L^{\star} \ph=0 \hbox{ in } L^2(0,T;H^{-1}(\Omega))\right\}$. The space $\widehat\Phi_\eps$ is again defined as the completion of $\Phi^0$ with respect to the inner product  	
\begin{equation}\nonumber
	(\ph, \phbar)_{\widehat\Phi_\eps} := \iint_{q_T} \rho_0^{-2}\ph \,\phbar\, dx\,dt + \eps(\ph(\cdot,T),\phbar (\cdot,T))
	+ \eta \left(\iint_{Q_T} \nabla\ph \,\nabla\phbar\, dx\,dt + \int_0^T(\ph_t,\phbar_t)_{H^{-1}}dt\right),
\end{equation}
defined over $\Phi^0$. We denote by $\|\cdot\|_{\widehat\Phi_\eps}$ the 
	associated norm such that 
\begin{equation} \label{eq:normhatWveps}
	\|\ph\|_{\widehat \Phi_\eps}^2 :=\|\rho^{-1}_0\ph\|^2_{L^2(q_T)} + \eps\|\ph(\cdot,T)\|^2_{L^2(\Om)}+ \eta (
	\|\nabla \ph\|^2_{L^2(Q_T)}+\|\ph_t\|^2_{L^2(0,T;H^{-1})}),	 \quad \forall \ph \in \widehat\Phi_\eps.
\end{equation}

\begin{lemma}
	The equality $\widehat W_\eps= W_\eps$ holds. Therefore, the minimization problem \eqref{eq:min_eps_refor2} is
	equivalent to the minimization \eqref{eq:min_eps_refor1}. 
\end{lemma}
\begin{proof}
	First, let us see that $ W_\eps \subset \widehat W_\eps$. To do this, it is enough see that $ \Phi_\eps\subset \widehat \Phi_\eps$. 
	In fact, if $\ph\in  \Phi_\eps$ then there exists a sequence $(\ph^n)_{n=1}^{\infty}$
	in $\Phi_0$ such that $\ph^n \to \ph$ in $\Phi_\eps$. So, we can conclude that  $\ph^n \to \ph$ in $L^2(0,T;H_0^1(\Om))$ and
	$\ph^n_t \to \ph_t$ in $L^2(0,T;H^{-1}(\Om))$. Hence, $\ph^n \to \ph$ in $\widehat\Phi_\eps$. 
	
	Secondly, let us see that $ \widehat W_\eps \subset  W_\eps$. Indeed, if $\widehat \ph \in \widehat W_\eps$ then 
	$\widehat \ph \in \widehat \Phi_\eps$ and $L^\star\widehat\ph=0$. Let us denote $\widehat \ph_T:=\widehat\ph(\cdot,T)$, 
	so there exists a sequence $(\ph^n_T)_{n=1}^{\infty}$ in $C^{\infty}_0(\Om)$ such that  $\ph^n_T \to \widehat\ph_T$ in $L^2(\Om)$.
	Now, if $(\ph^n)_{n=1}^{\infty}$ is a sequence such that $L^\star\ph^n=0$, $\ph^n=0$ on $\Sigma_T$ and $\ph^n(\cdot,T):= \ph_T^n$
	then this sequence belongs to $\Phi^0$. Hence, $\ph^n \to \widehat\ph$ in $\widehat\Phi_\eps$ and $\ph^n \to \widehat\ph$ in $\Phi_\eps$.
	Therefore, $\widehat\ph$ belongs to $ W_\eps$.
\end{proof}
		
	\
	
	The main variable is now $\ph$ submitted to the constraint equality $L^{\star}\ph=0 \in L^2(0,T;H^{-1})$. As before, this constraint is addressed by 
	introducing a mixed formulation given as follows~: find $(\ph_\eps,\lambda_\eps)\in \widehat \Phi_\eps\times \widehat\Lambda_\eps$ 
	solution of 
\begin{equation} 
\label{eq:mf2}
	\left\{
		\begin{array}{rcll}
			\noalign{\smallskip} 
			\hat a_\eps(\ph_\eps, \phbar) + \hat b(\phbar, \lambda_\eps) 
			& = &\hat l(\phbar), & \qquad \forall \phbar \in\widehat  \Phi_\eps \\
			\noalign{\smallskip}\hat b(\ph_\eps, \lambar) & = & 0, & \qquad \forall \lambar \in \widehat\Lambda_\eps,
		\end{array}
	\right.
\end{equation}
	where $\widehat\Lambda_\eps:= L^2(0,T;H^1_0(\Om))$ and 
\begin{align}
	\nonumber
	&\hat a_\eps :\widehat  \Phi_\eps \times\widehat  \Phi_\eps \to \mathbb{R},  \quad \hat a_\eps(\ph, \phbar) 
			   := \jjntqT \rho_0^{-2}\ph \,\phbar\, dx\,dt + \eps(\ph(\cdot,T),\phbar (\cdot,T))_{L^2(\Omega)} \\
	\nonumber
	&\hat b:\widehat  \Phi_\eps \times \widehat\Lambda_\eps \to \mathbb{R}, \\
	\nonumber
	&\hat b(\ph, \lambda) := - \int_0^T <L^{\star}\ph,\lambda>_{H^{-1}(\Omega),H^1_0(\Omega)}dt\\
	\nonumber
	&\hspace{1.1cm} = \int_0^T\!\!\!\!\! \langle\ph_t(t), \lambda(t)\rangle_{H^{-1},H^1_0}\, dt
			  - \jjntQT \biggl((c(x)\nabla \ph,\nabla \lambda) +d(x,t) \ph\lambda\biggr)\,dx\,dt\\
       \nonumber
	&\hat l:\widehat  \Phi_\eps \to \mathbb{R},  \quad\hat  l(\ph) := -\int_\Om y_0(x) \ph(x, 0) dx.
\end{align}

Similarly to Theorem \ref{th:mf1}, the following holds :
\begin{theorem}\label{th:mf2}
	\begin{enumerate}
	\item 
		The mixed formulation \eqref{eq:mf2} is well-posed.
	\item 
		The unique solution $( \ph_\eps, \lambda_\eps)\in \widehat \Phi_\eps\times \widehat\Lambda_\eps$ 
		is the unique saddle-point of the Lagrangian operator ${\mathcal{\widehat L_\eps} :  \widehat \Phi_\eps\times \widehat\Lambda_\eps \to \mathbb{R}}$ 
		defined by
		\begin{equation}
		\label{eq:calLhat}
			\mathcal{\widehat L_\eps}(\ph, \lambda) := {1\over2}\hat a_\eps(\ph, \ph) +\hat b(\ph, \lambda)  - \hat l(\ph).
		\end{equation}
	\item
	      The optimal function $\ph_{\eps}$ is the minimizer of  $\hat J_{\eps}^{\star}$ over $\widehat{W}_{\eps}$ while the optimal multiplier $\lambda_{\eps}\in \hat{\Lambda}_{\eps}$ is the weak solution of the heat equation \eqref{eq:heat}.  
\end{enumerate}
\end{theorem}
\begin{proof}
	We easily check that the bilinear form $\hat a_\eps$ is continuous over $ \widehat\Phi_\eps\times \widehat \Phi_\eps$, symmetric and 
	positive and that the bilinear form $\hat b$ is continuous over $\widehat\Phi_\eps\times \widehat\Lambda_\eps$. Furthermore, the continuity 
	of the linear form $\hat l$ over $\widehat\Phi_\eps$ is a direct by the continuous embedding $\widehat \Phi_\eps\hookrightarrow C^0([0,T];L^2(\Om))$.
	Therefore, the well-posedness of the mixed formulation is a consequence of the following two properties (see \cite{BrezziFortin}):
\begin{itemize}
	\item 
		$\hat a_\eps$ is coercive on $\mathcal{N}(\hat b)$, where $\mathcal{N}(\hat b)$ denotes the kernel of $\hat b$~:
		\[
			\mathcal{N}(\hat b) = \left\{ \ph \in \widehat\Phi_\eps \textrm{ such that }\hat b(\ph, \lambda) = 0 
			\textrm{ for every } \lambda \in \widehat\Lambda_\eps \right\}.
		\]
	\item  
		$\hat b$ satisfies the usual ``inf-sup" condition over $\widehat\Phi_\eps\times \widehat\Lambda_\eps$: there exists $\delta > 0$ such that 
		\begin{equation}
		\label{eq:infsup2}
			\inf_{\lambda \in  \widehat\Lambda_\eps} \sup_{\ph\in \widehat \Phi_\eps} \frac{\hat b(\ph, \lambda)}{\|\ph\|_{\widehat \Phi_\eps}
			 \|\lambda\|_{ \widehat\Lambda_\eps}} \geq \delta.
		\end{equation}
\end{itemize}

	From the definition of $\hat a_\eps$, the first point is clear : for all $\ph\in \mathcal{N}(\hat b_\eps)= \widehat W_\eps$, thanks to classical
	energy estimates, we have
\[
	\begin{alignedat}{2}
		\hat a_\eps(\ph,\ph)=&~\|\rho^{-1}_0\ph\|^2_{L^2(q_T)}+{\eps\over2}\| \ph(\cdot,T)\|^2_{L^2(\Om)}+{\eps\over2}\| \ph(\cdot,T)\|^2_{L^2(\Om)}\\
					 \geq&~\|\rho^{-1}_0\ph\|^2_{L^2(q_T)}+{\eps\over2}\| \ph(\cdot,T)\|^2_{L^2(\Om)}
					 		+\eps\, C(\| \nabla\ph\|^2_{L^2(Q_T)}+\|\ph_t\|^2_{L^2(0,T;H^{-1})})\\
					    \geq&~C_{\eps,\eta}\Vert \ph\Vert^2_{\widehat \Phi_\eps},
	\end{alignedat}
\]
	where $C= C(T,c_0,\|d\|_{\infty})>0$ and $C_{\eps,\eta}:= \min( 2^{-1}, C\,\eps\,\eta^{-1})$.
	
	Let us check the inf-sup condition. For any fixed $\lambda^0\in \widehat\Lambda_\eps$, we define the (unique) element 
	$\ph^0$ of 
	\[
	L^{\star}\ph^0=\Delta \lambda^0 \quad \textrm{in}\quad Q_T, \qquad \ph^0=0\quad \textrm{on}\quad \Sigma_T; \quad \ph^0(\cdot, T)=0 \quad \textrm{in}\quad \Omega,
\]
so that $\ph^0$ solves the backward heat equation with source term $\Delta\lambda^0$, null Dirichlet boundary
	condition and zero initial state. Since $\Delta\lambda^0\in L^2(0,T;H^{-1})$, then $\ph^0\in \widehat \Phi_\eps$: 
	precisely, using energy estimates, there exists a constant $C > 0$ such that $\ph^0$ satisfies the inequalities
\[
	 \|\nabla\ph^0\|^2_{L^2(Q_T)} \leq  C\|\nabla \lambda^0\|_{L^2(Q_T)}^2 
\]
	and
\[
\begin{alignedat}{2}
	 \|\ph^0\|^2_{\widehat \Phi_\eps}=&~\|\rho^{-1}_0\ph^0\|^2_{L^2(q_T)}+\eps\| \ph^0(\cdot,T)\|^2_{L^2(\Om)} 
	 						     +\eta(\|\nabla \ph^0\|^2_{L^2(Q_T)}+\|\ph^0_t\|^2_{L^2(0,T;H^{-1})})\\
							=&~\|\rho^{-1}_0\ph^0\|^2_{L^2(q_T)}+\eta(\|\nabla \ph^0\|^2_{L^2(Q_T)}+\|\ph^0_t\|^2_{L^2(0,T;H^{-1})})\\
						     \leq&~C_\eta\|\nabla\lambda^0\|^2_{L^2(Q_T)}.
\end{alignedat}
\]
where $C= C(T,\|c\|_{\infty},\|d\|_{\infty})>0$ and $C_{\eta}:= C (1+\eta)$.

	Consequently, $\ph^0\in\widehat \Phi_\eps$. In particular, we have $\hat b(\ph^0,\lambda^0)= \Vert\nabla \lambda^0\Vert^2_{L^2(Q_T)}$ and
\[
	\sup_{\ph \in \widehat \Phi_\eps} \frac{\hat b(\ph, \lambda^0)}{\|\ph\|_{\widehat \Phi_\eps} \|\lambda^0\|_{ \widehat\Lambda_\eps}} 
	\geq \frac{\hat b(\ph^0, \lambda^0)}{\|\ph^0\|_{\widehat \Phi_\eps} \|\lambda^0\|_{ \widehat\Lambda_\eps}} 
	\geq \frac{\Vert\nabla \lambda^0\Vert^2_{L^2(Q_T)}}{C_\eta^{1/2} \|\nabla\lambda_0\|_{L^2(Q_T)} \|\nabla\lambda_0\|_{L^2(Q_T)}}.
\]

	Combining the above two inequalities, we obtain
\[
	\sup_{\ph \in \widehat \Phi_\eps} \frac{\hat b(\ph, \lambda^0)}{\|\ph\|_{\widehat \Phi_\eps} \|\lambda^0\|_{\widehat\Lambda_\eps}} \geq \frac{1}{\sqrt{C_\eta}}
\]
	and, hence, \eqref{eq:infsup2} holds with $\delta =  C_\eta^{-\frac{1}{2}}$.

	The point $(ii)$ is due to the symmetry and to the positivity of the bilinear form $\hat a_\eps$. 
	
Concerning the third assertion, the equality 
	$b(\hat\ph_\epsilon,\lambar)=0$ for all $\lambar\in  \widehat\Lambda_\eps$ implies that $L^{\star}\ph_\epsilon=0$ as an 
	$L^2(0,T;H^{-1})$ function, so that if $(\ph_\eps,\lambda_\eps)\in\widehat \Phi_\eps\times  \widehat\Lambda_\eps$  solves the mixed 
	formulation \eqref{eq:mf2}, then $\ph_\eps\in\widehat W_\eps$ and $\mathcal{\widehat L_\eps}(\ph_\eps,\lambda_\eps)
	=\hat J_\eps^{\star}(\ph_\eps)$. This implies that $\ph_\eps$ of the two mixed formulations coincide.

	Assuming $y_0\in L^2(\Om)$ and $v\in L^2(q_T)$, it is said here that $y\in L^2(0,T;H^1_0(\Om))$ 
	is the (unique) solution by transposition of the heat equation \eqref{eq:heat} if and only if, for every $g\in L^2(0,T;H^{-1})$, we have
\[	
	 \int_0^T \langle g,y\rangle_{H^{-1},H^1_0} dt = \iint_{q_T} v \,\phbar\, dx\,dt + (\phbar(\cdot,0),y_0)_{L^2(\Om)},
\]	
	where $\phbar$ solves
\[
\label{eq:heat_trans}
	L^{\star}\phbar= g\quad \textrm{in}\quad  Q_T,  \qquad \phbar  = 0\quad \textrm{on}\quad  \Sigma_T, \quad \phbar(\cdot, T) = 0\quad \textrm{in}\quad \Om.
\]
As
	$g\mapsto (v,\phbar)_{L^2(q_T)}+(\phbar(\cdot,0),y_0)_{L^2(\Om)}$ is linear and continuous on $L^2(0,T;H^{-1})$
	the Riesz representation theorem guarantees that this definition makes sense.

	Finally, the first equation of the mixed formulation \eqref{eq:mf2} reads as follows: 
\[
\begin{alignedat}{2}
	 &\iint_{q_T} \rho_0^{-2}\ph_\eps \,\phbar\, dx\,dt + \eps(\ph_\eps(\cdot,T),\phbar (\cdot,T))
	 +\int_0^T \langle\phbar_t ,\lambda_\eps\rangle_{H^{-1},H^1_0} \\
	 &- \iint_{Q_T}(c(x)\nabla\phbar,\nabla\lambda_\eps) +d(x,t) \phbar\lambda_\eps\,dx\,dt
	 =\hat l(\phbar), \quad \forall \phbar\in \widehat\Phi_\eps,
\end{alignedat}
\]
	or equivalently, since the control is given by $v_\eps=\rho_0^{-2}\ph_\eps$ (recall that the formulations 
	\eqref{eq:min_eps_refor1} and \eqref{eq:min_eps_refor2} are equivalent), 
\[
\begin{alignedat}{2}
	&\iint_{q_T} v_\eps \,\phbar\, dx\,dt + (\eps\ph_\eps(\cdot,T),\phbar (\cdot,T))
	 + \int_0^T \langle\phbar_t ,\lambda_\eps\rangle_{H^{-1},H^1_0}\,dt  \\
	 &- \iint_{Q_T}(c(x)\nabla \phbar,\nabla\lambda_\eps) +d(x,t) \phbar\lambda_\eps\,dx\,dt
	 =\hat l(\phbar), \quad \forall \phbar\in\widehat \Phi_\eps.
\end{alignedat}
\]
	But this means that $\lambda_\eps\in\widehat\Lambda_\eps$ is solution of the heat equation in the
	transposition sense.	Since $y_0\in L^2(\Om)$ and $v_\eps\in L^2(q_T)$, $\lambda_\eps$ must coincide with the unique weak solution 
	to~\eqref{eq:heat} $(y_\eps=\lambda_\eps)$ and, in particular, we can conclude that $y_\eps(\cdot,T)=-\eps\ph_\eps(\cdot,T)$. 
	So from the unique of the weak solution, the solution $(\ph_\eps,\lambda_\eps)$ of the two mixed formulation coincides.
\end{proof}
	
	\
	
The equivalence of the mixed formulation (\ref{eq:mf2}) with the mixed formulation (\ref{eq:mf1}) is related to the regularizing property of the heat kernel. At the numerical level, the advantage is that this formulation leads naturally to continuous spaces of approximation both in time and space.


\subsection{Third mixed formulation of the controllability problem : the limit case $\eps=0$}
\label{third_mixed_form}

We consider in this section the limit case of Section \ref{penalized_one} corresponding to $\varepsilon=0$, i.e. to the null controllability. 
The conjugate functional $J^{\star}$ corresponding to this case is given in the introduction, see \eqref{eq:J}, with a weight $\rho^{-2}_0$ (recall that $\rho_0\in \mathcal{R}$ defined by \eqref{def_spaceR}) in the first term, precisely
\begin{equation}\label{eq:Jb}
\min_{\ph_T\in \mathcal{H}} \; J^{\star}(\ph_T) := \frac{1}{2} \jjntqT  \rho_0^{-2}(x,t) |\ph(x,t)|^2\, dx dt + (y_0, \ph(\cdot, 0))_{L^2(\Omega)} 
\end{equation}
where the variable $\ph$ solves the backward heat equation \eqref{eq:wr} and $\mathcal{H}$ is again defined as the completion of the $L^2(\Omega)$ space with respect to the norm $\Vert \ph_T\Vert_{\mathcal{H}}:=\Vert\rho_0^{-1} \ph\Vert_{L^2(q_T)}$. As explained in the introduction, the limit case is much more singular due to the hugeness of the space $\mathcal{H}$. At the limit $\eps=0$, the control of the terminal state $\ph(\cdot,T)$ is lost in $L^2(\Omega)$. 

Let $\rho\in \mathcal{R}$. Proceeding as before, we consider again the space $\widetilde\Phi_0=\{\ph\in C^2(\overline{Q_T}): \ph=0\,\textrm{on}\, \Sigma_T\}$ and then, for any $\eta>0$, we define the bilinear form 
\[
	(\ph, \phbar)_{\widetilde\Phi_{\rho_0,\rho}} := \jjntqT \rho_0^{-2}\ph \,\phbar\, dx\,dt  + \eta  \jjntQT \rho^{-2} L^{\star}\ph\, L^{\star}\phbar\, dx\,dt , \quad \forall \ph, \phbar \in \widetilde\Phi_0.
\]	
The introduction of the weight $\rho$, which does not appear in the original problem \eqref{eq:Jb} will be motivated at the end of this Section. 
From the unique continuation property for the heat equation, this bilinear form defines for any $\eta>0$ a scalar product. Let then $\widetilde\Phi_{\rho_0,\rho}$ be the completion of $\widetilde\Phi_0$ for this scalar product.  We denote the norm over $\widetilde\Phi_{\rho_0,\rho}$ by $\Vert\cdot\Vert_{\widetilde\Phi_{\rho_0,\rho}}$ such that 
\begin{equation} 
\label{eq:norm_PHI}
	\|\ph\|_{\widetilde\Phi_{\rho_0,\rho}}^2 :=\|\rho^{-1}_0\ph\|^2_{L^2(q_T)} + \eta \| \rho^{-1}L^{\star}\ph\|^2_{L^2(Q_T)}, \quad \forall \ph \in \widetilde\Phi_{\rho_0,\rho}.
\end{equation}
Finally, we defined the closed subset $\widetilde W_{\rho_0,\rho}$ of $\widetilde\Phi_{\rho_0,\rho}$ by 
\[
\widetilde W_{\rho_0,\rho}=\{\ph\in \widetilde\Phi_{\rho_0,\rho} : \rho^{-1}L^{\star}\ph=0 \,\, \textrm{in}\,\, L^2(Q_T)\}
\]
and we endow $\widetilde W_{\rho_0,\rho}$ with the same norm than $\widetilde\Phi_{\rho_0,\rho}$. 

We then define the following extremal problem : 
\begin{equation}
\label{eq:min_eps_refor3}
	\min_{\ph\in \widetilde W_{\rho_0,\rho}} \hat{J}^{\star}(\ph)={1\over2} \jjntqT \rho_0^{-2}|\ph(x, t)|^2 dx\,dt + (y_0,\, \ph(\cdot,0))_{L^2(\Omega)}.
\end{equation}
For any $\ph\in \widetilde W_{\rho_0,\rho}$, $L^{\star}\ph=0$ a.e. in $Q_T$ and $\Vert \ph\Vert_{\widetilde W_{\rho_0,\rho}}=\Vert \rho_0^{-1}\ph\Vert_{L^2(q_T)}$ so that $\ph(\cdot,T)$ belongs by definition to the abstract space $\mathcal{H}$: consequently, extremal problems \eqref{eq:min_eps_refor3} and (\ref{eq:Jb}) are equivalent. In particular, from the regularizing property of the heat kernel, $\ph(\cdot,0)$ belongs to $L^2(\Omega)$ and the linear term in $\ph$ in $\hat{J}^{\star}$ is well defined.

Then, we consider the following mixed formulation~: find $(\ph,\lambda)\in \widetilde \Phi_{\rho_0,\rho}\times L^2(Q_T)$ solution of 
\begin{equation} 
\label{eq:mf3}
	\left\{
		\begin{array}{rcll}
			\noalign{\smallskip} 
			\tilde a(\ph, \phbar) + \tilde b(\phbar,\lambda) 
			& = &\tilde l(\phbar), & \quad \forall \phbar \in\widetilde\Phi_{\rho_0,\rho} \\
			\noalign{\smallskip}\tilde b( \ph, \lambar) & = & 0, & \quad \forall \lambar \in L^2(Q_T),
		\end{array}
	\right.
\end{equation}
	where
\begin{align}
	\nonumber
	&\tilde a :\widetilde\Phi_{\rho_0,\rho} \times\widetilde\Phi_{\rho_0,\rho} \to \mathbb{R},  \quad \tilde a(\ph, \phbar) 
			   = \jjntqT \rho_0^{-2}\ph \,\phbar\, dx\,dt  \\
	\nonumber
	&\tilde b:\widetilde\Phi_{\rho_0,\rho} \times L^2(Q_T)\to \mathbb{R}, \quad \tilde b(\ph, \lambda) = -\jjntQT\rho^{-1}L^\star\ph \,\lambda\,dx\,dt\\
	\nonumber
	&\tilde l:\widetilde  \Phi_{\rho_0,\rho} \to \mathbb{R},  \quad\tilde  l(\ph) =- (y_0, \ph(\cdot, 0))_{L^2(\Omega)}.
\end{align}

Before studying this mixed formulation, let us do the following comment. The continuity of $\tilde l$ over the space $\widetilde \Phi_{\rho_0,\rho}$ holds true for a precise choice of the weights which appear in Carleman type estimates for parabolic equations (see \cite{FursikovImanuvilov}): we recall the following important result.  

\begin{proposition}[ \cite{FursikovImanuvilov}]
Let the weights $\rho^c,\rho_0^c \in \mathcal{R}$ (see \eqref{def_spaceR}) be defined  as follows :
 \begin{equation}
\label{weights}
\begin{aligned}
& \rho^c(x,t):=  \exp\left({\beta(x) \over T-t}\right), \quad \beta(x):=  K_{1}\left(e^{K_{2}} \!-\! e^{\beta_{0}(x)}\right), \\
& \rho_0^c(x,t):=  (T-t)^{3/2}\rho^c(x,t), 
\end{aligned}
\end{equation}
  where the $K_{i}$ are sufficiently large positive constants (depending on $T$, $c_0$ and $\|c\|_{C^1(\overline{\Omega})}$) such that 
\[
\beta_{0} \in C^\infty(\overline{\Omega}), \,\beta > 0\quad\textrm{in}\quad \Omega,\, \beta=0\quad\textrm{on}\quad \partial\Omega, \quad\textrm{Supp}(\nabla\beta) \subset \overline{\Omega}\setminus\omega.
\]
Then, there exists a constant $C>0$, depending only on $\omega, T$,  such that 
\begin{equation}
\Vert \ph(\cdot,0)\Vert_{L^2(\Omega)} \leq C \Vert \ph\Vert_{\widetilde\Phi_{\rho_0^c,\rho^c}}, \quad \forall \ph\in \widetilde\Phi_{\rho_0^c,\rho^c}.   \label{crucial_estimate}
\end{equation}
\end{proposition}
The estimate (\ref{crucial_estimate}) is a consequence of the celebrated global Carleman inequality satisfied by the solution of (\ref{eq:wr}), introduced and popularized in \cite{FursikovImanuvilov}.
It allows to obtain the following existence and uniqueness result : 
\begin{theorem}\label{th:mf3} Let $\rho_0\in \mathcal{R}$ and $\rho\in \mathcal{R}\cap L^{\infty}(Q_T)$ and assume that there exists a positive constant $K$ such that 
\begin{equation}
\rho_0\leq K \rho_0^c, \quad \rho\leq K \rho^c \quad \textrm{in}\quad Q_T.    \label{hypK}
\end{equation}
Then, we have :
	\begin{enumerate}
	\item 
		The mixed formulation \eqref{eq:mf3} defined over $\widetilde\Phi_{\rho_0,\rho}\times L^2(Q_T)$ is well-posed.
	\item 
		The unique solution $(\ph,\lambda)\in \widetilde \Phi_{\rho_0,\rho}\times L^2(Q_T)$ is the unique saddle-point of the Lagrangian ${\mathcal{\widetilde L} :  \widetilde \Phi_{\rho_0,\rho}\times L^2(Q_T) \to \mathbb{R}}$ defined by
		\begin{equation}
		\label{eq:calLtilde}
			\mathcal{\widetilde L}(\ph, \lambda) = {1\over2}\tilde a(\ph, \ph) +\tilde b(\ph, \lambda)  - \tilde l(\ph).
		\end{equation}
	\item  The optimal function $\ph$ is the minimizer of $\hat{J}^{\star}$ over $\widetilde \Phi_{\rho_0,\rho}$ while
	$\rho^{-1}\lambda\in L^2(Q_T)$ is the state of the heat equation \eqref{eq:heat} in the weak sense.
\end{enumerate}
\end{theorem}

\begin{proof} The proof is similar to the proof of Theorem \ref{th:mf1}. From the definition, the bilinear form $\tilde a$ is continuous over $\widetilde\Phi_{\rho_0,\rho}\times \widetilde\Phi_{\rho_0,\rho}$, symmetric and positive and the bilinear form $\tilde b$ is continuous over $\widetilde\Phi_{\rho_0,\rho}\times L^2(Q_T)$. Furthermore, the continuity of the linear form $\tilde l$ over $\widetilde\Phi_{\rho_0,\rho}$ is the consequence of the estimate \eqref{crucial_estimate}: precisely, from the assumptions \eqref{hypK}, the inclusion $\widetilde\Phi_{\rho_0,\rho}\subset \widetilde\Phi_{\rho_0^c,\rho^c}$ hold true. Therefore, estimate \eqref{crucial_estimate} implies 
\begin{equation}
\Vert \ph(\cdot,0)\Vert_{L^2(\Omega)} \leq C \Vert \ph\Vert_{\widetilde\Phi_{\rho_0^c,\rho^c}} \leq C K^{-1} \Vert \ph\Vert_{\widetilde\Phi_{\rho_0,\rho}}, \quad \forall \ph\in \widetilde\Phi_{\rho_0,\rho}.  
\end{equation}
Therefore, the well-posedeness of the formulation \eqref{eq:mf3} is the consequence of two properties: first, the coercivity of the form $\tilde a$ on the kernel $\mathcal{N}(\tilde b):=\{\ph\in\widetilde \Phi_{\rho_0,\rho} : \tilde{b}(\ph,\lambda)=0\, \forall \lambda\in L^2(Q_T)\}$: again, this holds true since the kernel coincides with the space $\widetilde W_{\rho_0,\rho}$. Second, the inf-sup property which reads as : 
\begin{equation}
\label{eq:infsup3}
\inf_{\lambda \in L^2(Q_T)} \sup_{\ph\in \widetilde\Phi_{\rho_0,\rho}} \frac{\tilde b(\ph, \lambda)}{\|\ph\|_{\widetilde\Phi_{\rho_0,\rho}} \|\lambda\|_{L^2(Q_T)}} \geq \delta
\end{equation}
for some $\delta >0$.
For any fixed $\lambda^0\in L^2(Q_T)$, we define the unique element $\ph^0$ solution of 
\begin{equation}
\rho^{-1}L^{\star}\ph=-\lambda^0 \, \textrm{in}\, Q_T, \quad \ph=0\,\textrm{ on }\,\Sigma_T, \quad \ph(\cdot,T)=0 \,\, \textrm{in}\,\, \Omega \nonumber.
\end{equation}
Using energy estimates, we have 
\begin{equation}
\Vert \rho_0^{-1}\ph^0\Vert_{L^2(q_T)} \leq \rho_{\star}^{-1} \Vert \ph^0\Vert_{L^2(Q_T)}  \leq \rho_{\star}^{-1}\Vert \rho \lambda^0\Vert_{L^2(Q_T)} \leq \rho_{\star}^{-1}\Vert \rho\Vert_{L^{\infty}(Q_T)} \Vert \lambda^0\Vert_{L^2(Q_T)}   \label{estimate_L2L2}
\end{equation}
which proves that $\ph^0\in\widetilde \Phi_{\rho_0,\rho}$ and that  
\[
	\sup_{\ph \in \widetilde\Phi_{\rho_0,\rho}} \frac{\tilde b(\ph, \lambda^0)}{\|\ph\|_{\widetilde\Phi_{\rho_0,\rho}} \|\lambda^0\|_{L^2(Q_T)}} 
	\geq \frac{\tilde b(\ph^0, \lambda^0)}{\|\ph^0\|_{\widetilde\Phi_{\rho_0,\rho}} \|\lambda^0\|_{L^2(Q_T)}} 
	= \frac{\|\lambda^0\|_{L^2(Q_T)}}{\left( \|\rho^{-1}_0\ph^0\|^2_{L^2(q_T)} + \eta \|\lambda_0\|^2_{L^2(Q_T)} \right)^\frac{1}{2}}.
\]

	Combining the above two inequalities, we obtain
\[
	\sup_{\ph_0 \in \widetilde\Phi_{\rho_0,\rho}} \frac{b(\ph_0, \lambda_0)}{\|\ph_0\|_{\widetilde\Phi_{\rho_0,\rho}} \|\lambda_0\|_{L^2(Q_T)}} \geq \frac{1}{\sqrt{\rho_{\star}^{-2}\Vert \rho\Vert^2_{L^{\infty}(Q_T)}  + \eta}}
\]
	and, hence, \eqref{eq:infsup3} holds with $\delta = \left( \rho_{\star}^{-2}\Vert \rho\Vert^2_{L^{\infty}(Q_T)}+ \eta \right)^{-1/2}$.

The point $(ii)$ is again due to the positivity and symmetry of the form $\tilde a$.

Concerning the last point of the Theorem, the equality 
	$\tilde b(\ph,\lambar)=0$ for all $\lambar\in L^2(Q_T)$ implies that $\rho^{-1}L^{\star}\ph=0$ as an $L^2(Q_T)$ 
	function, so that if $(\ph,\lambda)\in\widetilde \Phi_{\rho,\rho_0}\times L^2(Q_T)$  solves the mixed formulation \eqref{eq:mf3}, 
	then $\ph\in\widetilde W_{\rho,\rho_0}$ and $\mathcal{\widetilde L}(\ph,\lambda)=\hat{J}^{\star}(\ph)$. 
	 Finally, the first equation of the mixed formulation \eqref{eq:mf3} reads as follows: 
\begin{equation}\nonumber
\begin{alignedat}{2}
	 &\jjntqT \rho_0^{-2}\ph \,\phbar\, dx\,dt 
	- \jjntQT\rho^{-1}L^\star\phbar\lambda\,dx\,dt
	 = \tilde l(\phbar), \quad \forall \phbar\in \widetilde\Phi,
\end{alignedat}
\end{equation}
	or equivalently, since the control is given by $v:=\rho_0^{-2}\ph\,1_{\omega}$, 
\begin{equation}\nonumber
\begin{alignedat}{2}
	\iint_{q_T}v \,\phbar\, dx\,dt 
	 - \iint_{Q_T}L^\star\phbar(\rho^{-1}\lambda)\,dx\,dt
	 = \tilde l(\phbar), \quad \forall \phbar\in\widetilde \Phi,
\end{alignedat}
\end{equation}
This means that $\rho^{-1}\lambda\in  L^2(Q_T)$ is solution of the heat equation with source term $v \,1_{\omega}$ in the transposition sense and such that $(\rho^{-1}\lambda)(\cdot,T)=0$. 
	Since $y_0\in L^2(\Om)$ and $v\in L^2(q_T)$,~ $\rho^{-1}\lambda$ must coincide with the unique weak solution 
	to~\eqref{eq:heat} ($y=\rho^{-1} \lambda$) and, in particular, $y(\cdot,T)=0$. 
\end{proof}	
	
\begin{remark}
The well-posedness of the mixed formulation \eqref{eq:mf3}, precisely the inf-sup property \eqref{eq:infsup3}, is open in the case where the weight $\rho$ is simply in $\mathcal{R}$ $($instead of $\mathcal{R}\cap L^{\infty}(Q_T))$: in this case, the weight $\rho$ may blow up at $t\to T^{-}$.  In order to get \eqref{eq:infsup3}, it suffices to prove that the function $\psi:=\rho_0^{-1}\ph$ solution of the boundary value problem
\begin{equation}
\rho^{-1}L^{\star}(\rho_0 \psi) =-\lambda^0 \,\, \textrm{in}\,\, Q_T, \quad \psi=0\,\,\textrm{on}\,\,\Sigma_T, \quad \psi(\cdot,T)=0 \,\, \textrm{in}\,\, \Omega \nonumber
\end{equation}
for any $\lambda_0\in L^2(Q_T)$ satisfies the following estimate for some positive constant $C$  
\begin{equation}
\Vert \psi \Vert_{L^2(q_T)}   \leq  C \Vert \rho^{-1}L^{\star}(\rho_0 \psi)\Vert_{L^2(Q_T)}. \nonumber
\end{equation}
In the cases of interest for which the weights $\rho_0$ and $\rho$ blow up at $t\to T^{-}$ (for instance given by $\rho_0^c$ and $\rho^c$), this estimates is open and does not seem to be a consequence of the estimate \eqref{crucial_estimate}.  
\end{remark}

Let us now comment the introduction of the weight $\rho$. The solution $\ph$ of the mixed formulation \eqref{eq:mf3} belongs to $\widetilde W_{\rho_0,\rho}$ and therefore does not depend on the weight $\rho$ (recall that $\rho$ is strictly positive); this is in agreement with the fact that $\rho$ does not appear in the original formulation formulation \eqref{eq:Jb}. Therefore, this weight may be seen as a parameter to improve some specific properties of the mixed formulation, specifically at the numerical level. Precisely, in the limit case $\eps=0$, we recall that the trace $\ph|_{t=T}$ of the solution does not belong to $L^2(\Omega)$ but to a much larger and singular space. Very likely, a similar behavior occurs for the function $L^{\star}\ph$ near $\Om\times \{T\}$ so that the constraint $L^{\star}\ph=0$ in $L^2(Q_T)$
introduced in Section \ref{penalized_one} is too ``strong" and must be replaced at the limit in $\eps$ by the relaxed one $\rho^{-1}L^{\star}\ph =0$ in $L^2(Q_T)$ with $\rho^{-1}$ ``small" near $\Om\times \{T\}$. 
Remark that this is actually the effect and the role of the Carleman type weights $\rho^c$ defined by (\ref{weights}) and initially introduced in \cite{FursikovImanuvilov}.

As in Section \ref{penalized_one}, it is convenient to ``augment" the Lagrangian and consider instead the Lagrangian $\mathcal{L}_{r}$ defined for any $r>0$ by
\[
	\left\{
		\begin{aligned}
			& \mathcal{L}_{r}(\ph,\lambda):= \frac{1}{2}\tilde a_r(\ph,\ph) + \tilde b(\ph,\lambda) - \tilde l(\ph), \\ 
			& \tilde a_r(\ph,\ph):=\tilde a(\ph,\ph) +r \jjntQT \vert \rho^{-1} L^{\star}\ph\vert^2 \, dx\,dt.
		\end{aligned}   
	\right.
\]

Finally, similarly to Lemma \ref{propA} and Proposition \ref{prop_equiv_dual}, we have the following result.

	Let $\rho_0\in \mathcal{R}$ and $\rho\in \mathcal{R}\cap L^{\infty}(Q_T)$ 
\begin{proposition}\label{prop_equiv_rho}
	For any $r>0$, let $\rho_0\in \mathcal{R}$ and $\rho\in \mathcal{R}\cap L^{\infty}(Q_T)$ verifying \eqref{hypK}. 
	Let us define the linear operator $\mathcal{A}_r$ from $L^2(Q_T)$ into $L^2(Q_T)$ by 
\begin{equation}
	\mathcal{A}_r \lambda : =\rho^{-1}L^{\star}\ph, \quad\forall \lambda\in L^2(Q_T), \nonumber
\end{equation}
	where $\ph \in \widetilde\Phi_{\rho_0,\rho}$ is the unique solution to
\begin{equation}
	a_r(\ph, \phbar) = -b(\phbar, \lambda), \quad \forall \phbar \in \widetilde\Phi_{\rho_0,\rho}.    \nonumber
\end{equation}
	$\mathcal{A}_r$ is a strongly elliptic, symmetric isomorphism from $L^2(Q_T)$ into $L^2(Q_T)$. 
	Let  $\hat{J}_r^{\star\star}$ be the functional defined by 
\[
	\hat{J}_r^{\star\star}: L^2(Q_T)\mapsto L^2(Q_T),\quad\hat{J}_r^{\star\star}(\lambda) := \frac{1}{2} \jjntQT (\mathcal{A}_r \lambda) \,\lambda\, dx \,dt - \tilde b(\ph^0, \lambda).
\]
	where $\ph^0\in \widetilde\Phi_{\rho_0,\rho}$ is the unique solution of 
\[
	\tilde a_r(\ph^0,\phbar)= \tilde l(\phbar), \quad \forall \phbar\in \widetilde\Phi_{\rho_0,\rho}.
\]
	The following equality holds : 
\[
	\sup_{\lambda\in L^2(Q_T)}\inf_{\ph\in \widetilde\Phi_{\rho_0,\rho}} \mathcal{L}_r(\ph,\lambda) 
	= - \inf_{\lambda\in L^2(Q_T)} \hat{J}_r^{\star\star}(\lambda)~~+~ \mathcal{L}_r(\ph^0,0).
\]
\end{proposition}


\section{Numerical approximation and experiments} \label{numerics}

\subsection{Discretization of the mixed formulation (\ref{eq:mf1})} \label{discretization_mf1}

We now turn to the discretization of the mixed formulation (\ref{eq:mf1}) assuming $r>0$. 
Let then $\Phi_{\eps,h}$ and $M_{\eps,h}$ be two finite dimensional spaces parametrized by the variable $h$ such that, for any $\varepsilon>0$,  
\[
\Phi_{\eps,h}\subset \Phi_{\eps}, \quad M_{\eps,h}\subset L^2(Q_T), \quad \forall h>0.
\]
Then, we can introduce the following approximated problems : 
find $(\ph_h,\lambda_h)\in \Phi_{\eps,h}\times M_{\eps,h}$ solution of 
\begin{equation} \label{eq:mf1h}
\left\{
\begin{array}{rcll}
\noalign{\smallskip} a_{\eps,r}(\ph_h, \overline{\ph}_h) + b(\overline{\ph}_h, \lambda_h) & = & l(\overline{\ph}_h), & \qquad \forall \overline{\ph}_h \in \Phi_{\eps,h} \\
\noalign{\smallskip} b(\ph_h, \overline{\lambda}_h) & = & 0, & \qquad \forall \overline{\lambda}_h \in M_{\eps,h}.
\end{array}
\right.
\end{equation}

The well-posedness of this mixed formulation is again a consequence of two properties : the coercivity of the bilinear form $a_{\eps,r}$ on the subset 
$\mathcal{N}_h(b)=\{\ph_h\in \Phi_{\eps,h}; b(\ph_h,\lambda_h)=0\quad \forall \lambda_h\in M_{\eps,h}\}$. Actually, from the relation 
$$
a_{\eps,r}(\ph,\ph)\geq C_{r,\eta}\Vert \ph\Vert^2_{\Phi_\eps}, \quad \forall \ph\in \Phi_{\eps},
$$
where~$C_{r,\eta}=\min\{1,r/\eta\}$,~the form  $a_{\eps,r}$ is coercive on the full space $\Phi_\eps$, and so \textit{a fortiori} on $\mathcal{N}_h(b)\subset \Phi_{\eps,h}\subset \Phi_\eps$.
The second property is a discrete inf-sup condition : there exists $\delta_h>0$ such that 
\begin{equation} \label{infsupdiscret}
\inf_{\lambda_h \in M_{\eps,h}}\sup_{\ph_h\in \Phi_{\eps,h}} \frac{b(\ph_h,\lambda_h)}{\Vert \ph_h\Vert_{\Phi_{\eps,h}} \Vert \lambda_h\Vert_{M_{\eps,h}}}\geq \delta_h.
\end{equation}
For any fixed $h$, the spaces $M_{\eps,h}$ and $\Phi_{\eps,h}$ are of finite dimension so that the infimum and supremum in (\ref{infsupdiscret}) are reached: moreover, from the property of the bilinear form $a_{\eps,r}$, it is standard to prove that $\delta_h$ is strictly positive (see Section \ref{infsuptest}). Consequently, for any fixed $h>0$, there exists a unique couple $(\ph_h,\lambda_h)$ solution of (\ref{eq:mf1h}). On the other hand, the property $\inf_h \delta_h>0$ is in general difficult to prove and depends strongly on the choice made for the approximated spaces $M_{\eps,h}$ and $\Phi_{\eps,h}$. We shall analyze numerically this property in Section \ref{infsuptest}.  

\begin{remark}\label{condr}
For $r=0$, the discrete mixed formulation (\ref{eq:mf1h}) is not well-posed over $\Phi_{\eps,h}\times M_{\eps,h}$ because the bilinear form $a_{\eps,r=0}$ is not coercive over the discrete kernel of $b$: the equality $b(\lambda_h,\ph_h)=0$ for all $\lambda_h \in M_{\eps,h}$ does not imply that $L^{\star}\ph_h$ vanishes. Therefore, the term $r \Vert L^{\star}\ph_h\Vert^2_{L^2(Q_T)}$ may be understood as a numerical stabilization term: for any $h>0$, it ensures the uniform coercivity of the form $a_{\eps,r}$ (and so the well-posedness) and vanishes at the limit in $h$. We also emphasize that this term is not a regularization term as it does not add any regularity to the solution $\ph_h$.
\end	{remark}

As in \cite{NC-EFC-AM}, the finite dimensional and conformal space $\Phi_{\eps,h}$ must be chosen such that $L^{\star}\ph_h$ belongs to $L^2(Q_T)$ for any $\ph_h\in \Phi_{\eps,h}$. This is guaranteed as soon as $\ph_h$ possesses second-order derivatives in $L^2_{loc}(Q_T)$. Any conformal approximation based on standard triangulation of $Q_T$ achieves this sufficient property as soon as it is generated by spaces of functions continuously differentiable with respect to the variable $x$ and spaces of continuous functions with respect to the variable $t$.

We introduce a triangulation $\mathcal{T}_h$ such that $\overline{Q_T}=\cup_{K\in \mathcal{T}_h} K$ and we assume that $\{\mathcal{T}_h\}_{h>0}$ is a regular family. Then, we introduce the space $\Phi_{\eps,h}$ as follows : 
\[
\Phi_{\eps,h}=\{\ph_h \in C^1(\overline{Q_T}):  \ph_h\vert_K\in \mathbb{P}(K) \quad \forall K\in \mathcal{T}_h, \,\, \ph_h=0 \,\,\textrm{on}\,\,\Sigma_T\}
\]
where $\mathbb{P}(K)$ denotes an appropriate space of polynomial functions in $x$ and $t$. In this work, we consider for $\mathbb{P}(K)$ the so-called \textit{Bogner-Fox-Schmit} (BFS for short) $C^1$-element defined for rectangles. 

In the one dimensional setting considered in the sequel, it involves $16$ degrees of freedom, namely the values of $\ph_h, \ph_{h,x}, \ph_{h,t}, \ph_{h,xt}$ on the four vertices of each rectangle $K$. 
Therefore $\mathbb{P}(K)= \mathbb{P}_{3,x}\otimes \mathbb{P}_{3,t}$ where $\mathbb{P}_{r,\xi}$ is by definition the space of polynomial functions of order $r$ in the variable $\xi$. We refer to \cite{ciarletfem} page 76.

We also define the finite dimensional space 
$$
M_{\eps,h}=\{\lambda_h\in C^0(\overline{Q_T}):  \lambda_h\vert_K\in \mathbb{Q}(K) \quad \forall K\in \mathcal{T}_h\},
$$
where $\mathbb{Q}(K)$ denotes the space of affine functions both in $x$ and $t$ on the element $K$. 

Again, in the one dimensional setting, for rectangle, we simply have $\mathbb{Q}(K)= \mathbb{P}_{1,x}\otimes \mathbb{P}_{1,t}$. 

We also mention that the approximation is conformal : for any $h>0$, we have $\Phi_{\eps,h}\subset \Phi_\eps$ and $M_{\eps,h}\subset L^2(Q_T)$. 

Let $n_h=\dim \Phi_{\eps,h}, m_h=\dim M_{\eps,h}$ and let the real matrices $A_{\eps,r,h}\in \mathbb{R}^{n_h,n_h}$, $B_{h}\in \mathbb{R}^{m_h,n_h}$, $J_h\in \mathbb{R}^{m_h,m_h}$ and $L_h\in \mathbb{R}^{n_h}$ be defined by   
\[
\left\{
\begin{aligned}
& a_{\eps,r}(\ph_h,\overline{\ph_h})= <A_{\eps,r,h} \{\ph_h\}, \{\overline{\ph_h}\}>_{\mathbb{R}^{n_h},\mathbb{R}^{n_h}}, \quad \forall \ph_h,\overline{\ph_h}\in \Phi_{\eps,h}, \\
& b(\ph_h,\lambda_h)= <B_{h} \{\ph_h\}, \{\lambda_h\}>_{\mathbb{R}^{m_h},\mathbb{R}^{m_h}}, \quad \forall \ph_h\in \Phi_{\eps,h}, \forall \lambda_h\in M_{\eps,h},\\
& \int\!\!\!\int_{Q_T} \lambda_h\overline{\lambda_h}\,dx\,dt= <J_h \{\lambda_h\}, \{\overline{\lambda_h}\}>_{\mathbb{R}^{m_h},\mathbb{R}^{m_h}},\quad \forall \lambda_h, \overline{\lambda_h}\in M_{\eps,h},\\
& l(\ph_h)=<L_h,\{\ph_h\}> , \quad \forall \ph_h\in \Phi_{\eps,h}
\end{aligned}
\right.
\]
where $\{\ph_h\}\in \mathbb{R}^{n_h}$ denotes the vector associated to $\ph_h$ and $<\cdot,\cdot>_{\mathbb{R}^{n_h},\mathbb{R}^{n_h}}$ the usual scalar product over $\mathbb{R}^{n_h}$. With these notations, Problem (\ref{eq:mf1h}) reads as follows : find $\{\ph_h\}\in \mathbb{R}^{n_h}$ and $\{\lambda_h\}\in \mathbb{R}^{m_h}$ such that 
\begin{equation} \label{matrixmfh}
\left(
\begin{array}{cc}
A_{\eps,r,h} &  B_h^T \\
B_h & 0   
\end{array}
\right)_{\mathbb{R}^{n_h+m_h,n_h+m_h}}
\left(
\begin{array}{c}
\{\ph_h\}   \\
\{\lambda_h\}    
\end{array}
\right)_{\mathbb{R}^{n_h+m_h}}  =
\left(
\begin{array}{c}
L_h \\
0   
\end{array}
\right)_{\mathbb{R}^{n_h+m_h}}.
\end{equation}
The matrix $A_{\eps,r,h}$ as well as the mass matrix $J_h$ are symmetric and positive definite for any $h>0$ and any $r>0$. On the other hand, the matrix of order $m_h+n_h$ in (\ref{matrixmfh}) is symmetric but not positive definite. We use exact integration methods developed in \cite{dunavant} for the evaluation of the coefficients of the matrices. The system (\ref{matrixmfh}) is solved using the direct LU decomposition method. 

Let us also mention that for $r=0$, although the formulation (\ref{eq:mf1}) is well-posed, numerically, the corresponding matrix $A_{\eps,0,h}$ is not invertible in agreement with Remark \ref{condr}. In the sequel, we shall consider strictly positive values for $r$.  

Once an approximation $\ph_h$ is obtained, an approximation $v_{\eps,h}$ of the control $v_{\eps}$ is given by $v_{\eps,h}=\rho_0^{-2}\ph_{\eps,h}\,1_{\omega}$. The corresponding controlled state $y_{\eps,h}$ may be obtained  by solving (\ref{eq:heat}) with standard forward approximation (we refer to \cite{NC-EFC-AM}, Section 4 where this is detailed). Here, since the controlled state is directly given by the multiplier $\lambda$, we simply use $\lambda_h$ as an approximation of $y$ and we do not report here the computation of $y_h$.   

In the sequel, we only report numerical experiments in the one dimensional setting. We use uniform rectangular meshes. Each element is a rectangle of lengths $\Delta x$ and $\Delta t$; $\Delta  x>0$ and $\Delta t>0$ denote as usual the discretization parameters in space and time, respectively. We note 
\[
\vspace{-0.05cm}
	h:=\max\{\textrm{diam}(K), K\in \mathcal{T}_h \}
\]
where $\textrm{diam}(K)$ denotes the diameter of $K$.
	
\subsection{Normalization and discretization of the mixed formulation \eqref{eq:mf3}} \label{discretization_mf3}

The same approximation may be used for the mixed formulation (\ref{eq:mf3}). In particular, we easily check that the finite dimensional spaces 
$M_{\eps,h}$ and $\Phi_{\eps,h}$ (which actually do not depend on $\eps$) are conformal approximation of $L^2(Q_T)$ and $\widetilde{\Phi}_{\rho_0,\rho}$ respectively. 
However, in the limit case $\eps=0$, a normalization of the variable $\ph$, which is singular and takes arbitrarily large amplitude in the neighborhood of 
$\Om\times \{T\}$ is very convenient and suitable in practice. Following \cite{EFC-AM-sema}, we introduce the variable $\psi:=\rho_0^{-1}\ph\in \rho_0^{-1}\widetilde\Phi_{\rho_0,\rho}$
and replace the mixed formulation (\ref{eq:mf3}) by the equivalent one:  find $(\psi,\lambda)\in \rho_0^{-1}\widetilde \Phi_{\rho_0,\rho}\times L^2(Q_T)$ solution of 
\begin{equation} 
\label{eq:mf3bis}
	\left\{
		\begin{array}{rcll}
			\noalign{\smallskip} 
			\hat a(\psi, \psibar) + \hat b(\psibar,\lambda) 
			& = &\hat l(\psibar), & \quad \forall \psibar \in \rho_0^{-1}\widetilde\Phi_{\rho_0,\rho} \\
			\noalign{\smallskip}\hat b(\psi, \lambar) & = & 0, & \quad \forall \lambar \in L^2(Q_T),
		\end{array}
	\right.
\end{equation}
	where
\begin{align}
	\nonumber
	&\hat a :\rho_0^{-1}\widetilde\Phi_{\rho_0,\rho} \times\rho_0^{-1}\widetilde\Phi_{\rho_0,\rho} \to \mathbb{R},  \quad \hat a(\psi, \psibar) 
			   = \jjntqT \psi \,\psibar\, dx\,dt  \\
	\nonumber
	&\hat b:\rho_0^{-1}\widetilde\Phi_{\rho_0,\rho} \times L^2(Q_T)\to \mathbb{R}, \quad \hat b(\psi, \lambda) = -\jjntQT\rho^{-1}L^\star(\rho_0\,\psi) \,\lambda\,dx\,dt\\
	\nonumber
	&\hat l:  \rho_0^{-1}\widetilde\Phi_{\rho_0,\rho} \to \mathbb{R},  \quad\hat  l(\ph) =- (y_0, \rho_0(\cdot,0) \psi(\cdot, 0))_{L^2(\Omega)}.
\end{align}
Well-posedness of this formulation is the consequence of Theorem \ref{th:mf3}. Moreover, the optimal controlled state is still given by $\rho^{-1}\lambda$ while the optimal control is expressed in term of the variable $\psi$ as $v=\rho_0^{-1}\psi \, 1_{\om}$.

The corresponding discretization approximation (augmented with the term \linebreak $r \Vert \rho^{-1}L^{\star}(\rho_0 \psi)\Vert_{L^2(Q_T)}$) reads as follows: find $(\psi_h,\lambda_h)\in \Phi_h\times M_h$ solution of 

\begin{equation} \label{eq:mf3bish}
\left\{
\begin{array}{rcll}
\noalign{\smallskip} \hat a_r(\psi_h, \overline{\psi}_h) + \hat b(\overline{\psi}_h, \lambda_h) & = & \hat l(\overline{\ph}_h), & \qquad \forall \overline{\psi}_h \in \Phi_h \\
\noalign{\smallskip} \hat b(\psi_h, \overline{\lambda}_h) & = & 0, & \qquad \forall \overline{\lambda}_h \in M_h.
\end{array}
\right.
\end{equation}
with 
$$
\begin{aligned}
a_r(\psi_h,\psibar_h):=& a(\psi_h,\psibar_h)+r (\rho^{-1}L^{\star}(\rho_0 \psi_h),\rho^{-1}L^{\star}(\rho_0 \psibar_h))_{L^2(Q_T)}\\
=&(\psi_h,\psibar_h)_{L^2(q_T)}+ r (\rho^{-1}L^{\star}(\rho_0 \psi_h),\rho^{-1}L^{\star}(\rho_0 \psibar_h))_{L^2(Q_T)}.
\end{aligned}
$$
for any $r>0$. 

\begin{remark}
When the weights $\rho_0$ and $\rho$ are chosen in such a way that they are compensated each other in the term $\rho^{-1}L^{\star}(\rho_0 \psi)$, the change of variable has the effect to reduced the amplitude (with respect to the time variable) of the coefficients in the integrals of $\hat a_r$ and $\hat b$, and therefore, at the discrete level, to improve significantly the condition number of square matrix $\hat A_{r,h}$ so that $\hat a_r(\psi_h,\psibar_h)=<\hat A_{r,h} \{\psi_h\},\{\psibar_h\}>_{\mathbb{R}^{n_h},\mathbb{R}^{n_h}}$. In this respect, the change of variable, can be seen as a preconditioner for the mixed formulation (\ref{eq:mf3}).   
\end{remark}
Similarly to (\ref{matrixmfh}), we note the matrix form of (\ref{eq:mf3bish}) as follows : 
\begin{equation} \label{matrixmfh3bis}
\left(
\begin{array}{cc}
\hat A_{r,h} &  \hat B_h^T \\
\hat B_h & 0   
\end{array}
\right)_{\mathbb{R}^{n_h+m_h,n_h+m_h}}
\left(
\begin{array}{c}
\{\psi_h\}   \\
\{\lambda_h\}    
\end{array}
\right)_{\mathbb{R}^{n_h+m_h}}  =
\left(
\begin{array}{c}
\hat L_h \\
0   
\end{array}
\right)_{\mathbb{R}^{n_h+m_h}},
\end{equation}
where $\hat B_h$ is the matrix so that~$\hat b(\psi_h,\lambda_h) = <\hat B_{h} \{\psi_h\}, \{\lambda_h\}>_{\mathbb{R}^{m_h},\mathbb{R}^{m_h}}$
and $\hat L_h$ is the matrix so that~$\hat l(\psi_h)=<\hat L_h,\{\psi_h\}> $.


\subsection{The discrete inf-sup test} \label{infsuptest}

	Before giving and discussing some numerical experiments, we first test numerically the discrete inf-sup condition \eqref{infsupdiscret}. 
	Taking $\eta=r>0$ so that $a_{\eps,r}(\ph,\phbar)= (\ph,\phbar)_{\Phi_\eps}$ exactly for all $\ph,\phbar\in \Phi_{\eps}$, it is readily 
	seen (see for instance \cite{chapelle}) that the discrete inf-sup constant satisfies  
\begin{equation}\label{eigenvalue}
	\delta_{\eps,r,h} = \inf\biggl\{\sqrt{\delta}:   B_h A_{\eps,r,h}^{-1} B_h^T  \{\lambda_h\} = \delta \,J_h \{\lambda_h\}, 
	\quad \forall\, \{\lambda_h\}\in \mathbb{R}^{m_h}\setminus\{0\}\biggr\}. 
\end{equation}
	The matrix $B_h A_{\eps,r,h}^{-1} B_h^T$ enjoys the same properties than the matrix $A_{\eps,r,h}$: it is symmetric and positive definite 
	so that the scalar $\delta_{\eps,h}$ defined in term of the (generalized) eigenvalue problem (\ref{eigenvalue}) is strictly positive. 
	This eigenvalue problem is solved using the power iteration algorithm (assuming that the lowest eigenvalue is simple):  for any 
	$\{v^0_h\}\in \mathbb{R}^{m_h}$ such that $\Vert \{v_h^0\}\Vert_2=1$, compute for any $n\geq 0$, $\{\ph_h^n\}\in \mathbb{R}^{n_h},
	\{\lambda_h^n\}\in \mathbb{R}^{m_h}$ and $\{v_h^{n+1}\}\in \mathbb{R}^{m_h}$ iteratively as follows :  
\[
	\left\{
		\begin{aligned}
			& A_{\eps,r,h}\{\ph_h^n\} + B_h^T \{\lambda_h^n\}=0 \\
			& B_h \{\ph_h^n\} = - J_h \{v_h^n\}
		\end{aligned}
	\right.\quad, 
\quad 
			\{v_h^{n+1}\} =\frac{ \{\lambda_h^n \} }{  \Vert \{ \lambda_h^n\}\Vert_2}.    
\]
	The scalar $\delta_{\eps,r,h}$ defined by (\ref{eigenvalue}) is then given by : $\delta_{\eps,r,h} = \lim_{n\to\infty} (\Vert \{\lambda_h^n\}\Vert_2)^{-1/2}$.

	We now give some numerical values of $\delta_{\eps,r,h}$ with respect to $h$ for the $C^1$-finite element introduced 
	in Section \ref{discretization_mf1}. 
	
	We consider the one dimensional case for which $\Omega=(0,1)$ and take for simplicity $c:=1/10$ and $d:=0$. Values of the diffusion $c$ and of the potential $d$ do not affect qualitatively the results.  
	
In the spirit of the previous work \cite{EFC-AM-sema}, we consider the following choice for the weight $\rho_0\in \mathcal{R}$: 
\begin{equation}
\rho_0(x,t):= (T-t)^{3/2} \, \exp\biggl(\frac{K_1}{(T-t)}\biggr),  \quad  (x,t)\in Q_T, \qquad K_1:=\frac{3}{4}  \label{weight_rho0}
\end{equation}	
so that $\rho_0$ blows exponentially as $t\to T^{-}$. This allows a smooth behavior of the corresponding control $v:=\rho_0^{-2} \ph\, 1_{\omega}$. Let us insist however that the mixed formulation is well-posed for any weight $\rho_0\in \mathcal{R}$, in particular $\rho_0:=1$ (leading to the control of minimal $L^2$-norm and for which we refer to \cite{MunchZuazuaRemedies}). $\rho_0$ is independent of the variable $x$ for simplicity.

We consider the following data  $\omega=(0.2,0.5)$, $T=1/2$, and $\Omega=(0,1)$. Tables \ref{tab:deltahmf1_r001_0205}, \ref{tab:deltahmf1_r1_0205} and \ref{tab:deltahmf1_r100_0205} provides the values of $\delta_{\eps,r,h}$ with respect to $h$ and $\eps$ for $r=10^{-2}, 1$ and $r=10^2$, respectively. In view of the definition, we check that $\delta_{\eps,r,h}$ increases as $r\to 0$ and $\eps\to 0$. We also observe, that for $r$ large enough (see Tables \ref{tab:deltahmf1_r1_0205} and \ref{tab:deltahmf1_r100_0205}), the value of the inf-sup constant is almost constant with respect to $\eps$ and behaves like 
\begin{equation}
\delta_{\eps,r,h} \approx  C_{\eps,r,h} \times r^{-1/2}      \label{behavior_deltah}   
\end{equation}
for some constant $C_{\eps,r,h}\in (0,1)$. 
More importantly, we observe that for any $r$ and $\eps$, the value of $\delta_{\eps,r,h}$ is bounded by below uniformly with respect to the discretization parameter $h$. The same behavior is observed for other discretizations such that $\Delta t \neq \Delta x$, other supports $\omega$ and other choices for the weight $\rho_0$ (in particular $\rho_0:=1$). 

Consequently, we may conclude that the finite approximation we have used do "pass" the discrete inf-sup test. It is interesting to note that this is in contrast with the situation for the wave equation for which the parameter $r$ have to be adjusted carefully with respect to $h$; we refer to \cite{NC-AM-mixedwave}.  
Moreover, as it is usual in mixed finite element theory, such a property together with the uniform coercivity of form $a_{\eps,r}$ then implies the convergence of the approximation sequence $(\ph_h,\lambda_h)$ solution of (\ref{eq:mf1h}). 

\begin{table}[http]
	\centering
		\begin{tabular}{|c|cccc|}
			\hline
			$h$  			       & $7.07\times 10^{-2}$ & $3.53\times 10^{-2}$ & $1.76\times 10^{-2}$ & $8.83\times 10^{-3}$
			\tabularnewline
			\hline
			$\eps=10^{-2}$ 	        	& $8.358$ 			& $8.373$ 			& $8.381$ 			& $8.386$ 	  
			\tabularnewline
			$\eps=10^{-4}$	 	& $9.183$		 	& $9.213$ 			& $9.229$ 		        & $9.237$  		         
			\tabularnewline
			$\eps=10^{-8}$ 	        & $9.263$			& $9.318$	 		& $9.354$ 			& $9.383$ 
			\tabularnewline  
			\hline
		\end{tabular}
	\caption[$\delta_{\eps,r,h}$  w.r.t. $\eps$ and $h$ ; $r=10^{-2}$]{$\delta_{\eps,r,h}$  w.r.t. $\eps$ and $h$ ; $r=10^{-2}$ ; $\Omega=(0,1)$, $\omega=(0.2,0.5)$, $T=1/2$.}
	\label{tab:deltahmf1_r001_0205}
\end{table}

\begin{table}[http]
	\centering
		\begin{tabular}{|c|cccc|}
			\hline
			$h$  			       & $7.07\times 10^{-2}$ & $3.53\times 10^{-2}$ & $1.76\times 10^{-2}$ & $8.83\times 10^{-3}$
			\tabularnewline
			\hline
			$\eps=10^{-2}$ 	        	& $9.933\times 10^{-1}$ 			& $9.938\times 10^{-1}$ 			& $9.940\times 10^{-1}$ 	 &      $9.941\times 10^{-1}$ 	  
			\tabularnewline
			$\eps=10^{-4}$	 	& $9.933\times 10^{-1}$		 	& $9.938\times 10^{-1}$ 			& $9.941\times 10^{-1}$ 		        & $9.942\times 10^{-1}$  		         
			\tabularnewline
			$\eps=10^{-8}$ 	        & $9.933\times 10^{-1}$			& $9.938\times 10^{-1}$	 		& $9.941\times 10^{-1}$ 			& $9.942\times 10^{-1}$ 
			\tabularnewline  
			\hline
		\end{tabular}
	\caption[$\delta_{\eps,r,h}$  w.r.t. $\eps$ and $h$ ; $r=1$]{$\delta_{\eps,r,h}$  w.r.t. $\eps$ and $h$ ; $r=1$ ; $\Omega=(0,1)$, $\omega=(0.2,0.5)$, $T=1/2$.}
	\label{tab:deltahmf1_r1_0205}
\end{table}

\begin{table}[http]
	\centering
		\begin{tabular}{|c|cccc|}
			\hline
			$h$  			       & $7.07\times 10^{-2}$ & $3.53\times 10^{-2}$ & $1.76\times 10^{-2}$ & $8.83\times 10^{-3}$
			\tabularnewline
			\hline
			$\eps=10^{-2}$ 	        	& $9.933\times 10^{-2}$ 			& $9.939\times 10^{-2}$ 			& $9.940\times 10^{-2}$ 	 &      $9.941\times 10^{-2}$ 	  
			\tabularnewline
			$\eps=10^{-4}$	 	& $9.933\times 10^{-2}$		 	& $9.939\times 10^{-2}$ 			& $9.941\times 10^{-2}$ 		        & $9.942\times 10^{-2}$  		         
			\tabularnewline
			$\eps=10^{-8}$ 	        & $9.933\times 10^{-2}$			& $9.939\times 10^{-2}$	 		& $9.941\times 10^{-2}$ 			& $9.942\times 10^{-2}$ 
			\tabularnewline  
			\hline
		\end{tabular}
	\caption[$\delta_{\eps,r,h}$  w.r.t. $\eps$ and $h$ ; $r=10^2$]{$\delta_{\eps,r,h}$  w.r.t. $\eps$ and $h$ ; $r=10^2$ ; $\Omega=(0,1)$, $\omega=(0.2,0.5)$, $T=1/2$.}
	\label{tab:deltahmf1_r100_0205}
\end{table}

Similarly, Table \ref{tab:deltah_mf3} displays the discrete inf-sup constant corresponding to the limit case of the mixed formulation (\ref{eq:mf3bis}): 
\begin{equation}
\nonumber
\begin{aligned}
& \inf_{\lambda_h \in  \widehat M_h} \sup_{\psi_h\in \widehat \Phi_h} \frac{\hat b(\psi_h, \lambda_h)}{ \|\lambda_h\|_{L^2(Q_T)}\|\psi_h\|_{\rho_0^{-1}\widetilde\Phi_{\rho_0,\rho} }
			} \\
			 & =\inf_{\lambda_h \in  \widehat M_h} \sup_{\psi_h\in \widehat \Phi_h}   \frac{\jjntqT \lambda_h \,\rho^{-1}L^{\star} (\rho_0 \psi_h) \,dx\,dt}{ \Vert \lambda_h\Vert_{L^2(Q_T)} (\Vert \psi_h\Vert^2_{L^2(q_T)}+ r \Vert \rho^{-1} L^{\star}(\rho_0\,\psi_h) \Vert^2_{L^2(Q_T)}  )^{1/2} }.
\end{aligned}
\end{equation}
We take here a weight $\rho$ independent of the variable $x$ given by 
\begin{equation}
\rho(x,t):= \exp\biggl(\frac{K_1}{(T-t)}\biggr), \quad (x,t)\in Q_T,  \quad K_1:=\frac{3}{4}. \label{weight_rho}
\end{equation}
Again, for the limit case, the value given in the Table suggest a similar behavior observed for $\eps>0$: the constant is uniformly bounded by below with respect to the parameter $h$ and behaves like $r^{-1/2}$ for $r$ large enough (up to $1$).  Remark that, due to the introduction of the weight $\rho\neq 1$, the inf-sup constants given by Table \ref{tab:deltah_mf3} are not the limit (as $\eps\to 0$) of the previous Tables.

\begin{table}[http]
	\centering
		\begin{tabular}{|c|ccccc|}
			\hline
			$h$  			       & $7.07\times 10^{-2}$ & $3.53\times 10^{-2}$ & $1.76\times 10^{-2}$ & $8.83\times 10^{-3}$ & $4.41\times 10^{-3}$
			\tabularnewline
			\hline
			$r=10^2$ 	        		& $6.9\times 10^{-2}$ 	& $6.91\times 10^{-2}$ 	& $7.06\times 10^{-2}$ 	& $8.08\times 10^{-2}$  & $9.52\times 10^{-2}$	  
			\tabularnewline
			$r=1$	 		& $6.89\times 10^{-1}$		 	& $6.91\times 10^{-1}$ 			& $6.96\times 10^{-1}$ 		        & $7.94\times 10^{-1}$  	 & 	$8.66\times 10^{-1}$         
			\tabularnewline
			$r=10^{-2}$ 	        & $1.944$			& $1.922$	 		& $1.845$ 			& $1.775$  & $1.731$
			 \tabularnewline	
			\hline
		\end{tabular}
	\caption[$\eps=0$;   $\delta_{r,h}$  w.r.t. $r$ and $h$]{$\eps=0$;   $\delta_{r,h}$  w.r.t. $r$ and $h$; $\Omega=(0,1)$, $\omega=(0.2,0.5)$, $T=1/2$.}
	\label{tab:deltah_mf3}
\end{table}


\subsection{Numerical experiments for the mixed formulation (\ref{eq:mf1})} 
\label{eq:mf1_numer} 

We report in this section experiments for the mixed formulation (\ref{eq:mf1}) and for simplicity we consider only the one dimensional case: $\Omega=(0,1)$ and $T=1/2$.

Let us first remark that in general explicit solutions $(\ph_{\eps},\lambda_\eps)$ of (\ref{eq:mf1}) are not available. However, when the coefficient $c$ and $d$ are constant, we may obtain a semi-explicit representation (using Fourier decomposition) of the minimizer $\ph_{\eps,T}$ of the conjugate functional $J^{\star}_{\eps}$ (see (\ref{eq:min_eps_adj})), and consequently of the corresponding adjoint variable $\ph_{\eps}$, the control of weighted minimal square integrable norm $v=\rho_0^{-2}\ph_\eps\,1_{\omega}$ and finally the controlled state $y_\eps$ solution of 
\eqref{eq:heat}. In practice, the obtention of the Fourier representation amounts to solve a symmetric linear system. We refer to the Appendix for the details. 

Such representation allows to evaluate precisely the distance of the exact solution  $(\ph_{\eps},\lambda_\eps)$ from the approximation $(\ph_h,\lambda_h)$ with respect to $h$ and validate the convergence of the approximation with respect to $h$. 

As for the initial data, we first simply take the first mode of the Laplacian, that is, 
$y_0(x)=\sin(\pi x)$, $x\in (0,1)$. In view of the regularization property of the heat equation, the regularity of the initial data has a very restricted effect on the optimal control and the robustness of the method. 

We take $c(x):=10^{-1}$, $d(x,t):=0$ and recall that in the uncontrolled case ($\omega=\emptyset$), these data leads to $\Vert y(\cdot,T)\Vert_{L^2(0,1)}=\sqrt{1/2}e^{-\pi^2 c T}\approx 4.31\times 10^{-1}$. Finally, we take $\omega=(0.2,0.5)$.

For $r=1$, Tables \ref{tab:ex_r1_e02_0205}, \ref{tab:ex_r1_e04_0205}  and \ref{tab:ex_r1_e08_0205} report some norms with respect to $h$ for $\eps=10^{-2}$, $10^{-4}$ and $\eps=10^{-8}$ respectively. The cases $r=10^2$ and $r=10^{-2}$  are reported in the Appendix, in Tables \ref{tab:ex_r100_e02_0205}, \ref{tab:ex_r100_e04_0205}, \ref{tab:ex_r100_e08_0205} and  \ref{tab:ex_r001_e02_0205}, \ref{tab:ex_r001_e04_0205}, \ref{tab:ex_r001_e08_0205} respectively. In the Tables, $\ph_{\eps}$ and $y_{\eps}$ denotes the unique solution of (\ref{eq:mf1}) given by (\ref{ph_fourier}) and 
(\ref{y_fourier}). In the Tables, $\kappa_{\eps,h}$ denotes the condition number associated to \eqref{matrixmfh}, independent of the initial data $y_0$ \footnote{The condition number 
	$\kappa(\mathcal{M}_h)$ of any square matrix $\mathcal{M}_h$ is defined by $\kappa(\mathcal{M}_h)=\vert\vert\vert 
	\mathcal{M}_h\vert\vert\vert_2  \vert\vert\vert \mathcal{M}_h^{-1}\vert\vert\vert_2$ where the norm $\vert\vert\vert \mathcal{M}_h
	\vert\vert\vert_2$ stands for the largest singular value of $\mathcal{M}_h$.}. 
	
We first check that the $L^2$-norm $\Vert \lambda_{\eps,h}(\cdot,T)\Vert_{L^2(0,1)}$ of the final state  is of the order of $\sqrt{\eps}$ and that the condition number $\kappa_{\eps,h}$ behave polynomially with respect to $h$; on the other hand, we observe a low variation of $\kappa_{\eps,h}$ with respect to $\eps$; $\kappa_{\eps,h}\approx O(h^{5.9})$ for $\eps=10^{-2}$ and $\kappa_{\eps,h}\approx O(h^{7.3})$ for $\eps=10^{-8}$.

Then, we check the convergence as $h$ tends to zero of the approximations $(v_{\eps,h},\lambda_{\eps,h})$ toward the optimal pair $(v_{\eps},y_{\eps})$ in $L^2(q_T)\times L^2(Q_T)$ for any values of  $\eps$ and $r$. 

More precisely, for large enough value of $\eps$ (here $\eps=10^{-2}$), we observe a quasi linear rate of convergence for both   $\frac{\|\rho_0 (v_{\eps} - v_{\eps,h})\|_{L^2(q_T)}}{\Vert \rho_0 v_{\eps} \Vert_{L^2(q_T)}}$ and $\frac{\|y_{\eps} - \lambda_{\eps,h}\|_{L^2(Q_T)}}{\Vert y_{\eps} \Vert_{L^2(Q_T)}}$ with respect to $h$, independent of the value of the parameter $r$. We refer to Figure \ref{contsolepsem2}. 
For small values of $\eps$, we observe a reduced convergence both for the control and the state (see Figure \ref{contsolepsem4} for $\eps=10^{-4}$ and Figure \ref{contsolepsem8} for $\eps=10^{-8}$).
We recall that as $\eps$ tends to zero, the space $\Phi_{\eps}$ degenerates into a much larger space and $\ph_{\eps}$ highly oscillates near $T$. Remark also that for $\eps=10^{-8}$, the constraint  
$L^{\star}\ph_{\eps}=0$ as an $L^2(Q_T)$ function is badly represented: this is due to the loss of regularity on the variable $\ph_{\eps}$ (in the neighborhood of $T$) as $\eps\to 0^+$. This does not prevent the convergence of the variable $\ph_{\eps,h}$ for the norm $\Phi_{\eps}$, in particular the control $v_{\eps,h}=\rho^{-2}\ph_{\eps,h}\,1_{\omega}$, and of the variable $\lambda_{\eps,h}$. We will come back to this situation in detail in the section devoted to the limit case $\eps=0$. Moreover, for small value of $\eps$, the parameter $r$ does have an influence; precisely, a low value of $r$ (here $r=10^{-2}$) leads to better relative errors : this is in agreement with the behavior of the inf-sup constant $\delta_{\eps,r,h}$ which increases with $r^{-1/2}$.


\begin{table}[http]
\centering
\scalebox{0.95}{
\begin{tabular}{|c|ccccc|}
\hline 
$h$ & $1.41\times 10^{-1}$ & $7.07\times 10^{-2}$ & $3.53\times 10^{-2}$ & $1.76\times 10^{-2}$ & $8.83\times 10^{-3}$ \tabularnewline
\hline

$m_h+n_h$ & $330$ & $1\ 155$ & $4\ 305$ & $16\ 605$ & $65\ 205$\tabularnewline

$\|L^{\star}\ph_{\eps,h}\|_{L^2(Q_T)}$ & $1.32\times 10^{-1}$ & $3.75\times 10^{-2}$ & $9.66\times 10^{-3}$ & $2.42\times 10^{-3}$ & $7.82\times 10^{-4}$\tabularnewline

$\frac{\|\rho_0 (v_{\eps} - v_{\eps,h})\|_{L^2(q_T)}}{\Vert \rho_0 v_{\eps} \Vert_{L^2(q_T)}}$ &  $1.10\times 10^{-1}$ &  $6.21\times 10^{-2}$ &  $3.29\times 10^{-2}$ &   $1.68\times 10^{-2}$ &  $8.57\times 10^{-3}$ \tabularnewline

$\frac{\|y_{\eps} - \lambda_{\eps,h}\|_{L^2(Q_T)}}{\Vert y_{\eps} \Vert_{L^2(Q_T)}}$ & $5.13\times 10^{-2}$ & $2.84\times 10^{-2}$ & $1.48\times 10^{-2}$ & $7.60\times 10^{-3}$ & $3.89\times 10^{-3}$  \tabularnewline

$\|\lambda_{\eps,h}(\cdot,T)\|_{L^2(0,1)}$ & $1.54\times 10^{-1}$ & $1.61\times 10^{-1}$ & $1.65\times 10^{-1}$ & $1.67\times 10^{-1}$ & $1.68\times 10^{-1}$  \tabularnewline

$\kappa_{\eps,h}$ & $1.52\times 10^{9}$ & $1.10\times 10^{11}$ & $6.80\times 10^{12}$ & $3.83\times 10^{14}$ & $1.96\times 10^{16}$  \tabularnewline

\hline
\end{tabular}}
\caption[Mixed formulation \eqref{eq:mf1} -  $r=1$ and $\eps=10^{-2}$]{Mixed formulation (\ref{eq:mf1}) -  $r=1$ and $\eps=10^{-2}$ with  $\om=(0.2,0.5)$.}
\label{tab:ex_r1_e02_0205}
\end{table}

\begin{table}[http]
\centering
\scalebox{0.95}{
\begin{tabular}{|c|ccccc|}
\hline 
$h$ & $1.41\times 10^{-1}$ & $7.07\times 10^{-2}$ & $3.53\times 10^{-2}$ & $1.76\times 10^{-2}$ & $8.83\times 10^{-3}$ \tabularnewline
\hline

$\|L^{\star}\ph_{\eps,h}\|_{L^2(Q_T)}$ & $1.383$ & $1.471$ & $9.05\times 10^{-1}$ & $2.56\times 10^{-1}$ & $6.54\times 10^{-2}$\tabularnewline

$\frac{\|\rho_0 (v_{\eps} - v_{\eps,h})\|_{L^2(q_T)}}{\Vert \rho_0 v_{\eps} \Vert_{L^2(q_T)}}$ &  $6.72\times 10^{-1}$ &  $3.22\times 10^{-1}$ &  $1.15\times 10^{-1}$ &   $5.49\times 10^{-2}$ &  $2.74\times 10^{-2}$ \tabularnewline

$\frac{\|y_{\eps} - \lambda_{\eps,h}\|_{L^2(Q_T)}}{\Vert y_{\eps} \Vert_{L^2(Q_T)}}$ & $2.73\times 10^{-1}$ & $1.86\times 10^{-1}$ & $5.89\times 10^{-2}$ & $2.51\times 10^{-2}$ & $1.26\times 10^{-2}$  \tabularnewline

$\|\lambda_{\eps,h}(\cdot,T)\|_{L^2(0,1)}$ & $8.50\times 10^{-2}$ & $5.74\times 10^{-2}$ & $3.39\times 10^{-2}$ & $3.11\times 10^{-2}$ & $3.13\times 10^{-2}$  \tabularnewline

$\kappa_{\eps,h}$ & $3.02\times 10^{9}$ & $3.91\times 10^{11}$ & $3.86\times 10^{13}$ & $3.25\times 10^{15}$ & $2.46\times 10^{17}$  \tabularnewline

\hline
\end{tabular}}
\caption[Mixed formulation \eqref{eq:mf1} -  $r=1$ and $\eps=10^{-4}$]{Mixed formulation (\ref{eq:mf1}) -  $r=1$ and $\eps=10^{-4}$ with  $\om=(0.2,0.5)$.}
\label{tab:ex_r1_e04_0205}
\end{table}

\begin{table}[http]
\centering
\scalebox{0.95}{
\begin{tabular}{|c|ccccc|}
\hline 
$h$ & $1.41\times 10^{-1}$ & $7.07\times 10^{-2}$ & $3.53\times 10^{-2}$ & $1.76\times 10^{-2}$ & $8.83\times 10^{-3}$ \tabularnewline
\hline

$\|L^{\star}\ph_{\eps,h}\|_{L^2(Q_T)}$ & $1.48$ & $2.03$ & $2.50$ & $2.52$ & $2.61$\tabularnewline


$\frac{\|\rho_0 (v_{\eps} - v_{\eps,h})\|_{L^2(q_T)}}{\Vert \rho_0 v_{\eps} \Vert_{L^2(q_T)}}$ &  $1.44$ &  $1.01$ &  $7.92\times 10^{-1}$ &   $6.65\times 10^{-1}$ &  $4.89\times 10^{-1}$ \tabularnewline


$\frac{\|y_{\eps} - \lambda_{\eps,h}\|_{L^2(Q_T)}}{\Vert y_{\eps} \Vert_{L^2(Q_T)}}$ & $8.42\times 10^{-1}$ & $8.27\times 10^{-1}$ & $5.73\times 10^{-1}$ & $4.35\times 10^{-1}$ & $2.89\times 10^{-1}$  \tabularnewline


$\|\lambda_{\eps,h}(\cdot,T)\|_{L^2(0,1)}$ & $8.63\times 10^{-2}$ & $6.65\times 10^{-2}$ & $2.39\times 10^{-2}$ & $1.23\times 10^{-2}$ & $4.43\times 10^{-3}$  \tabularnewline

$\kappa_{\eps,h}$ & $3.12\times 10^{9}$ & $4.30\times 10^{11}$ & $6.05\times 10^{13}$ & $1.13\times 10^{16}$ & $1.90\times 10^{18}$  \tabularnewline

\hline
\end{tabular}}
\caption[Mixed formulation \eqref{eq:mf1} -  $r=1$ and $\eps=10^{-8}$]{Mixed formulation (\ref{eq:mf1}) -  $r=1$ and $\eps=10^{-8}$ with  $\om=(0.2,0.5)$.}
\label{tab:ex_r1_e08_0205}
\end{table}

\begin{figure}[http]
\begin{center}
\begin{minipage}{0.45\textwidth}
 \includegraphics[width=\textwidth]{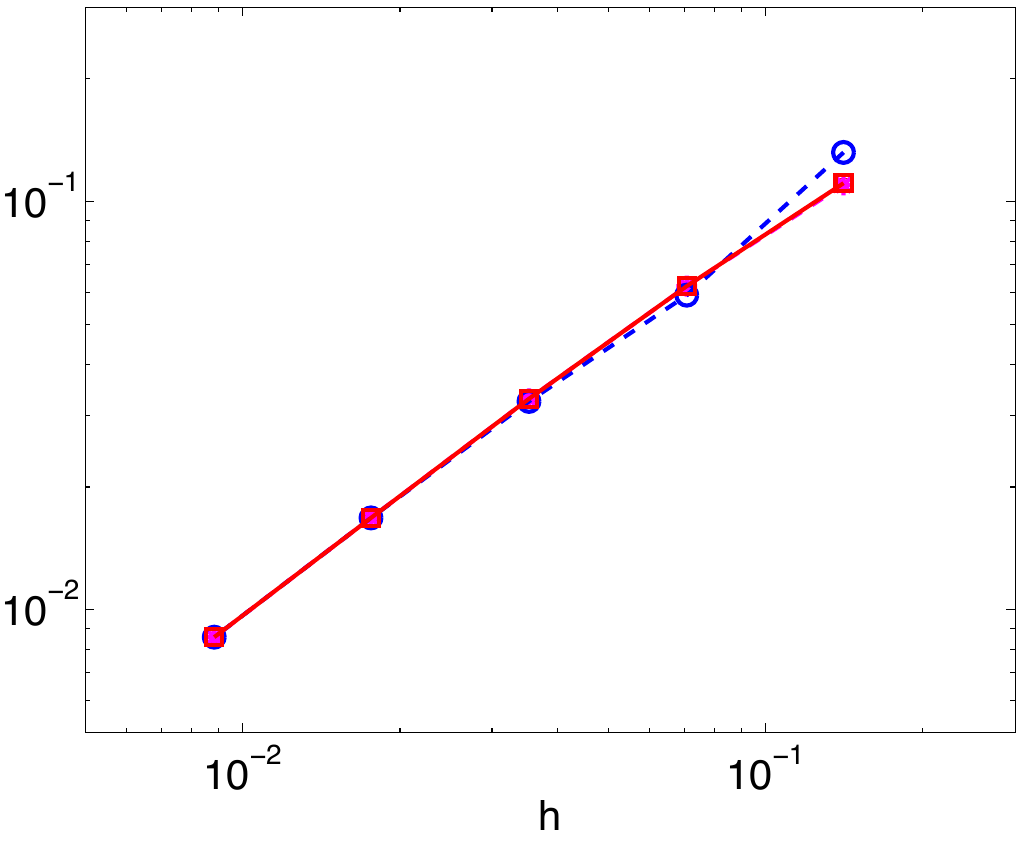}
\end{minipage}
\begin{minipage}{0.45\textwidth}
\hspace*{0.5cm}
\includegraphics[width=\textwidth]{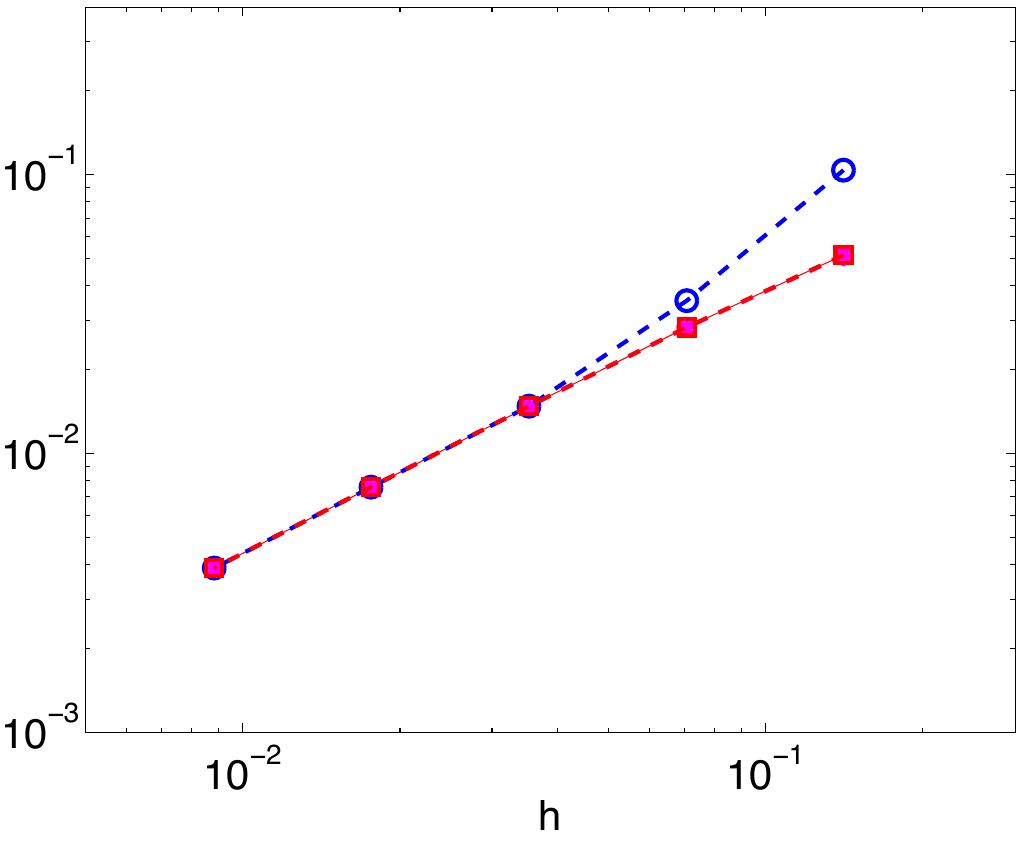}
\end{minipage}
\caption[Convergence for the control and the state, $\eps=10^{-2}$]{$\omega=(0.2,0.5)$; $y_0(x)=\sin(\pi x)$: $\eps=10^{-2}$. ;   $\frac{\|\rho_0 (v_{\eps} - v_{\eps,h})\|_{L^2(q_T)}}{\Vert \rho_0 v_{\eps} \Vert_{L^2(q_T)}}$ ({\bf Left}) and $\frac{\|y_{\eps} - \lambda_{\eps,h}\|_{L^2(Q_T)}}{\Vert y_{\eps} \Vert_{L^2(Q_T)}}$ ({\bf Right}) vs. $h$ for $r=10^{2}$ ({\color{blue} $\circ$}), $r=1.$ ({\color{magenta}$\star$}) and $r=10^{-2}$ ({\color{red}$\square$}).}\label{contsolepsem2}
\end{center}
\end{figure}

\begin{figure}[http]
\begin{center}
\begin{minipage}{0.45\textwidth}
 \includegraphics[width=\textwidth]{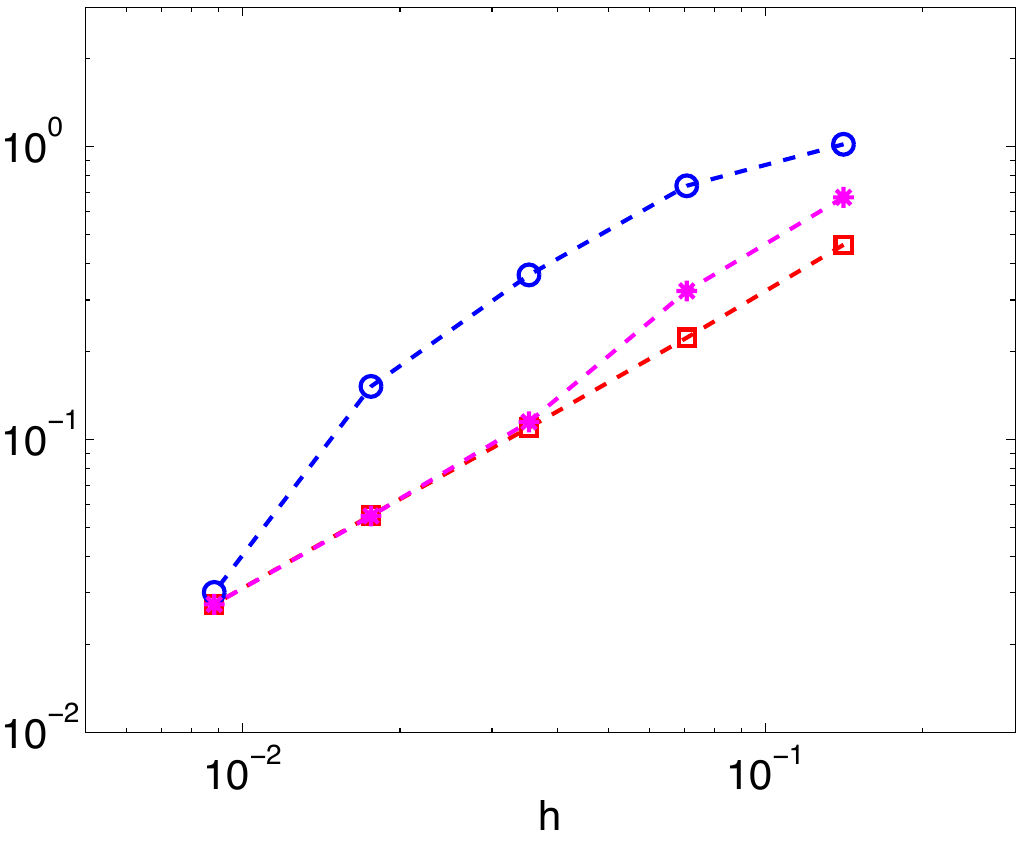}
\end{minipage}
\begin{minipage}{0.45\textwidth}
\hspace*{0.5cm}
\includegraphics[width=\textwidth]{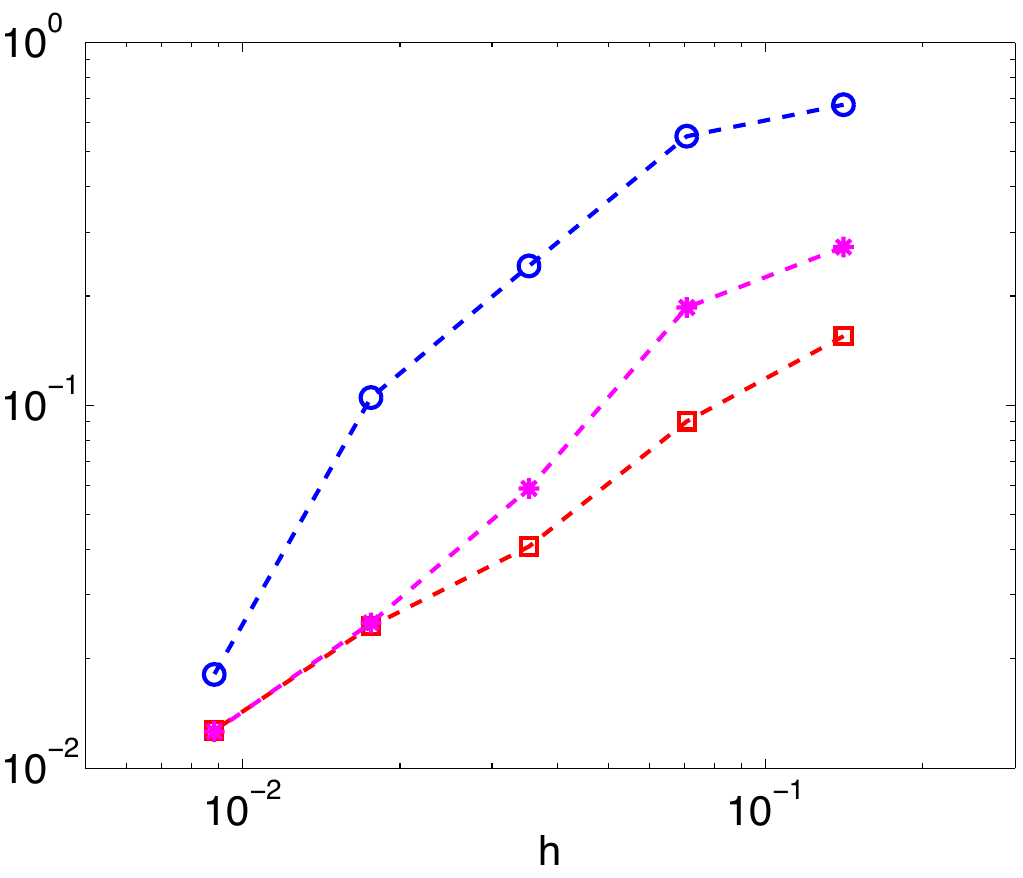}
\end{minipage}
\caption[Convergence for the control and the state, $\eps=10^{-4}$]{$\omega=(0.2,0.5)$; $y_0(x)=\sin(\pi x)$: $\eps=10^{-4}$. ;   $\frac{\|\rho_0 (v_{\eps} - v_{\eps,h})\|_{L^2(q_T)}}{\Vert \rho_0 v_{\eps} \Vert_{L^2(q_T)}}$ ({\bf Left}) and $\frac{\|y_{\eps} - \lambda_{\eps,h}\|_{L^2(Q_T)}}{\Vert y_{\eps} \Vert_{L^2(Q_T)}}$ ({\bf Right}) vs. $h$ for $r=10^{2}$ ({\color{blue} $\circ$}), $r=1.$ ({\color{magenta}$\star$}) and $r=10^{-2}$ ({\color{red}$\square$}).}\label{contsolepsem4}
\end{center}
\end{figure} 

\begin{figure}[http]
\begin{center}
\begin{minipage}{0.45\textwidth}
 \includegraphics[width=\textwidth]{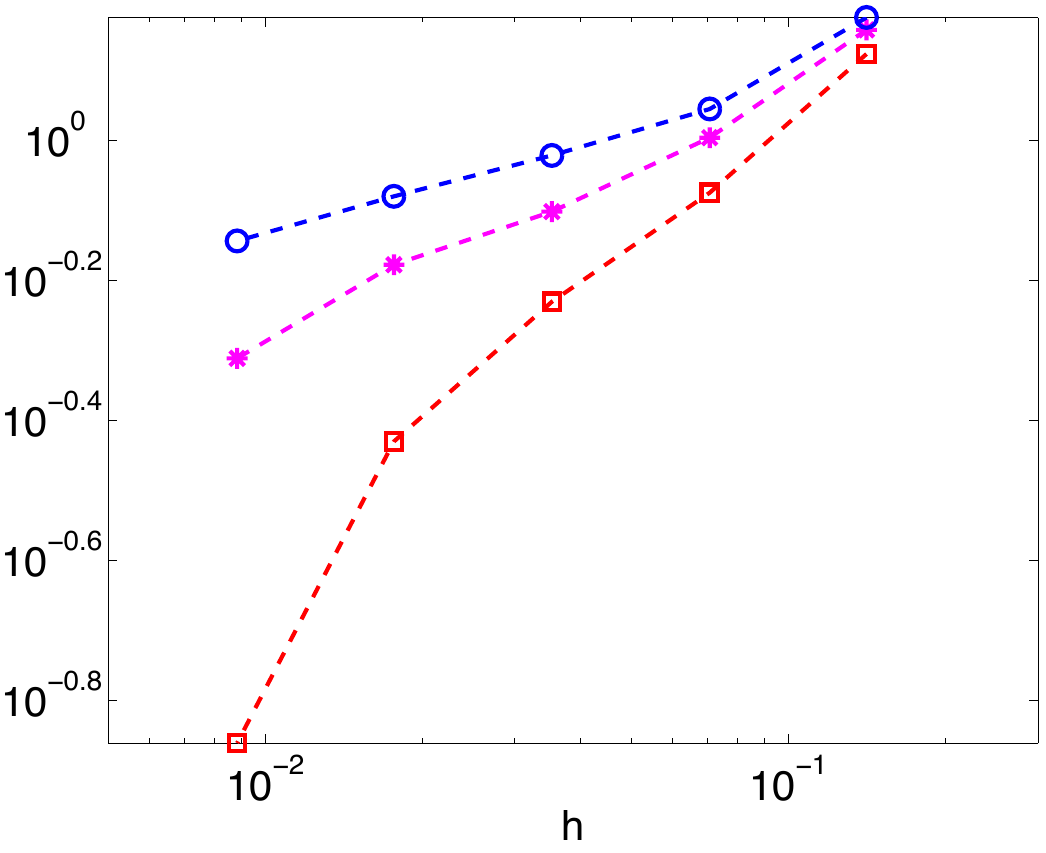}
\end{minipage}
\begin{minipage}{0.45\textwidth}
\hspace*{0.5cm}
\includegraphics[width=\textwidth]{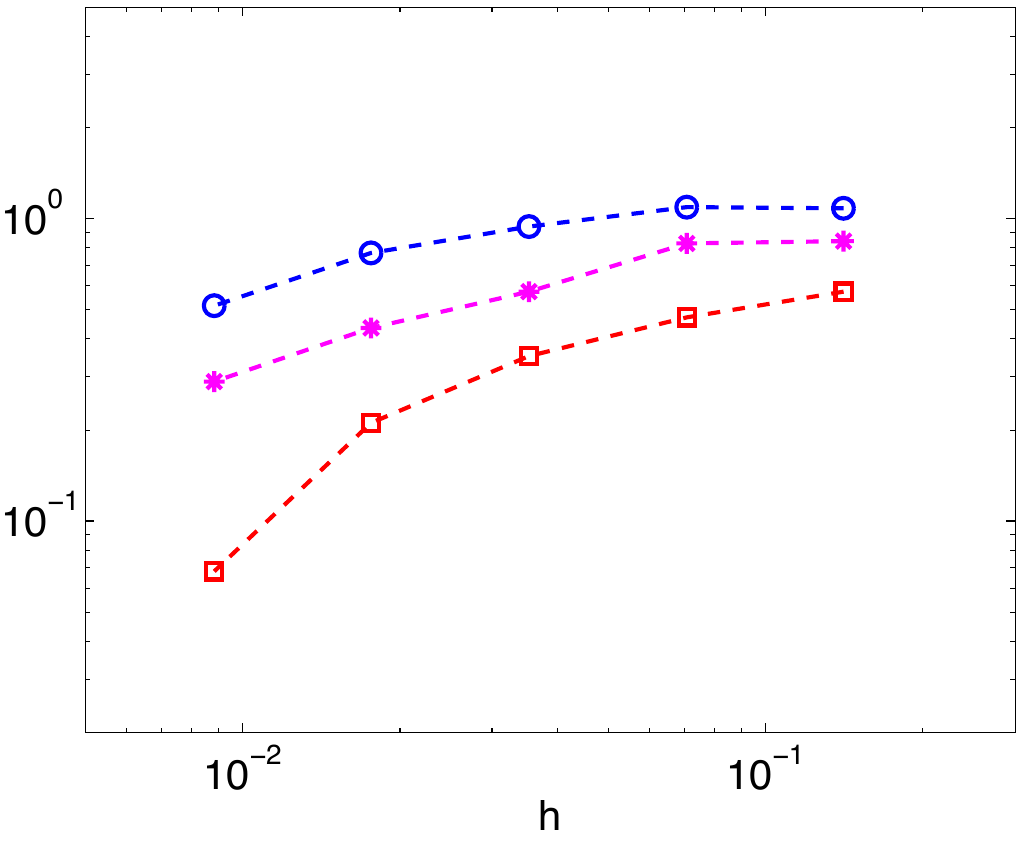}
\end{minipage}
\caption[Convergence for the control and the state, $\eps=10^{-8}$]{$\omega=(0.2,0.5)$; $y_0(x)=\sin(\pi x)$: $\eps=10^{-8}$. ;   $\frac{\|\rho_0 (v_{\eps} - v_{\eps,h})\|_{L^2(q_T)}}{\Vert \rho_0 v_{\eps} \Vert_{L^2(q_T)}}$ ({\bf Left}) and $\frac{\|y_{\eps} - \lambda_{\eps,h}\|_{L^2(Q_T)}}{\Vert y_{\eps} \Vert_{L^2(Q_T)}}$ ({\bf Right}) vs. $h$ for $r=10^{2}$ ({\color{blue} $\circ$}), $r=1.$ ({\color{magenta}$\star$}) and $r=10^{-2}$ ({\color{red}$\square$}).}\label{contsolepsem8}
\end{center}
\end{figure} 

Remarkably, we highlight that the variational approach developed here allows, for any $\eps$, a direct and robust approximation of one control for the heat equation. As discussed at length in \cite{EFC-AM-dual, MunchZuazuaRemedies}, the minimization of the conjugate functional $J^{\star}_{\eps}$ using conjugate gradient algorithm requires a great number of iterates for small $\eps$ (typically $\eps=10^{-8}$ and $\omega=(0.2,0.5)$) and diverge for small values of $h$.  

	Eventually, we present one experiment for the mixed formulation \eqref{eq:mf2} introduced in 
	Section~\ref{second_mixed_form}, which require only the use of continuous finite element approximation.
	Precisely, we use here $\mathbb{P}_1$ finite elements for both the states $\varphi_\eps$ and $\lambda_\eps$.
	With the same data as before, Table~\ref{tab:second_mixed} reports some norms with respect to $h$ for $r=1$ and
	$\eps=10^{-4}$. We still observe the convergence when $h$ tends to zero, but as expected (comparing with 
	Table~\ref{tab:ex_r1_e04_0205} corresponding to the mixed formulation \eqref{eq:mf1}) with lower convergence rates:
	for instance, we obtain that $\|y_\eps-\lambda_{\eps,h}\|_{L^2(Q_T)}=\mathcal{O}(h^{0.91})$ while, 
	from Table~\ref{tab:ex_r1_e04_0205}, we obtain  $\|y_\eps-\lambda_{\eps,h}\|_{L^2(Q_T)}=\mathcal{O}(h^{1.17})$.
	We also refer to Section $4.1$ of \cite{EFC-AM-sema} where a different mixed formulation is discussed in this context.

\begin{table}[http]
\centering
\scalebox{0.9}{
\begin{tabular}{|c|ccccc|}
\hline 
$h$ & $1.41\times 10^{-1}$ & $7.07\times 10^{-2}$ & $3.53\times 10^{-2}$ & $1.76\times 10^{-2}$ & $8.83\times 10^{-3}$ \tabularnewline
\hline

$m_h+n_h$ & $132$ & $462$ & $1\,732$ & $6\,642$ & $26\,082$\tabularnewline

$\|L^{\star}\ph_{\eps,h}\|_{L^2(0,T;H^{-1}(0,1))}$ & $7.90\times 10^{-1}$ & $4.42\times 10^{-1}$ & $2.47\times 10^{-1}$ & $1.37\times 10^{-1}$ & $7.72\times 10^{-2}$\tabularnewline

$\frac{\|\rho_0 (v_{\eps} - v_{\eps,h})\|_{L^2(q_T)}}{\Vert \rho_0 v_{\eps} \Vert_{L^2(q_T)}}$ &  $6.55\times 10^{-1}$ &  $3.50\times 10^{-1}$ &  $1.86\times 10^{-1}$ &   $9.87\times 10^{-2}$ &  $5.27\times 10^{-2}$ \tabularnewline

$\frac{\|y_{\eps} - \lambda_{\eps,h}\|_{L^2(Q_T)}}{\Vert y_{\eps} \Vert_{L^2(Q_T)}}$ & $3.86\times 10^{-1}$ & $2.06\times 10^{-1}$ & $1.09\times 10^{-1}$ & $5.82\times 10^{-2}$ & $3.11\times 10^{-2}$  \tabularnewline

\hline
\end{tabular}}
\caption[Mixed formulation \eqref{eq:mf2} -  $r=1$ and $\eps=10^{-4}$]{Mixed formulation (\ref{eq:mf2}) -  $r=1$ and $\eps=10^{-4}$ with  $\om=(0.2,0.5)$.}
\label{tab:second_mixed}
\end{table}


\subsection{Conjugate gradient for $J_{\eps,r}^{\star\star}$}
\label{dualdualsection}

	We illustrate here the Section \ref{duasectionmf1} and minimize the functional $J_{\eps,r}^{\star\star}:L^2(Q_T)\to \mathbb{R}$ defined in Proposition \ref{prop_equiv_dual} with respect to the variable $\lambda_{\eps}$. From the ellipticity of the operator $\mathcal{A}_{\eps,r}$, we use a conjugate gradient method  which in the present case reads as follows : 
	
\begin{enumerate}
	\item  {\it Initialization}\\
	Let $\lambda_{\eps}^0 \in L^2(Q_T)$ be a given function;\\
	Solve 
	\[
	\left\{
	\begin{array}{l}
		\overline\ph_{\eps}^0 \in \Phi_\eps\\
		a_{\eps,r}(\phbar_{\eps}^0, \phbar) + b_\eps(\phbar,\lambda_\eps^0) = l_\eps(\phbar), \quad \forall \phbar \in \Phi_\eps
	\end{array}
	\right.
	\]
	
	and set $g_\eps^0 = L^\star\phbar_\eps^0$ and set $w_\eps^0=g_\eps^0$.
	
	For $n \geq 0$, assuming that $\lambda_\eps^n, g_\eps^n$ and $w_\eps^n$ are known with $g_\eps^n\neq0$ and $w_\eps^n\neq0$, compute
	$\lambda_\eps^{n+1}, g_\eps^{n+1}$ and $w_\eps^{n+1}$ as follows
	
	\item {\it Steepest descent}\\
	Compute $\phbar^n \in \Phi_\eps$ solution to
	\[
		a_{\eps,r}(\phbar_\eps^n, \phbar) =-b_\eps(\phbar, w_\eps^n), \quad \forall \phbar \in \Phi_\eps
	\]
	and $\overline w_\eps^n = L^\star\phbar_\eps^n$ and then compute 
	\[
	\dis {\smallskip} \rho_n = \|g_\eps^n\|^2_{L^2(Q_T)} /(\overline w_\eps^n, w_\eps^n)_{L^2(Q_T)}.
	\]
	and set
	\[
		\lambda_\eps^{n+1} = \lambda_\eps^n - \rho_n w_\eps^n.
	\]
	
	\item {\it Testing the convergence and construction of the new descent direction}\\
	Update $g_\eps^n$ by
	\[
		g_\eps^{n+1} = g_\eps^n - \rho^n \overline w_\eps^n.
	\]
	
	If $\|g_\eps^{n+1}\|_{L^2(Q_T)} / \|g_\eps^0\|_{L^2(Q_T)} \leq \gamma$, take $\lambda_\eps = \lambda_\eps^{n+1}$. Else, compute 
	\[
		\gamma_n=\|g_\eps^{n+1}\|^2_{L^2(Q_T)} / \|g_\eps^n\|^2_{L^2(Q_T)}
	\]
	and update $w_\eps^n$ via 
	\[
	w_\eps^{n+1}=g_\eps^{n+1}+\gamma_n w_\eps^n. 
	\]
	Do $n=n+1$ and return to step (ii).
\end{enumerate}
	As mentioned in \cite{GlowinskiBook1} where this approach is discussed at length for Stokes and Navier-Stokes systems, 
	this algorithm can be viewed as a sophisticated version of Uzawa type algorithm to solve the mixed formulation (\ref{eq:mf1}). 
	Concerning the speed of convergence of this algorithm, it follows, for instance, from \cite{Daniel1971} that 
\[
	\Vert \lambda_\eps^n-\lambda_\eps\Vert_{L^2(Q_T)} \leq 2\sqrt{\nu(\mathcal{A}_{\eps,r})}   \biggl(\frac{\sqrt{\nu(\mathcal{A}_{\eps,r})}-1}{\sqrt{\nu(\mathcal{A}_{\eps,r})}+1}
	\biggr)^n \Vert \lambda_\eps^0-\lambda_\eps\Vert_{L^2(Q_T)}, \quad \forall n\geq 1
\]
	where $\lambda_\eps$ minimizes $J^{\star\star}_{\eps,r}$. $\nu(\mathcal{A}_{\eps,r})=\Vert \mathcal{A}_{\eps,r}\Vert \Vert \mathcal{A}_{\eps,r}^{-1}\Vert$ denotes the 
	condition number of the operator $\mathcal{A}_{\eps,r}$.

	Eventually, once the above algorithm has converged we can compute $\ph_\eps \in \Phi_\eps$ as solution of
\[
	a_{\eps,r}(\ph_{\eps}, \phbar) + b_\eps(\phbar, \lambda_{\eps}) = l_\eps(\phbar), \quad \forall \phbar \in \Phi_\eps.
\]

We use the same spaces $\Phi_{\eps,h}$ and $M_{\eps,h}$ as described in Section \ref{discretization_mf1}. In practice, each iteration amounts to solve a linear system
involving the matrix $A_{\eps,r,h}$ of size $n_h=4m_h$ (see (\ref{matrixmfh})) which is sparse, symmetric and positive definite. We use the Cholesky method.

From the previous estimate, the performances of the algorithm are related to the condition number of the operator $\mathcal{A}_{\eps,r}$ restricted to $M_{\eps,h}\subset L^2(Q_T)$, which coincides here  (see \cite{BrezziFortin}) with the condition number of the symmetric and positive definite matrix $B_h A^{-1}_{\eps,r,h}B_h^T$ introduced in (\ref{eigenvalue}). Using again the power iteration algorithm, we obtain that, for any $h$, the largest eigenvalue of $B_h A^{-1}_{\eps,r,h}B_h^T$  is very closed to $r^{-1}$ (and bounded by $r^{-1}$). This is in agreement with the estimate $\Vert \mathcal{A}_{\eps,r} \lambda\Vert_{L^2(Q_T)} \leq r^{-1} \Vert \lambda\Vert_{L^2(Q_T)}$ for all $\lambda\in L^2(Q_T)$. Consequently, the condition number is expressed in term of $r$ and of the discrete inf-sup constant $\delta_{\eps,h}$ as follows : 
$$
\nu(B_h A^{-1}_{\eps,r,h}B_h^T) \approx r^{-1}\delta_{\eps,r,h}^{-2}.
$$
Since, from our observation in Section \ref{infsuptest}, the discrete inf-sup constant $\delta_{\eps,r,h}$ is uniformly bounded by below with respect to $h$, we deduce that the condition number is uniformly bounded by above with respect to the discretization parameter. This implies that the convergence of the sequence $\{\lambda^n_{\eps,h}\}_{(n>0)}$, minimizing for $J^{\star\star}_{\eps,r}$ over $M_{\eps,h}$ is independent of $h$. This is exactly what we observe from our numerical experiments. Morever, from (\ref{behavior_deltah}), we get that the number 
$\nu(B_h A^{-1}_{\eps,r,h}B_h^T) \approx C_{\eps,r,h}^{-2}$ is very closed to one. We refer to Tables \ref{tab:condmf1_r001_0205} and \ref{tab:condmf1_r1_r100_0205} for the values. 

\begin{table}[http]
	\centering
		\begin{tabular}{|c|cccc|}
			\hline
			$h$  			       & $7.07\times 10^{-2}$ & $3.53\times 10^{-2}$ & $1.76\times 10^{-2}$ & $8.83\times 10^{-3}$
			\tabularnewline
			\hline
			$\eps=10^{-2}$ 	      & $1.431$ 			& $1.426$ 			& $1.423$ 			& $1.423$ 	  
			\tabularnewline
			$\eps=10^{-4}$	 	& $1.185$		 	& $1.177$ 			& $1.173$ 		        & $1.171$  		         
			\tabularnewline
			$\eps=10^{-8}$ 	        & $1.165$			& $1.151$	 		& $1.142$ 			& $1.135$ 
			\tabularnewline  
			\hline
		\end{tabular}
	\caption[$r^{-1}\delta_{\eps,r,h}^{-2}$  w.r.t. $\eps$ and $h$ ; $r=10^{-2}$ ]{$r^{-1}\delta_{\eps,r,h}^{-2}$  w.r.t. $\eps$ and $h$ ; $r=10^{-2}$ ; $\Omega=(0,1)$, $\omega=(0.2,0.5)$, $T=1/2$.}
	\label{tab:condmf1_r001_0205}
\end{table}

\begin{table}[http]
	\centering
		\begin{tabular}{|c|cccc|}
			\hline
			$h$  			       & $7.07\times 10^{-2}$ & $3.53\times 10^{-2}$ & $1.76\times 10^{-2}$ & $8.83\times 10^{-3}$
			\tabularnewline
			\hline
			$\eps=10^{-2}$ 	        	& $1.013$ 			& $1.012$ 			& $1.012$ 	 &      $1.011$ 	  
			\tabularnewline
			$\eps=10^{-4}$	 	& $1.013$		 	& $1.012$ 			& $1.011$ 		        & $1.011$  		         
			\tabularnewline
			$\eps=10^{-8}$ 	        & $1.013$			& $1.012$	 		& $1.011$ 			& $1.011$ 
			\tabularnewline  
			\hline
		\end{tabular}
	\caption[$r^{-1}\delta_{\eps,r,h}^{-2}$  w.r.t. $\eps$ and $h$ ; $r=1$ and $r=10^2$]{$r^{-1}\delta_{\eps,r,h}^{-2}$  w.r.t. $\eps$ and $h$ ; $r=1$ and $r=10^2$; $\Omega=(0,1)$, $\omega=(0.2,0.5)$, $T=1/2$.}
	\label{tab:condmf1_r1_r100_0205}
\end{table}

We consider the same data as in Section \ref{eq:mf1_numer}, that is, $\omega=(0.2,0.5)$, $y_0(x)=\sin(\pi x)$ and $T=1/2$. We take $\gamma=10^{-10}$ as a stopping threshold for 
the algorithm (that is the algorithm is stopped as soon as the norm of the residue $g^n$ at the iterate $n$ satisfies 
$\Vert g_{\eps}^n\Vert_{L^2(Q_T)}\leq 10^{-10} \Vert g_{\eps}^0\Vert_{L^2(Q_T)}$). The algorithm is initiated with $\lambda_{\eps,h}^0=0$ in $Q_T$. 
	
We check that the method provides, for the same value of $r$, $\eps$ and $h$, exactly the same approximation $\lambda_{\eps,h}$ than the previous direct method (see Tables \ref{tab:ex_r1_e02_0205}, etc).
Table \ref{tab:cg_r1_0205}, \ref{tab:cg_r100_0205} and \ref{tab:cg_r001_0205}, we simply give the number of iterates of the conjugate gradient algorithm for $r=10^2$, $r=1$ and $r=10^{-2}$ with respect to $h$ and $\eps$ respectively. For each case, the convergence is reached in very few iterates, independent of $h$. Once again, this is in contrast with the 
behavior of the conjugate gradient algorithm when this latter is used to minimize $J^{\star}_{\eps}$ with respect to $\ph_T$ defined by (\ref{eq:min_eps_adj}).
The number of convergence is also almost independent of $\eps$ and $r$. Since the gradient of $J^{\star\star}_{\eps,r}$ is given by 
$\nabla J^{\star\star}_{\eps,r}(\lambda')=\mathcal{A}_{\eps,r}(\lambda')-L^{\star}\ph_0$ for all $\lambda'\in L^2(Q_T)$, in particular $\nabla J^{\star\star}_{\eps,r}(\lambda_\eps)=L^{\star}\ph_\eps$,
a larger value of the augmentation parameter $r$ reduces (slightly here) the number of iterates.

According to this very low number of iterates, it seems more advantageous not only in term of memory resource but also in term of time execution to solve the extremal problem in the variable $\lambda_{\eps}$
than the (equivalent) mixed formulation (\ref{eq:mf1h}). The matrix $A_{\eps,r,h}$ of order $n_h$ is very sparse, symmetric, positive definite, diagonal bloc (for which the Cholesky method is very efficient) while the matrix defined by (\ref{matrixmfh}), of order $m_h+ n_h = 5/4 n_h$ requires the use of for instance the Gauss decomposition method. Note however that the condition number of the matrix $A_{\eps,r,h}$ is not independent of $h$ but behaves polynomially (see Table \ref{tab:cg_r1_0205} where the value is reported for $r=1$.). On the other hand, the condition number slightly decreases with $r$ (recall that the norm over $\Phi_{\eps}$ contains the term $r\Vert L^{\star}\ph\Vert_{L^2(Q_T)}$): consequently, for very stiff situation (typically $\omega$ very small), there may be  a balance between large values of $r$ leading to a better numerical robustness and low values of $r$ leading to smaller relative errors on $v_{\eps,h}$ and $\lambda_{\eps,h}$.

For very small values of both $h$ (leading to very fine meshes) of the order $h=10^{-3}$ and $\eps$, we observe some instabilities on the approximation $\lambda_{\eps,h}$ (very likely due to the condition number of the matrix $A_{\eps,r,h}$ which exceeds $10^{25}$ in this case). A preconditioning technique introduced in the next section is needed in these cases.  

\begin{table}[http]
\centering
\scalebox{0.94}{
\begin{tabular}{|c|ccccc|}
\hline 
$h$ & $1.41\times 10^{-1}$ & $7.07\times 10^{-2}$ & $3.53\times 10^{-2}$ & $1.76\times 10^{-2}$ & $8.83\times 10^{-3}$ \tabularnewline
\hline

$m_h=card(\{\lambda_{\eps,h}\})$ & $66$ & $231$ & $861$ & $3\ 321$ & $13\ 041$\tabularnewline

$\sharp$ iterates - $\eps=10^{-2}$ & $5$ & $5$ & $5$ & $5$ & $5$\tabularnewline

$\sharp$ iterates - $\eps=10^{-4}$ & $5$ & $5$ & $5$ & $4$ & $4$\tabularnewline

$\sharp$ iterates - $\eps=10^{-8}$ & $5$ & $5$ & $5$ & $5$ & $5$\tabularnewline

$\kappa(A_{\eps,r,h})$ - $\eps=10^{-2}$  & $1.51\times 10^{9}$ & $1.10\times 10^{11}$ & $6.81\times 10^{12}$ & $3.83\times 10^{14}$ & $1.91\times 10^{16}$  \tabularnewline

\hline
\end{tabular}}
\caption[Mixed formulation \eqref{eq:mf1} -  $r=1$; Conjugate gradient algorithm]{Mixed formulation (\ref{eq:mf1}) -  $r=1$ - $\om=(0.2,0.5)$ ; Conjugate gradient algorithm.}
\label{tab:cg_r1_0205}
\end{table}

\begin{table}[http]
\centering
\scalebox{0.94}{
\begin{tabular}{|c|ccccc|}
\hline 
$h$ & $1.41\times 10^{-1}$ & $7.07\times 10^{-2}$ & $3.53\times 10^{-2}$ & $1.76\times 10^{-2}$ & $8.83\times 10^{-3}$ \tabularnewline
\hline

$\sharp$ iterates - $\eps=10^{-2}$ & $5$ & $5$ & $4$ & $4$ & $4$\tabularnewline

$\sharp$ iterates - $\eps=10^{-4}$ & $5$ & $5$ & $5$ & $4$ & $4$\tabularnewline

$\sharp$ iterates - $\eps=10^{-8}$ & $5$ & $5$ & $5$ & $5$ & $4$\tabularnewline

\hline
\end{tabular}}
\caption[Mixed formulation \eqref{eq:mf1} -  $r=10^2$; Conjugate gradient algorithm]{Mixed formulation (\ref{eq:mf1}) -  $r=10^2$ - $\om=(0.2,0.5)$ ; Conjugate gradient algorithm.}
\label{tab:cg_r100_0205}
\end{table}

\begin{table}[http]
\centering
\scalebox{0.94}{
\begin{tabular}{|c|ccccc|}
\hline 
$h$ & $1.41\times 10^{-1}$ & $7.07\times 10^{-2}$ & $3.53\times 10^{-2}$ & $1.76\times 10^{-2}$ & $8.83\times 10^{-3}$ \tabularnewline
\hline

$\sharp$ iterates - $\eps=10^{-2}$ & $9$ & $9$ & $8$ & $8$ & $8$\tabularnewline

$\sharp$ iterates - $\eps=10^{-4}$ & $8$ & $8$ & $8$ & $8$ & $8$\tabularnewline

$\sharp$ iterates - $\eps=10^{-8}$ & $8$ & $8$ & $7$ & $7$ & $7$\tabularnewline

\hline
\end{tabular}}
\caption[Mixed formulation \eqref{eq:mf1} -  $r=10^{-2}$; Conjugate gradient algorithm]{Mixed formulation (\ref{eq:mf1}) -  $r=10^{-2}$ - $\om=(0.2,0.5)$ ; Conjugate gradient algorithm.}
\label{tab:cg_r001_0205}
\end{table}

We do not describe experiments for the mixed formulation introduced in \ref{second_mixed_form}, which require the use of continuous finite element approximation. We refer to \cite{EFC-AM-sema} in a closed context.


\subsection{Numerical experiments for the mixed formulation (\ref{eq:mf3}) - limit case $\eps=0$.} 

We now report in this section some experiments corresponding to the limit case, that is $\eps=0$, of the mixed formulation (\ref{eq:mf3}).
We consider again the first mode : $y_0(x)=\sin(\pi x)$, take $\omega=(0.2,0.5)$, $T=1/2$ and the exponential type weights $\rho_0$ and $\rho$ given by (\ref{weight_rho0}) and (\ref{weight_rho}) respectively.

This particular choice of the weights allows to rewrite the quantity $\rho^{-1}\, L^{\star}\ph$ in term of the new variable $\psi$ as follow
\begin{equation}
\label{change_of_variable}
\begin{aligned}
\rho^{-1} \, L^{\star} (\rho_0 \psi) & = \rho^{-1}\rho_0  L^{\star}\psi - \rho^{-1}\rho_{0t} \psi  \\
& = (T-t)^{3/2} L^{\star}\psi +   \biggl(-\frac{3}{2}(T-t)^{1/2}+K_1 (T-t)^{-1/2}\biggr) \psi
\end{aligned}
\end{equation}
and thus eliminate the exponential singularity near $T^{-1}$. Only a much weaker polynomial singularity, precisely $(T-t)^{-1/2}$ remains.  

Moreover, we define as ``exact" solution $(y,v)$ the solution obtained with a very fine mesh corresponding to $h\approx 1.1\times 10^{-3}$, a number of element equal to $819\ 200$ and a number of degrees of freedom equal to $m_h+n_h=3\ 284\ 484$. With these values, we get the following norms :
\begin{equation}
\Vert \rho^{-1}\lambda_{h=1.1\times 10^{-3}} \Vert_{L^2(Q_T)}\approx 3.592\times 10^{-1}, \quad \Vert \rho_0 v_{h=1.1\times 10^{-3}} \Vert_{L^2(q_T)}\approx 18.6634.\nonumber
\end{equation}    
We do not use the Fourier expansion approach described in the Appendix, since the optimality equation (\ref{eq_apaq}) is ill-posed for $\eps=0$ and leads to instability as the number of modes used in the sum increases. On the contrary, the minimization of $J^{\star\star}_r$-  equivalent to the resolution of the mixed formulation (\ref{eq:mf3}) exhibits a remarkable robustness as $h\to 0$. Eventually, we mention that the mesh used is so fine that the corresponding result is (almost) independent of the parameter $r$.

Tables \ref{tab:ex_r001_rho_rho0_0205}, \ref{tab:ex_r1_rho_rho0_0205} and \ref{tab:ex_r100_rho_rho0_0205} reports some norms with respect to $h$ for $r=10^{-2}$, $r=1$ and $r=10^2$, respectively. Let us first mention that we again obtain exactly the same approximations from the direct resolution of the system (\ref{matrixmfh3bis}) and from the minimization of $J^{\star\star}_r$. 

As in the case $\eps>0$, we observe the convergence of $\rho^{-1}\lambda_h$ and $\rho_0 v_h$ in $L^2(Q_T)$ and $L^2(q_T)$ respectively as $h\to 0^+$. For instance, for $r=1$, 
we obtain 
\begin{equation}
\frac{\|\rho_0 (v - v_h)\|_{L^2(q_T)}}{\Vert \rho_0 v \Vert_{L^2(q_T)}} \approx e^{1.04}h^{0.429}, \quad \frac{\|y_ - \rho^{-1}\lambda_h\|_{L^2(Q_T)}}{\Vert y \Vert_{L^2(Q_T)}}\approx e^{1.27}h^{0.704}. \nonumber
\end{equation}
Figure \ref{contsoleps0} depicts the evolution of these relatives errors with respect to $h$ for $r=10^{-2}, 1$ and $r=10^2$. Again, in view of the values of the inf-sup constant of Table \ref{tab:deltah_mf3}, we check that the lower value $r=10^{-2}$ provides a faster convergence of the approximation.  It is also interesting to remark that low errors for the state $\rho^{-1}\lambda_h$ and the control $v_h$ are obtained with a relatively large value of the norm $\Vert \rho^{-1}L^{\star}\ph_h\Vert_{L^2(Q_T)}$. This suggests that the constraint equality $L^{\star}\ph=0$ in $L^2(Q_T)$ may be replaced by a weaker one as discussed in Section \ref{second_mixed_form}. We do not present experiments for the weaker formulation (\ref{eq:mf2}) and refer to Section 4 of \cite{EFC-AM-sema} in a closed context. 
Tables  \ref{tab:ex_r001_rho_rho0_0205}, \ref{tab:ex_r1_rho_rho0_0205} and \ref{tab:ex_r100_rho_rho0_0205} also report some results from the minimization of the functional $J_r^{\star\star}$ using the conjugate gradient algorithm. For $r=1$ and $r=10^2$, the quantity $r^{-1}\delta_{r,h}^{-2}$ - bounded by above of the condition number of $\hat{B}_h \hat{A}_{r,h}^{-1}\hat{B}_h^T$ - slightly decreases with $h$; the convergence of the algorithm is reached in few iterations independent of $h$. The value $r=10^{-2}$ requires about 50 iterations for all the discretization considered. 

Remarkably, the change of variable performed in the limit case allows to reduce very significantly the condition number $\kappa(A_{r,h})$ of the matrix $A_{r,h}$ (almost independent of $r$): see Table \ref{tab:ex_r001_rho_rho0_0205}. This allows to consider very small values of the parameter $h$ without producing any instabilities. 

This high robustness of the approximation is definitively in contrast classical dual methods discussed in \cite{MunchZuazuaRemedies} and the references therein: We recall that for $\eps=0$, the minimization of $J^{\star}_{\eps=0}$ defined by (\ref{eq:min_eps_adj}) fails as soon as $h$ is small enough. 

Figure \ref{fig:stateh_0205} and Figure \ref{fig:conth_0205} depict over $Q_T$ the approximation $y_h:=\rho^{-1}\lambda_h$ and $v_h:=\rho_0^{-1}\psi_h\, 1_{q_T}$ for $h=8.83\times 10^{-3}$. In particular, the smallness of both the diffusion coefficient and the size of the support $\omega$ leads to a large amplitude of the control at the initial time. This is in contrast 
with the boundary control situation where one acts directly on the state (or its first derivative).

Eventually, in order to validate one more time our computations, we have approximated by a standard time-marching algorithm the solution of (\ref{eq:heat}) with $v=v_h$. Specifically, we have used a $C^1$-approximation with $\mathbb{P}_{3,x}(0,1)$ elements in space and the second-step implicit Gear scheme (of order two) for the time discretization. Tables  \ref{tab:ex_r001_rho_rho0_0205}, \ref{tab:ex_r1_rho_rho0_0205} and \ref{tab:ex_r100_rho_rho0_0205} report the $L^2$-norm of the state at the final time, i.e. $\Vert y_h(\cdot,T)\Vert_{L^2(0,1)}$. For each value of $r$, the $L^2$-norm decreases linearly to $0$ with $h$. For any $h$, the non-zero value of $\Vert y_h(\cdot,T)\Vert_{L^2(0,1)}$ is, first due to the fact that $v_h$ is not an exact null-control for any discrete system, and second to the consistency error of the approximation used.  

\begin{table}[http]
\centering
\scalebox{0.9}{
\begin{tabular}{|c|ccccc|}
\hline 
$h$  & $3.53\times 10^{-2}$ & $1.76\times 10^{-2}$ & $8.83\times 10^{-3}$ & $4.41\times 10^{-3}$ & $2.2\times 10^{-3}$ \tabularnewline
\hline

$\|\rho^{-1}L^{\star}(\rho_0 \psi_{h})\|_{L^2(Q_T)}$ &   $29.76$ &   $24.86$ &  $21.12$  & $17.92$ & $15.42$ \tabularnewline

$\frac{\|\rho_0 (v - v_h)\|_{L^2(q_T)}}{\Vert \rho_0 v \Vert_{L^2(q_T)}}$  &  $5.35\times 10^{-1}$ &   $3.34\times 10^{-1}$ &  $2.42\times 10^{-1}$  & $1.63\times 10^{-1}$ & $8.45\times 10^{-2}$\tabularnewline

$\|\rho_0 v_h\|_{L^2(q_T)}$   &  $15.20$ &   $16.642$ &  $17.52$ & $18.07$ & $18.43$\tabularnewline

$\|\rho^{-1} \lambda_h\|_{L^2(Q_T)}$    &  $3.15\times 10^{-1}$ &   $3.34\times 10^{-1}$ &  $3.46\times 10^{-1}$  & $3.52\times 10^{-1}$ & $3.56\times 10^{-1}$ \tabularnewline

$\frac{\|y_ - \rho^{-1}\lambda_h\|_{L^2(Q_T)}}{\Vert y \Vert_{L^2(Q_T)}}$   & $1.96\times 10^{-1}$ & $1.20\times 10^{-1}$ & $6.97\times 10^{-2}$  & $3.67\times 10^{-2}$ & $1.49\times 10^{-2}$\tabularnewline

$\sharp$ CG iterates  & $52$ & $55$ & $56$ & $56$  & $55$\tabularnewline

$r^{-1}\delta_{r,h}^{-2}$ & $27.04$ & $29.37$ & $31.73$ & $33.37$  & $-$ \tabularnewline

$\kappa(A_{r,h})$   & $9.5\times 10^{4}$ & $1.4\times 10^{7}$ & $3.03\times 10^{9}$ & $1.1\times 10^{12}$ & $-$  \tabularnewline

$n_h$=size($A_{r,h}$)  & $3\ 444$ & $13\ 284$ & $52\ 264$ & $206\ 724$ & $823\ 044$  \tabularnewline

$\Vert y_h(\cdot,T)\Vert_{L^2(0,1)}$   & $1.52\times 10^{-1}$ & $6.109\times 10^{-2}$ & $2.59\times 10^{-2}$ & $1.162\times 10^{-2}$ & $5.41\times 10^{-3}$ \tabularnewline

\hline
\end{tabular}}
\caption[Mixed formulation \eqref{eq:mf3} --  $r=10^{-2}$ and $\eps=0$]{Mixed formulation (\ref{eq:mf3}) -  $r=10^{-2}$ and $\eps=0$ with  $\om=(0.2,0.5)$.}
\label{tab:ex_r001_rho_rho0_0205}
\end{table}

\begin{table}[http]
\centering
\scalebox{0.9}{
\begin{tabular}{|c|ccccc|}
\hline 
$h$  & $3.53\times 10^{-2}$ & $1.76\times 10^{-2}$ & $8.83\times 10^{-3}$ & $4.41\times 10^{-3}$ & $2.2\times 10^{-3}$ \tabularnewline
\hline

$\|\rho^{-1}L^{\star}(\rho_0 \psi_{h})\|_{L^2(Q_T)}$   &  $3.659$ &   $3.276$ &  $2.808$  & $2.377$  & $2.002$\tabularnewline

$\frac{\|\rho_0 (v - v_h)\|_{L^2(q_T)}}{\Vert \rho_0 v \Vert_{L^2(q_T)}}$   &  $6.97\times 10^{-1}$ &   $4.82\times 10^{-1}$ &  $3.69\times 10^{-1}$  & $2.81\times 10^{-1}$ & $2.06\times 10^{-1}$ \tabularnewline

$\|\rho_0 v_h\|_{L^2(q_T)}$   &  $13.37$ &   $15.33$ &  $16.62$ & $17.45$ & $17.99$ \tabularnewline

$\|\rho^{-1} \lambda_h\|_{L^2(Q_T)}$   &  $3.35\times 10^{-1}$ &   $3.40\times 10^{-1}$ &  $3.41\times 10^{-1}$  & $3.42\times 10^{-1}$ & $3.52\times 10^{-1}$\tabularnewline

$\frac{\|y_ - \rho^{-1}\lambda_h\|_{L^2(Q_T)}}{\Vert y \Vert_{L^2(Q_T)}}$  & $3.28\times 10^{-1}$ & $2.13\times 10^{-1}$ & $1.33\times 10^{-1}$  & $8.09\times 10^{-2}$ & $4.63\times 10^{-2}$\tabularnewline

$\sharp$ CG iterates   & $12$ & $11$ & $10$ & $9$ & $9$\tabularnewline

$r^{-1}\delta_{r,h}^{-2}$  & $2.092$ & $2.062$ & $1.585$ & $1.333$  & $-$ \tabularnewline

$\Vert y_h(\cdot,T)\Vert_{L^2(0,1)}$    & $1.19\times 10^{-1}$ & $5.39\times 10^{-2}$ & $2.42\times 10^{-2}$ & $1.12\times 10^{-2}$ & $5.29\times 10^{-3}$ \tabularnewline

\hline
\end{tabular}}
\caption[Mixed formulation \eqref{eq:mf3} --  $r=1$ and $\eps=0$]{Mixed formulation (\ref{eq:mf3}) -  $r=1$ and $\eps=0$ with  $\om=(0.2,0.5)$.}
\label{tab:ex_r1_rho_rho0_0205}
\end{table}

\begin{table}[http]
\centering
\scalebox{0.9}{
\begin{tabular}{|c|ccccc|}
\hline 
$h$  & $3.53\times 10^{-2}$ & $1.76\times 10^{-2}$ & $8.83\times 10^{-3}$ & $4.41\times 10^{-3}$  & $2.2\times 10^{-3}$ \tabularnewline
\hline

$\|\rho^{-1}L^{\star}(\rho_0 \psi_{h})\|_{L^2(Q_T)}$ &   $0.428$ &   $0.426$ &  $0.380$  & $0.321$  & $0.215$\tabularnewline

$\frac{\|\rho_0 (v - v_h)\|_{L^2(q_T)}}{\Vert \rho_0 v \Vert_{L^2(q_T)}}$  &   $8.83\times 10^{-1}$ &   $6.80\times 10^{-1}$ &  $5.24\times 10^{-1}$  & $4.16\times 10^{-1}$  & $3.25\times 10^{-1}$\tabularnewline

$\|\rho_0 v_h\|_{L^2(q_T)}$   &  $9.880$ &   $12.706$ &  $14.82$ & $16.256$ & $17.338$\tabularnewline

$\|\rho^{-1} \lambda_h\|_{L^2(Q_T)}$   &  $0.2546$ &   $0.2926$ &  $0.3189$  & $0.3352$ & $0.3477$\tabularnewline

$\frac{\|y_ - \rho^{-1}\lambda_h\|_{L^2(Q_T)}}{\Vert y \Vert_{L^2(Q_T)}}$  &  $5.86\times 10^{-1}$ & $4.04\times 10^{-1}$ & $2.63\times 10^{-1}$  & $1.66\times 10^{-1}$ & $9.88\times 10^{-2}$ \tabularnewline

$\sharp$ CG iterates  & $10$ & $8$ & $7$ & $5$ &  $5$ \tabularnewline

$r^{-1}\delta_{r,h}^{-2}$  & $2.092$ & $2.007$ & $1.53$ & $1.103$  & $-$ \tabularnewline

$\Vert y_h(\cdot,T)\Vert_{L^2(0,1)}$    & $8.26\times 10^{-2}$ & $4.24\times 10^{-2}$ & $2.11\times 10^{-2}$ & $1.03\times 10^{-2}$ & $5.12\times 10^{-3}$ \tabularnewline

\hline
\end{tabular}}
\caption[Mixed formulation \eqref{eq:mf3} --  $r=10^2$ and $\eps=0$]{Mixed formulation (\ref{eq:mf3}) -  $r=10^{2}$ and $\eps=0$ with  $\om=(0.2,0.5)$.}
\label{tab:ex_r100_rho_rho0_0205}
\end{table}

\begin{figure}[http]
\begin{center}
\begin{minipage}{0.45\textwidth}
 \includegraphics[width=\textwidth]{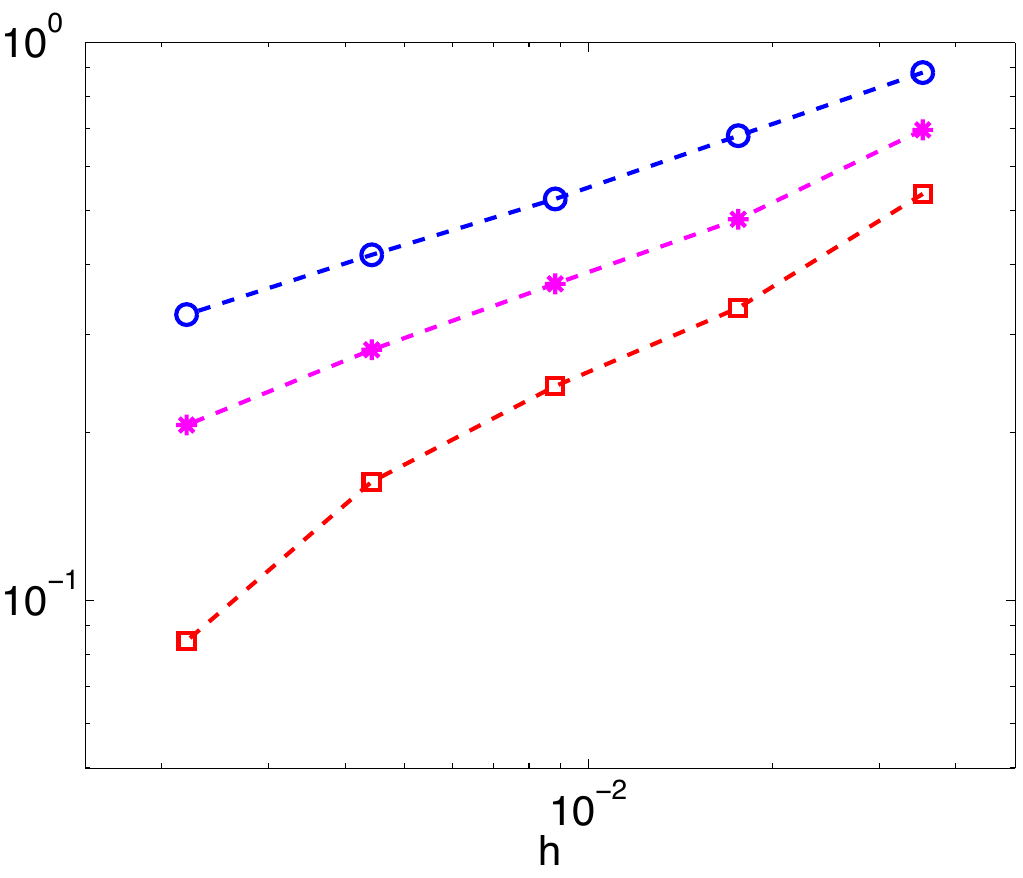}
\end{minipage}
\begin{minipage}{0.45\textwidth}
\hspace*{0.5cm}
\includegraphics[width=\textwidth]{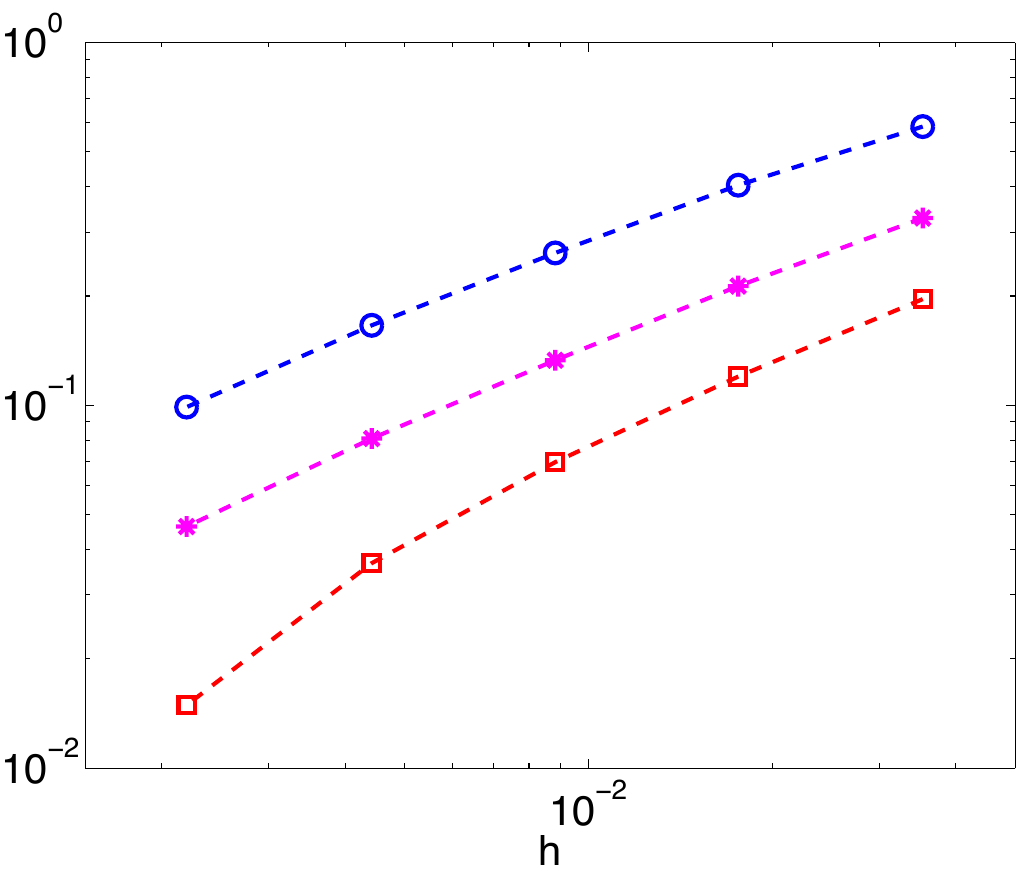}
\end{minipage}
\caption[Convergence for the control and the state, $\eps=0$]{$\omega=(0.2,0.5)$; $y_0(x)=\sin(\pi x)$: $\eps=0$. ;   $\frac{\|\rho_0 (v - v_h)\|_{L^2(q_T)}}{\Vert \rho_0 v \Vert_{L^2(q_T)}}$ ({\bf Left}) and $\frac{\|y - \rho^{-1}\lambda_h\|_{L^2(Q_T)}}{\Vert y \Vert_{L^2(Q_T)}}$ ({\bf Right}) vs. $h$ for $r=10^{2}$ ({\color{blue} $\circ$}), $r=1.$ ({\color{magenta}$\star$}) and $r=10^{-2}$ ({\color{red}$\square$}).}\label{contsoleps0}
\end{center}
\end{figure} 

\begin{figure}[http]
\begin{center}
\includegraphics[scale=0.6]{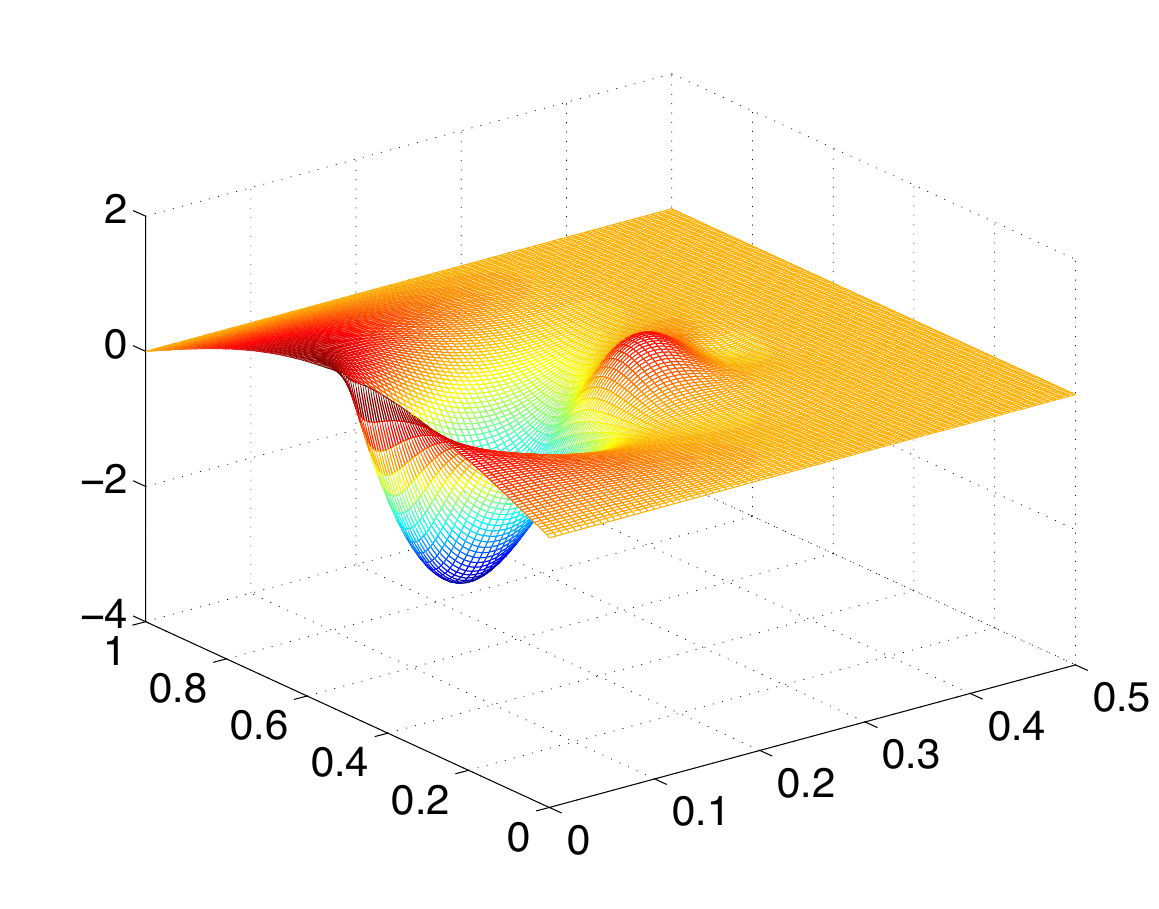}
\caption[Approximation of the controlled state]{$\omega=(0.2,0.5)$; Approximation $\rho^{-1} \lambda_h$ of the controlled state  $y$  over $Q_T$   - $r=1$ and $h=8.83\times 10^{-3}$.}\label{fig:stateh_0205}
\end{center}
\end{figure}

\begin{figure}[http]
\begin{center}
\includegraphics[scale=0.6]{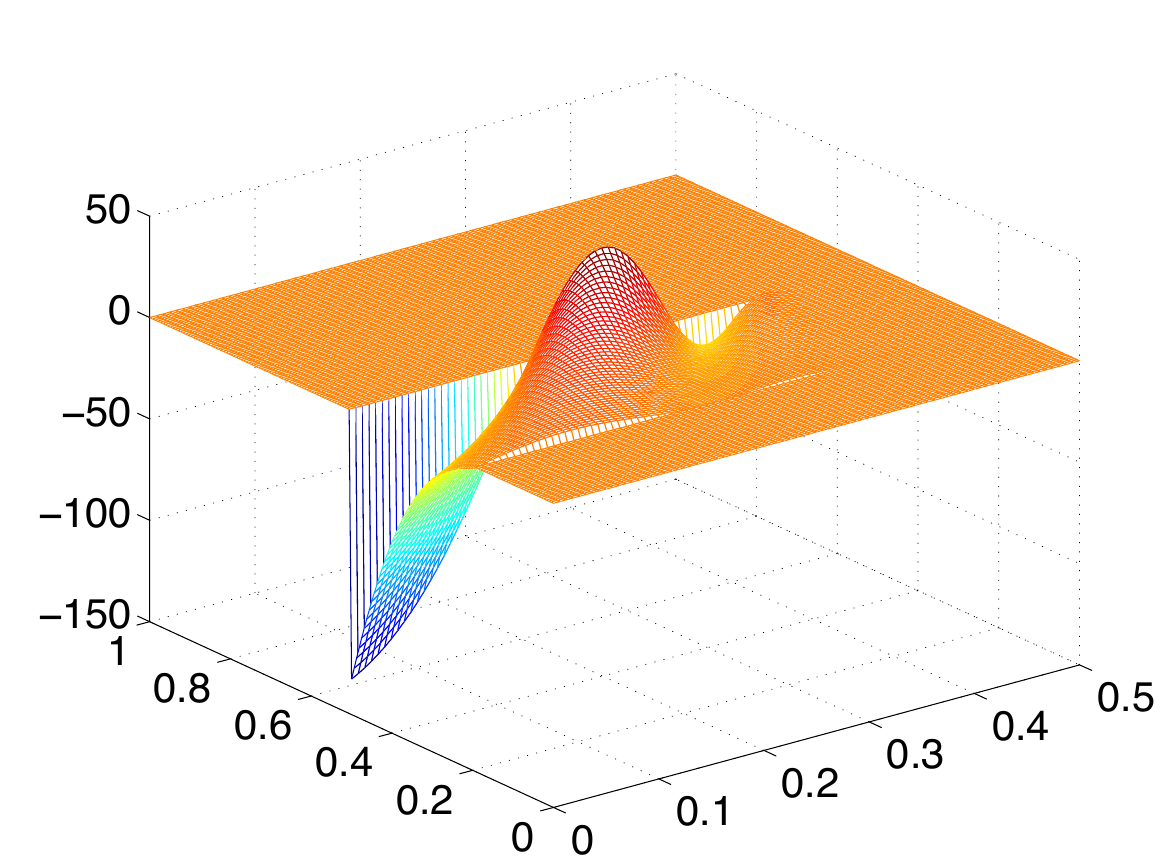}
\caption[Approximation of the null control]{$\omega=(0.2,0.5)$; Approximation $v_h=\rho_0^{-1}\psi_h$ of the null control  $v$  over $Q_T$   - $r=1$ and $h=8.83\times 10^{-3}$.}\label{fig:conth_0205}
\end{center}
\end{figure}

The experiments reported here - in the limit case $\eps=0$ - are obtained for a specific choice of the weights $\rho_0$ and $\rho$. Precisely, the weight $\rho_0$ is such that the approximation $v_h:=\rho_0^{-2}\ph_h \,1_{\omega}$ vanishes exponentially as $t\to T^{-}$. This allows in particular to avoid the high oscillatory behavior of the control of minimal $L^2$-norm, that is when $\rho_0:=1$ in $q_T$. The exponential behavior of the control implies a similar behavior of the corresponding controlled stated   $\rho^{-1}\lambda$, so that the choice of the parameter $\rho$ made here, is also natural. Remark that $\rho$ is not bounded and therefore does not strictly satisfied the hypothesis of Theorem \ref{th:mf3}. Seemingly, this has no influence at the numerical level.   
This specific choice of the parameter $\rho$ allows to perform a change of variable and therefore reduce significantly the condition number of the discrete problem. We also point out that, if the mixed formulation (\ref{eq:mf3}) is well-posed for any $\rho,\rho_0$ satisfying the hypothesis of Theorem \ref{th:mf3}, the constant of continuity of the linear form $\tilde{l}$ depends strongly - in view of the Carleman estimate (\ref{crucial_estimate}) -  of $\rho$ and $\rho_0$. This affects the convergence and the robustness of the method. Thus, for $\rho_0$ as before and $\rho:=1$, the 
condition number is too large for small values of $h$ (typically $h\approx 10^{-3}$) and leads to wrong results. Remark that for $\rho:=1$, the exponential decreases of $\rho_0^{-1}$ cannot be compensated by $\rho$ (see  (\ref{crucial_estimate})) so that the change of variable is inefficient.

\section{Concluding remarks and Perspectives}
\label{concluding_remarks}

The mixed formulation we have introduced here in order to address the null controllability of the heat equation seems original and adapted the work \cite{NC-AM-mixedwave} devoted to the wave equation. This formulation is nothing else than the Euler system associated to  the conjugate functional and depends on both the dual adjoint variable and a Lagrange multiplier, which turns out to be the primal state of the heat equation to be controlled. The approach, recently used in a different way in \cite{EFC-AM-sema}, leads to a variational problem defined over time-space functional Hilbert spaces, without distinction between the time and the space variable. The main ingredients allowing to prove the well-posedness of the mixed formulation are an observability inequality and a direct inequality (usually deduced from energy estimates). For these reasons, the mixed reformulation may also be employed to any other controllable systems for which such inequalities are available. In particular, we may consider the Stokes system as in \cite{munch_MCSS2014}. 

At the practical level, the discrete mixed time-space formulation is solved in a systematic way in the framework of the finite element theory: in contrast to the classical approach initially developed in \cite{GL96}, there is no need to take care of the time discretization nor of the stability of the resulting scheme, which is often a delicate issue.  The resolution amounts to solve a sparse symmetric linear system : the corresponding matrix can be preconditioned if necessary, and may be computed once for all as it does not depend on the initial data to be controlled. Eventually, as discussed in \cite{NC-AM-mixedwave}, Section 4.3 (but not employed here), the space-time discretization of the domain allows an adaptation of the mesh so as to reduce the computational cost and capture the main features of the solutions.  We also emphasize that the higher dimensional case is very similar as it requires $C^1$ approximation in space.

The numerical experiments reported in this work suggest a very good behavior of the approach: the strong convergence of the sequences $\{v_h\}_{h>0}$, approximation of the controls of minimal weighted square integrable norm, are clearly observed as the discretization parameter $h$ tends to zero (as the consequence of the uniform inf-sup discrete property). It is worth to mention that, within this mixed formulation approach, the strong convergence of the approximations  (as obtained within a closed but different approach in \cite{EFC-AM-sema} assuming that the weights $\rho_0$ and $\rho$ coincide with the Carleman weight) is still to be done. From the uniform coercivity of the bilinear form in the primal variable, a strong convergence is guaranteed by a uniform discrete inf-sup property. In view of the complicated and unusual constraint $L^\star \ph=0$ and of the $C^1$ nature of the approximation, the proof of such uniform property is probably very hard to get. However, it seems possible to bypass this property by adding to the Lagrangian the stabilization terms (for instance in the limit case $\eps=0$)
\[
-\Vert L(\rho^{-1} \lambda_h) - \rho_0^{-2} \ph_h \,1_{\omega} \Vert^2_{L^2(Q_T)}, \quad -\Vert \lambda_h(\cdot,0)-y_0\Vert^2_{L^2(0,1)} 
\]
which vanish at the continuous level (writing $L y=v\,1_{\omega}$ with $y=\rho^{-1}\lambda$ and $v=\rho^{-2}\ph 1_{\omega}$, see Theorem \ref{th:mf3}) and give coercivity property for the variable $\lambda_h$. This will be examined in a future work.

The approach may also be extended to the boundary case. We also mention that the variational approach developed here based on a space-time formulation is also very well-adapted to the case where the support of the inner control evolves in time and takes the form
\[
q_T:= \{(x,t)\in Q_T; \quad a(t)<  x < b(t) \quad t\in (0,T)\}
\]
with any $a,b$ in  $C^0([0,T], ]0,1[)$. We refer to \cite{CC-NC-AM} which examines this case for the wave equation. 

Eventually, we also mention that this approach which consists in solving directly the optimality conditions of a controllability problem may be employed to solve inverse problems 
where, for instance, the solution of the heat equation has to be recovered from a partial observation, typically localized on a sub-domain $q_T$ of the working domain: actually, the optimality conditions associated to a least-square type functional can be expressed as a mixed formulation very closed to (\ref{th:mf3}). 
This issue will be analyzed in a future work.

\section{Appendix}

\subsection{Appendix : Fourier expansion of the control of minimal $L^2(\rho_0,q_T)$ norm.}

We expand in term of Fourier series the control of minimal $L^2(\rho_0, q_T)$ norm $v$ for the (\ref{eq:heat}) and the corresponding controlled solution $y$. We use these expansions in Section \ref{eq:mf1_numer} to evaluate with respect to $h$ the error $\Vert y_{\eps}-\lambda_{\eps,h}\Vert_{L^2(Q_T)}$ and $\Vert \rho_0 (v_{\eps}-v_{\eps,h})\Vert_{L^2(q_T)}$ where the sequence $(\ph_{\eps,h},\lambda_{\eps,h})$ solves the discrete mixed formulation (\ref{eq:mf1h}). We use the characterization of the couple $(y_{\eps},v_{\eps})$ in term of the adjoint solution $\ph_{\eps}$ (see \eqref{eq:wr}), unique minimizer in $L^2(\Omega)$ of $J^{\star}_{\eps}$ defined by (\ref{eq:min_eps_adj}).

We first note $(a_{\eps,p})_{(p>0)}$ the Fourier coefficients in $l^2(\mathbb{N})$ of the minimizer $\ph_{T,\eps}\in L^2(0,1)$ of (\ref{eq:min_eps_adj}) such that 
\begin{equation}
\ph_{\eps,T}(x)=\sum_{p>0} a_{\eps,p}\, \sin(p\pi x), \quad x\in (0,1).   \label{ph_fourier}
\end{equation}
The adjoint state takes the form $\ph_{\eps}(x,t)=\sum\limits_{p\geq1}a_{\eps,p} e^{c\pi^2p^2(t-T)}\sin (p\pi x)$  in $Q_T$. 

The optimality equation associated to the functional 
$J_{\eps}^{\star}$ then reads, 
\begin{equation}
DJ^{\star}_{\eps}(\ph_{\eps,T})\cdot \overline{\ph_T} = \jjntqT \rho_0^{-2}\ph_{\eps}\overline{\ph} \,dx\,dt + \eps \int_0^1 \ph_{\eps,T}\overline{\ph_T} + (y_0, \overline{\ph}(\cdot,0))=0, \forall \overline{\ph_T} \in L^2(0,1)  \nonumber
\end{equation}
and can be rewritten in terms of the $(a_{\eps,p})_{p>0}$ as follows : 
\begin{equation} 
<
\{\overline{a_p}\}_{p>0},
\mathcal{M}_{q_T,\eps} \{a_p\}_{p>0}   > =
< \{\overline{a_p}\}_{p>0} , \mathcal{F}_{y_0}> \quad \forall  \overline{a}_{\eps,p}\in l^2(\mathbb{N})\label{eq_apaq}
\end{equation}
where $\mathcal{M}_{q_T,\eps}$ denotes a symmetric positive definite matrix and $\mathcal{F}_{y_0}$ a vector obtained from the expansion (\ref{ph_fourier}). The resolution of the infinite dimensional system (reduced to a finite dimension one by truncation of the sums) allows an approximation of the minimizer $\ph_{T,\eps}$ of $J^{\star}_{\eps}$.

Finally, we use that the control of minimal $L^2(\rho_0,q_T)$
norm is given by $v_{\eps}= \rho_0^{-2}\ph_{\eps}\,1_{\omega}$ and find that the corresponding controlled solution may be expanded as follows  
\begin{equation}
y_{\eps}(x,t)=\sum_{q>0} \biggl(e^{-c\pi^2 q^2t}b^0_q +    \sum_{p\geq1}a_{\eps,p}  c_{q,p}(\omega) d_{q,p}(t)]    \biggr) \sin(p\pi x), \quad (x,t)\in Q_T \label{y_fourier}
\end{equation}
 with 
$$	
c_{p,q}(\omega)\!:=\!2\!\int_{\omega} \!\sin(p\pi x)\sin(q\pi x)\,dx; \,~  d_{p,q}(t):=\int_0^t \!\!\rho_0^{-2}(s)e^{c\pi^2(p^2(s-T)+q^2(s-t))} \,ds,\,~ t\in (0,T). \nonumber
$$ 
$(b^0_q)_{q>0}$ denotes the Fourier coefficients of the initial data $y_0\in L^2(0,1)$.

\subsection{Appendix: Tables}

\begin{table}[http]
\centering
\scalebox{0.95}{
\begin{tabular}{|c|ccccc|}
\hline 
$h$ & $1.41\times 10^{-1}$ & $7.07\times 10^{-2}$ & $3.53\times 10^{-2}$ & $1.76\times 10^{-2}$ & $8.83\times 10^{-3}$ \tabularnewline
\hline

$\|L^{\star}\ph_{\eps,h}\|_{L^2(Q_T)}$ & $3.84\times 10^{-2}$ & $2.90\times 10^{-2}$ & $9.27\times 10^{-3}$ & $2.41\times 10^{-3}$ & $7.78\times 10^{-4}$\tabularnewline

$\frac{\|\rho_0 (v_{\eps} - v_{\eps,h})\|_{L^2(q_T)}}{\Vert \rho_0 v_{\eps} \Vert_{L^2(q_T)}}$ &  $1.32\times 10^{-1}$ &  $5.90\times 10^{-2}$ &  $3.24\times 10^{-2}$ &   $1.68\times 10^{-2}$ &  $8.57\times 10^{-3}$ \tabularnewline

$\frac{\|y_{\eps} - \lambda_{\eps,h}\|_{L^2(Q_T)}}{\Vert y_{\eps} \Vert_{L^2(Q_T)}}$ & $1.04\times 10^{-1}$ & $3.54\times 10^{-2}$ & $1.48\times 10^{-2}$ & $7.59\times 10^{-3}$ & $3.89\times 10^{-3}$  \tabularnewline

$\|\lambda_{\eps,h}(\cdot,T)\|_{L^2(0,1)}$ & $2.02\times 10^{-1}$ & $1.68\times 10^{-1}$ & $1.65\times 10^{-1}$ & $1.67\times 10^{-1}$ & $1.68\times 10^{-1}$  \tabularnewline

$\kappa_{\eps}$ & $4.44\times 10^{9}$ & $4.20\times 10^{11}$ & $3.84\times 10^{13}$ & $3.25\times 10^{15}$ & $5.72\times 10^{16}$  \tabularnewline

\hline
\end{tabular}}
\caption[Mixed formulation \eqref{eq:mf1} --  $r=10^2$ and $\eps=10^{-2}$]{Mixed formulation (\ref{eq:mf1}) -  $r=10^2$ and $\eps=10^{-2}$ with  $\om=(0.2,0.5)$.}
\label{tab:ex_r100_e02_0205}
\end{table}

\begin{table}[http]
\centering
\scalebox{0.95}{
\begin{tabular}{|c|ccccc|}
\hline 
$h$ & $1.41\times 10^{-1}$ & $7.07\times 10^{-2}$ & $3.53\times 10^{-2}$ & $1.76\times 10^{-2}$ & $8.83\times 10^{-3}$ \tabularnewline
\hline

$\|L^{\star}\ph_{\eps,h}\|_{L^2(Q_T)}$ & $6.19\times 10^{-2}$ & $1.57\times 10^{-1}$ & $1.56\times 10^{-1}$ & $1.50\times 10^{-1}$ & $6.21\times 10^{-2}$\tabularnewline


$\frac{\|\rho_0 (v_{\eps} - v_{\eps,h})\|_{L^2(q_T)}}{\Vert \rho_0 v_{\eps} \Vert_{L^2(q_T)}}$ &  $1.02$ &  $7.36\times 10^{-1}$ &  $3.65\times 10^{-1}$ &   $1.52\times 10^{-1}$ &  $3.01\times 10^{-2}$ \tabularnewline


$\frac{\|y_{\eps} - \lambda_{\eps,h}\|_{L^2(Q_T)}}{\Vert y_{\eps} \Vert_{L^2(Q_T)}}$ & $6.74\times 10^{-1}$ & $5.51\times 10^{-1}$ & $2.42\times 10^{-1}$ & $1.05\times 10^{-1}$ & $1.81\times 10^{-2}$  \tabularnewline


$\|\lambda_{\eps,h}(\cdot,T)\|_{L^2(0,1)}$ & $2.23\times 10^{-1}$ & $1.76\times 10^{-1}$ & $7.86\times 10^{-2}$ & $4.87\times 10^{-2}$ & $3.28\times 10^{-2}$  \tabularnewline

$\kappa_{\eps}$ & $5.31\times 10^{9}$ & $8.31\times 10^{11}$ & $9.64\times 10^{13}$ & $1.47\times 10^{16}$ & $1.50\times 10^{18}$  \tabularnewline

\hline
\end{tabular}}
\caption[Mixed formulation \eqref{eq:mf1} --  $r=10^2$ and $\eps=10^{-3}$]{Mixed formulation (\ref{eq:mf1}) -  $r=10^2$ and $\eps=10^{-4}$ with  $\om=(0.2,0.5)$.}
\label{tab:ex_r100_e04_0205}
\end{table}

\begin{table}[t]
\centering
\scalebox{0.95}{
\begin{tabular}{|c|ccccc|}
\hline 
$h$ & $1.41\times 10^{-1}$ & $7.07\times 10^{-2}$ & $3.53\times 10^{-2}$ & $1.76\times 10^{-2}$ & $8.83\times 10^{-3}$ \tabularnewline
\hline

$\|L^{\star}\ph_{\eps,h}\|_{L^2(Q_T)}$ & $6.23\times 10^{-2}$ & $1.63\times 10^{-1}$ & $1.77\times 10^{-1}$ & $2.66\times 10^{-1}$ & $2.24\times 10^{-1}$\tabularnewline


$\frac{\|\rho_0 (v_{\eps} - v_{\eps,h})\|_{L^2(q_T)}}{\Vert \rho_0 v_{\eps} \Vert_{L^2(q_T)}}$ &  $1.50$ &  $1.11$ &  $9.53\times 10^{-1}$ &   $8.33\times 10^{-1}$ &  $7.19\times 10^{-1}$ \tabularnewline


$\frac{\|y_{\eps} - \lambda_{\eps,h}\|_{L^2(Q_T)}}{\Vert y_{\eps} \Vert_{L^2(Q_T)}}$ & $1.08 $ & $1.09$ & $9.4\times 10^{-1}$ & $7.69\times 10^{-1}$ & $5.15\times 10^{-1}$  \tabularnewline


$\|\lambda_{\eps,h}(\cdot,T)\|_{L^2(0,1)}$ & $2.24\times 10^{-1}$ & $1.79\times 10^{-1}$ & $8.10\times 10^{-2}$ & $5.67\times 10^{-2}$ & $1.71\times 10^{-2}$  \tabularnewline

$\kappa_{\eps}$ & $5.32\times 10^{9}$ & $8.59\times 10^{11}$ & $9.86\times 10^{13}$ & $1.84\times 10^{16}$ & $3.07\times 10^{18}$  \tabularnewline

\hline
\end{tabular}}
\caption[Mixed formulation \eqref{eq:mf1} --  $r=10^2$ and $\eps=10^{-8}$]{Mixed formulation (\ref{eq:mf1}) -  $r=10^2$ and $\eps=10^{-8}$ with  $\om=(0.2,0.5)$.}
\label{tab:ex_r100_e08_0205}
\end{table}

\begin{table}[t]
\centering
\scalebox{0.95}{
\begin{tabular}{|c|ccccc|}
\hline 
$h$ & $1.41\times 10^{-1}$ & $7.07\times 10^{-2}$ & $3.53\times 10^{-2}$ & $1.76\times 10^{-2}$ & $8.83\times 10^{-3}$ \tabularnewline
\hline

$\|L^{\star}\ph_{\eps,h}\|_{L^2(Q_T)}$ & $2.86\times 10^{-1}$ & $7.15\times 10^{-2}$ & $1.84\times 10^{-2}$ & $4.86\times 10^{-3}$ & $1.40\times 10^{-3}$\tabularnewline


$\frac{\|\rho_0 (v_{\eps} - v_{\eps,h})\|_{L^2(q_T)}}{\Vert \rho_0 v_{\eps} \Vert_{L^2(q_T)}}$ &  $1.11\times 10^{-1}$ &  $6.21\times 10^{-2}$ &  $3.29\times 10^{-2}$ &   $1.68\times 10^{-2}$ &  $8.57\times 10^{-3}$ \tabularnewline


$\frac{\|y_{\eps} - \lambda_{\eps,h}\|_{L^2(Q_T)}}{\Vert y_{\eps} \Vert_{L^2(Q_T)}}$ & $5.16\times 10^{-2}$ & $2.84\times 10^{-2}$ & $1.48\times 10^{-2}$ & $7.59\times 10^{-3}$ & $3.89\times 10^{-3}$  \tabularnewline


$\|\lambda_{\eps,h}(\cdot,T)\|_{L^2(0,1)}$ & $1.53\times 10^{-1}$ & $1.61\times 10^{-1}$ & $1.65\times 10^{-1}$ & $1.67\times 10^{-1}$ & $1.68\times 10^{-1}$  \tabularnewline

$\kappa_{\eps}$ & $9.15\times 10^{8}$ & $2.07\times 10^{10}$ & $8.05\times 10^{11}$ & $3.25\times 10^{13}$ & $1.45\times 10^{15}$  \tabularnewline

\hline
\end{tabular}}
\caption[Mixed formulation \eqref{eq:mf1} --  $r=10^{-2}$ and $\eps=10^{-2}$]{Mixed formulation (\ref{eq:mf1}) -  $r=10^{-2}$ and $\eps=10^{-2}$ with  $\om=(0.2,0.5)$.}
\label{tab:ex_r001_e02_0205}
\end{table}

\begin{table}[t]
\centering
\scalebox{0.95}{
\begin{tabular}{|c|ccccc|}
\hline 
$h$ & $1.41\times 10^{-1}$ & $7.07\times 10^{-2}$ & $3.53\times 10^{-2}$ & $1.76\times 10^{-2}$ & $8.83\times 10^{-3}$ \tabularnewline
\hline

$\|L^{\star}\ph_{\eps,h}\|_{L^2(Q_T)}$ & $10.77$ & $3.821$ & $1.018$ & $2.59\times 10^{-1}$ & $6.56\times 10^{-2}$\tabularnewline


$\frac{\|\rho_0 (v_{\eps} - v_{\eps,h})\|_{L^2(q_T)}}{\Vert \rho_0 v_{\eps} \Vert_{L^2(q_T)}}$ &  $4.63\times 10^{-1}$ &  $2.23\times 10^{-1}$ &  $1.10\times 10^{-1}$ &   $5.52\times 10^{-2}$ &  $2.74\times 10^{-2}$ \tabularnewline


$\frac{\|y_{\eps} - \lambda_{\eps,h}\|_{L^2(Q_T)}}{\Vert y_{\eps} \Vert_{L^2(Q_T)}}$ & $1.55\times 10^{-1}$ & $9.03\times 10^{-2}$ & $4.08\times 10^{-2}$ & $2.46\times 10^{-2}$ & $1.27\times 10^{-2}$  \tabularnewline


$\|\lambda_{\eps,h}(\cdot,T)\|_{L^2(0,1)}$ & $3.22\times 10^{-2}$ & $2.85\times 10^{-2}$ & $2.99\times 10^{-2}$ & $3.08\times 10^{-2}$ & $3.12\times 10^{-2}$  \tabularnewline

$\kappa_{\eps}$ & $3.04\times 10^{9}$ & $1.33\times 10^{11}$ & $7.55\times 10^{12}$ & $3.88\times 10^{14}$ & $1.96\times 10^{16}$  \tabularnewline

\hline
\end{tabular}}
\caption[Mixed formulation \eqref{eq:mf1} --  $r=10^{-2}$ and $\eps=10^{-4}$]{Mixed formulation (\ref{eq:mf1}) -  $r=10^{-2}$ and $\eps=10^{-4}$ with  $\om=(0.2,0.5)$.}
\label{tab:ex_r001_e04_0205}
\end{table}

\begin{table}[t]
\centering
\scalebox{0.915}{
\begin{tabular}{|c|ccccc|}
\hline 
$h$ & $1.41\times 10^{-1}$ & $7.07\times 10^{-2}$ & $3.53\times 10^{-2}$ & $1.76\times 10^{-2}$ & $8.83\times 10^{-3}$ \tabularnewline
\hline

$\|L^{\star}\ph_{\eps,h}\|_{L^2(Q_T)}$ & $21.872$ & $19.388$ & $26.098$ & $28.310$ & $21.249$\tabularnewline

$\|\rho_0 (v_{\eps} - v_{\eps,h})\|_{L^2(q_T)}$ &  $14.989$ &  $9.459$ &  $6.606$ &  $4.175$ &  $1.556$ \tabularnewline

$\frac{\|\rho_0 (v_{\eps} - v_{\eps,h})\|_{L^2(q_T)}}{\Vert \rho_0 v_{\eps} \Vert_{L^2(q_T)}}$ &  $1.33$ &  $8.43\times 10^{-1}$ &  $5.89\times 10^{-1}$ &   $3.72\times 10^{-1}$ &  $1.38\times 10^{-1}$ \tabularnewline


$\frac{\|y_{\eps} - \lambda_{\eps,h}\|_{L^2(Q_T)}}{\Vert y_{\eps} \Vert_{L^2(Q_T)}}$ & $5.73\times 10^{-1}$ & $4.71\times 10^{-}$ & $3.51\times 10^{-1}$ & $2.11\times 10^{-1}$ & $6.82\times 10^{-2}$  \tabularnewline


$\|\lambda_{\eps,h}(\cdot,T)\|_{L^2(0,1)}$ & $3.31\times 10^{-2}$ & $1.31\times 10^{-2}$ & $5.99\times 10^{-3}$ & $2.83\times 10^{-3}$ & $8.26\times 10^{-4}$  \tabularnewline

$\kappa_{\eps}$ & $4.08\times 10^{9}$ & $3.04\times 10^{11}$ & $4.54\times 10^{13}$ & $6.79\times 10^{15}$ & $1.30\times 10^{18}$  \tabularnewline

\hline
\end{tabular}}
\caption[Mixed formulation \eqref{eq:mf1} --  $r=10^{-2}$ and $\eps=10^{-8}$]{Mixed formulation (\ref{eq:mf1}) -  $r=10^{-2}$ and $\eps=10^{-8}$ with  $\om=(0.2,0.5)$.}
\label{tab:ex_r001_e08_0205}
\end{table}

\

\noindent\textbf{Acknowledgments:} This work has been partially done while the second author was visiting the 
		Blaise Pascal University (Clermont-Ferrand, France). 
		He wishes to thank the members of the Laboratory of Mathematics for their kind hospitality.


\end{document}